\documentclass[twoside,11pt]{article}

\usepackage{jmlr2e_arxiv}

\DeclareSymbolFontAlphabet{\Bbb}{AMSb}

\usepackage{amsmath} 
\usepackage[arrow, matrix, curve]{xy}

\DeclareSymbolFont{wideparensymbol}{OMX}{yhex}{m}{n}
\DeclareMathAccent{\wideparen}{\mathord}{wideparensymbol}{"F3}

\newcommand{\set}[2]{\lbrace \,{#1}\,|\,{#2}\,\rbrace}
\newcommand{\bs}{\boldsymbol}

\newcommand{\R}{\mathbb{R}}

\newcommand{\E}{\mathbb{E}}

\newcommand{\dx}[1]{\hspace*{0.25ex}d\hspace*{-0.15ex}#1}
\newcommand{\dxy}[2]{\dx{#1}(#2)}

\renewcommand{\a}{\alpha}
\renewcommand{\b}{\beta}
\newcommand{\g}{\gamma}
\newcommand{\G}{\Gamma}
\renewcommand{\d}{\delta}
\newcommand{\D}{\Delta}
\newcommand{\e}{\varepsilon}

\newcommand{\z}{\zeta}
\newcommand{\n}{\eta}

\renewcommand{\k}{\kappa}
\newcommand{\lb}{\lambda}

\renewcommand{\r}{\rho}
\newcommand{\s}{\sigma}

\renewcommand{\t}{\tau}
\renewcommand{\th}{\theta}

\newcommand{\om}{\omega}

\DeclareMathOperator{\diam}{diam}
\DeclareMathOperator{\supp}{supp}

\DeclareMathOperator{\vol}{vol}

\newcommand{\cl}[1]{\wideparen{#1}}

\DeclareMathOperator{\sign}{sign}
\newcommand{\fcl}{\wideparen{f}}

\DeclareMathOperator{\id}{id}
\newcommand{\eins}{\boldsymbol{1}}

\newcommand{\Lx}[2]{{L_{#1}(#2)}}

\newcommand{\snorm}[1]{\Vert #1 \Vert}

\newcommand{\inorm}[1]{\Vert #1 \Vert_\infty}
\newcommand{\enorm}[1]{\Vert #1 \Vert_2}

\newcommand{\RP}[2]{{{\cal R}_{#1,P}(#2)}}
\newcommand{\RPB}[1]{{{\cal R}_{#1,P}^{*}}}

\newcommand{\Rx}[3]{{{\cal R}_{#1,#2}(#3)}}

\newcommand{\lclass}{{{L}_{\mathrm{class}}}}

\jmlrheading{?}{????}{?-??}{?/??}{??/??}{Ingrid Blaschzyk and Ingo Steinwart}

\ShortHeadings{Improved Classification Rates for Localized SVMs}{Blaschzyk and Steinwart}
\firstpageno{1}

\begin{document}

\title{Improved Classification Rates for Localized SVMs}

\author{\name Ingrid Blaschzyk \email ingrid.blaschzyk@mathematik.uni-stuttgart.de \\
	\name Ingo Steinwart \email ingo.steinwart@mathematik.uni-stuttgart.de \\
	\addr Institute for Stochastics and Applications\\ 
	University of Stuttgart\\
	70569 Stuttgart, Germany}

\editor{?}

\maketitle

\begin{abstract}
Localized support vector machines solve SVMs on many spatially defined small chunks  and one of their main characteristics besides the computational benefit compared to global SVMs is the freedom of choosing arbitrary kernel and regularization parameter on each cell. We take advantage of this observation to derive \textit{global} learning rates for localized SVMs with Gaussian kernels and hinge loss. 
Under certain assumptions our rates outperform known classification rates for localized SVMs, for global SVMs, and other learning algorithms based on e.g.,  plug-in rules, trees, or DNNs. These rates are achieved under a set of margin conditions that describe the behavior of the data-generating distribution, where no assumption on the existence of a density is made. We observe that a margin condition that relates the distance to the decision boundary to the amount of noise is crucial to obtain rates. The statistical analysis relies on a careful analysis of the excess risk which includes a separation of the input space into a subset that is close to the decision boundary and into a subset that is sufficiently far away. Moreover, we show that our rates are obtained adaptively, that is, without knowing the parameters resulting from the  margin conditions.
\end{abstract}


\section{Introduction}

Experimental results show that support vector machines (SVMs) handle  small- and medium-sized datasets in supervised learning tasks, see \citep{hundert},  \citep{EbSt16},  \citep{ThBlMeSt17} or \citep{KlGueUnMa17}. Recently, it was shown that they even outperform self-normalizing neural-networks (SNNs) for such datasets, see \citep{KlGueUnMa17}. However, many learning tasks, e.g., diagnostics of diseases on patient data, demand learning methods that handle large-scale datasets, where observations have high dimensions and/or the number of observations is large. At this point global SVMs and more generally kernel methods suffer from their computational complexity, which for SVMs is at least quadratically in space and time. To reduce this complexity \cite{EbSt16} proposed a data decomposition strategy, called localized SVMs, which solve SVMs on many \textit{spatially} defined chunks and which lead to improved time and space complexities.
In \citep{ThBlMeSt17} experimental results with  \texttt{liquidSVM} \citep{StTH17} showed that localized SVMs can tackle datasets with 32 million of training samples.  
Another approach to handle massive amount of data is the approach of random chunking, see e.g., \citep{BoVa92a}, \citep{zhang15}. Recently proposed algorithms use matrix or kernel approximations, see e.g., \citep{WiSe01a}, \citep{RaRe08a}, \citep{RuCaRo15}, \citep{ruro17} or they fall into the category of distributed learning, see e.g., \citep{li17} or \citep{MuBl18}.

For localized SVMs the underlying partition can base on clusters \citep{ChTaJi07a}, decision trees \citep{BeBl98a}, or k-nearest-neighbors \citep{ZhBeMaMa06a}, but the previous examples are rather experimentally investigated. In contrast, there also exist several theoretical results for localized SVMs. Based on possible overlapping regions or decomposition with k-nearest neighbor universal consistency and/or robustness for those classifiers are proved in \citep{DuCh18} and \citep{Ha13}. For Gaussian kernels and least-squares-loss  \cite{EbSt16} showed optimal learning rates under usual smoothness assumptions on the Bayes decision function, whereas \cite{ThBlMeSt17} obtained learning rates for classification under margin conditions.

In classification, margin conditions that describe the interplay between the marginal distribution $P_X$ and the conditional distribution of labels are commonly used to obtain learning rates for classifiers, see e.g., \citep{MaTs99},  \citep{KoKr07}, \citep{StCh08}, \citep{BlSt18}. The most popular exponent, the Tsybakov noise exponent, was introduced in \citep{MaTs99} and  measures the amount of noise in the input space, where  noise equals the probability of wrongly labelling some given input  $x \in X$. Under the assumption of Tsybakov noise exponent and some smoothness assumption on the regression function, fast rates for plug-in classifier are achieved in \citep{AuTs07}, \citep{KoKr07}, and \citep{BeHsMi18}, for tree-based classifiers in \citep{BiCoDaDe14}, for DNN-classifier in \citep{KiOhKi18}, or for a special case of SVMs in \citep{LiZeCh17}. Some of the mentioned authors additionally make assumptions on the density of the marginal distribution $P_X$ to improve their rates that were achieved without density assumptions, or to even find rates. However, it is well known  that boundedness assumptions on the density of $P_X$ together with smoothness and noise exponent assumptions limit the class of considered distributions, see e.g., \citep{AuTs07}, \citep{KoKr07} or \citep{BiCoDaDe14}. Hence, density assumptions are not preferable. Without such assumptions and without smoothness assumptions on the regression function, but with a margin condition that takes also the amount of mass around the decision boundary into consideration, rates for SVMs are achieved in \citep{St07}, \citep{StCh08}, \citep{LiZeCh17}, \citep{ThBlMeSt17}. Recently, \cite{BlSt18} showed under a mild regularity assumption on the decision boundary and under certain  margin conditions that rates for the histogram rule can be obtained, which even outperform known rates for SVMs under a certain set of assumption, which makes both methods comparable.

In this paper we investigate the statistical properties of classifiers derived by local SVM using Gaussian kernels and hinge loss.  We show that the achieved learning rates outperform the rates of several learning algorithms mentioned in the previous paragraph under suitable assumptions. In order to derive global finite sample bounds on the excess classification risk we apply the splitting technique developed in \citep{BlSt18}, that is we split the input space into two sets that depend on a splitting parameter $s>0$, one that is 
close to the decision boundary and one that is sufficiently far away from the decision boundary, and 
analyze the excess risk separately on these sets. To derive in a first step local finite sample bounds by a standard decomposition into a stochastic and an approximation error we make the observation that the approximation error has to be handled differently on cells that intersect the decision boundary and on those which do not. 
On cells with the latter property the assumption on a margin condition that relates the distance to the decision boundary to the amount of noise is crucial. Descriptively, it restricts the location of noise, that means, if we have noise for some $x \in X$, this $x$ has to be close to the decision boundary. From these local finite sample bounds we  derive rates by taking advantage of the great flexibility local SVMs enable us by definition, that is, that kernel and regularization parameter can be chosen on each cell individually. By choosing in a final step the splitting parameter $s$ appropriately, we then  derive global learning rates that depend on the margin parameters. Moreover, we show that training validation  support vector machines (TV-SVMs) achieve the same learning rates adaptively, that is, without knowing these parameters. Furthermore, we compare our rates with the rates achieved by methods mentioned above.  It turns out that we improve or match the rates of the compared methods and that these improvements result essentially from the above mentioned margin condition.

The paper is organized as follows. In Section~\ref{sec:pre} we briefly describe the localized SVM ansatz, introduce notation and close with theoretical assumptions. Section~\ref{sec:classrates} is divided up into two subsections: In Section~\ref{subsec:learnrates} we present our main result followed by a detailed description  that lead to this result. In Section~\ref{subsec:comp} we compare our rates carefully with other known rates. The proofs of our main results are contained in Section~\ref{sec:proof}. The results on individual sets, that is, bounds on the approximation error, oracle inequalities and learning rates, on predefined sets can be found  in Subsection~\ref{Bounds on Approximation Error} up to Subsection~\ref{sec:oracratesN2F}. Some results on margin conditions and some technical results can be found in the Appendix.





\section{Preliminaries}\label{sec:pre}

Given a dataset $D:=((x_1,y_1),\ldots,(x_n,y_n))$ of observations,  where $y_i \in Y:=\lbrace{-1,1\rbrace}$, the learning target in classification is to find a decision function $f_D \colon X \to  Y$  such that for  new data $(x,y)$ we have $f_D(x)=y$ with high probability. We assume that $x_i \in B_{\ell^d_2}$, where 
$B_{\ell^d_2}$ denotes the closed unit ball of the $d$-dimensional Euclidean space $\ell^d_2$ and assume that our data $D$ is generated independently and identically by a probability measure $P$ on $\R^d \times Y$. We denote by $P_X$ the marginal distribution on $\R^d$, write $X:= \supp(P_X)$, and assume  $X \subset B_{\ell^d_2}$ and $P_X(\partial X)=0$. 

We briefly describe the localized SVM approach in a generalized manner. Given a dataset $D$ local SVMs construct a function $f_D$ by solving SVMs on  spatially defined small chunks of $D$. To be more precise, let  $\mathcal{A}:=(A_{j})_{j=1,\ldots,m}$ be an arbitrary partition of $B_{\ell^d_2}$. We define for every  $j\in\{1,\ldots,m\}$ the index set
\begin{align*}
 I_{j}:=\{ i \in \lbrace 1, \ldots ,n \rbrace : x_i \in A_j\}
\end{align*}
with $\sum_{j=1}^m |I_j|=n$, that indicates the samples of $D$ contained in $A_j$ and we define  the corresponding 
local data set $D_j$ by
\begin{align*}
 D_{j}:=\left((x_i,y_i)\in D : i \in I_j\right).
\end{align*}
Then, one learns an individual SVM on \emph{each} cell by solving the optimization problem
\begin{align}\label{locSVM}
 f_{D_{j},\lb_j}=\underset{f \in H_j}{\arg\min}\, \lb_j \|f\|_{H_j}^2 + \frac{1}{n}\sum_{x_i,y_i \in D_j}L(x_i,y_i,f(x_i)) 
\end{align}
for every $j\in\{1,\ldots,m\}$, where $\lb_j>0$ is a regularization parameter, where  $H_j$ is a reproducing kernel Hilbert space (RKHS) over $A_j$ with arbitrary reproducing kernel $k_j: A_j \times A_j \to \R$, see \citep[Chap.~4]{StCh08}, and where $L: X \times Y \times \R \to [0,\infty)$ is a measurable function, called loss function, describing our learning goal. The final decision function  $f_{D,\bs\lb} : X \to \R$ is then defined by
\begin{align}\label{locSVMpredictor}
 f_{D,\bs\lb}(x) := \sum_{j=1}^{m} \eins_{A_{j}}(x)f_{D_{j},\lb_j}(x),
\end{align}
where $\bs\lb:=(\lb_1,\ldots,\lb_m)\in (0,\infty)^m$.
We make the following assumptions.
\begin{itemize}
 \item[\textbf{(H)}]
 For every $j \in \lbrace 1,\ldots, m \rbrace$ let $k_j: A_j \times A_j \to \R$ be the Gaussian kernel with width $\g_j>0$, defined by
 \begin{align}\label{Gausskern}
 k_{\g_j}(x,x'):=\exp\left(-\g_j^{-2}\|x-x'\|_2^2\right), 
\end{align}
 with corresponding RKHS $H_{\g_j}$ over $A_j$ and denote by $\hat{H}_j:=\lbrace \eins_{A_j} f : f \in H_j \rbrace$ the extended RKHS over $B_{\ell^d_2}$. For some $J \subset \lbrace 1,\ldots, m \rbrace$ define the joint RKHS $H_J$ over $B_{\ell^d_2}$ by $H_J:= \bigoplus_{j \in J} \hat{H}_j$, see \cite[Sec.~3]{EbSt16}.
%
\end{itemize}
We write $f_{D_{j},\lb_j,\g_j}$ for the local SVM predictor in \eqref{locSVM} to remember its local dependency on the kernel parameter $\g_j$ and the regularization parameter $\lb_j$ on each cell $A_j$ for $j \in \lbrace 1,\ldots, m\rbrace$. Clearly, we are free to choose different kernel and regularization parameters on each cell, since the predictors in \eqref{locSVM} are computed \textit{independently} on each cell.  Moreover, we write $f_{D,\bs\lb,\bs\g}$ for the final decision function in \eqref{locSVMpredictor}, where $\bs\g:=(\g_1,\ldots,\g_m)\in (0,\infty)^m$. 
Note that we have immediately that $f_{D,\bs\lb,\bs\g} \in H_J$ for $J = \lbrace 1,\ldots, m \rbrace$ since $\eins_{A_{j}} f_{D_{j},\lb_j,\g_j} \in \hat{H}_j$ for every $j \in J$. 
To measure the quality of the predictor locally,
we define a (local) loss $L_j: X \times Y \times \R \to [0,\infty)$ by
\begin{align*}
L_j(x,y,t):=\eins_{A_j}(x)L(x,y,t).
\end{align*}
Moreover, we define for an arbitrary index set $J \subset \lbrace 1,\ldots, m \rbrace$  the set $T:=\bigcup _{j\in J} A_j$ and the associated loss $L_{J_T} :X \times Y \times \R \to [0,\infty)$ by
\begin{align*}
L_{J_T}(x,y,t):=\eins_{T}(x)L(x,y,t),
\end{align*}
where we sometimes use the abbreviation $L_T:=L_{J_T}$ to avoid multiple subscripts.
A typical loss function  is the 
classification loss $L_{\text{class}}: Y\times\R\to[0,\infty)$, defined by
\begin{align*}
 L_{\text{class}}(y,t):=\mathbf{1}_{(-\infty,0]}(y\, \text{sign}t),
\end{align*}
where $\text{sign } 0:=1$. For (local) SVMs the optimization problem is not solvable for the classification loss. A suitable convex surrogate is for example the hinge loss  $L_{\text{hinge}}: Y\times\R\to[0,\infty)$, defined by
\begin{align*}
 L_{\text{hinge}}(y,t):=\max \lbrace 0,1-yt \rbrace\,
\end{align*}
for $y=\pm 1,\, t \in \R$. Note that for convex losses the existence and uniqueness of \eqref{locSVM} are secured, see e.g.\ \citep[Chap.~5.1]{StCh08}, \citep{EbSt16}. Since we are not interested in the loss of single labels, we consider the expected loss and define for a loss function $L$ the $L$-risk of a measurable function $f:X\to\R$ by
\begin{align*}
 \RP{L}f=\int_{X\times Y} L(x,y,f(x)) \,dP(x,y).
\end{align*}
Moreover, we define the optimal $L$-risk, called Bayes risk, with respect to $P$ and $L$, by
\begin{align*}
 \RPB{L}:=\inf\left\{\RP{L}f \ |
 \ f : X\to \R \text{ measurable}\right\} \,
\end{align*}
and call a function $f^*_{L,P} : X\to\R$ attaining the infimum, Bayes decision function. For the classification loss, a Bayes decision function is given by $f^{\ast}_{L_{\text{class}},P}(x):=\text{sign}(2P(y=1|x)-1), x \in X$.  
 A well-known result by Zhang, see \citep[Theorem~2.31]{StCh08}, shows that the excess classification-risk is bounded by the excess hinge-risk, that is, 
\begin{align*}
\RP{L_{\text{class}}}f - \RPB{L_{\text{class}}}  \leq \RP{L_{\text{hinge}}}f-\RPB{L_{\text{hinge}}}
\end{align*}
for all functions $f:X \to \mathbb{R}$. Hence, we restrict our analysis to the hinge loss and we write in the following $L:= L_{\text{hinge}}$. Since a short calculation shows that 
\begin{align*}
L(y, \max \lbrace -1, \min\lbrace f(x),1 \rbrace) \leq L(y,f(x))
\end{align*}
for all $f:X \to \mathbb{R}$ and $y \in \lbrace -1,1\rbrace$, see e.g.\ \citep[Example~2.27]{StCh08}, it suffices to consider the loss and thus the risk for functions values restricted to the interval $[-1,1]$.  Thus, we define the clipping operator by 
\begin{align*}
\cl{t}:=\max \lbrace -1, \min\lbrace t,1 \rbrace\rbrace
\end{align*}
for $t \in \R$, which restricts values of $t$ to $[-1,1]$, see \citep[Chap.~2.2]{StCh08}. For our decision function in \eqref{locSVMpredictor} this means  that the clipped decision function $\cl{f}_{D,\bs\lb,\bs\g} : X \to [-1,1]$ is then defined by the sum of the clipped empirical solutions $\cl{f}_{D_{j},\lb_j,\g_j}$ since for all $x \in X$ there is exactly one $f_{D_{j},\lb_j,\g_j}$ with $f_{D_{j},\lb_j,\g_j}(x) \neq 0$.

In order to derive learning rates for the localized SVM predictor in \eqref{locSVMpredictor} that measure the speed of convergence of the excess risk $\RP{L}{f_{D,\bs\lb,\bs\g}}-\RPB{L}$ it is necessary to specify our partition $\mathcal{A}$. 
To this end, we denote the ball with radius $r>0$ and center $s \in B_{\ell^d_2}$ by $B_r(s):=\set{ t \in \R^d}{\enorm{t-s} \leq r }$ with Euclidean norm $\enorm{\cdot}$ in $\mathbb{R}^d$ and we define the radius $r_{A}$ of a set $A \subset B_{\ell^d_2}$ by
\begin{align*}
r_{A}= \inf \lbrace r >0: \exists  s \in B_{\ell^d_2} \,\text{such that}\, A \subset B_r(s)  \rbrace.
\end{align*}
\begin{itemize}
 \item[\textbf{(A)}]
Let $\mathcal{A}:=(A_{j})_{j=1,\ldots,m}$ be a partition of $B_{\ell_2^d}$ and $r>0$
such that we have $\mathring{A}_j \neq \emptyset$ for every $j \in \lbrace 1,\ldots,m \rbrace$, and such that there exist  $z_1,\ldots,z_m \in B_{\ell^d_2}$ such that $A_j \subset B_r(z_j)$, and $\enorm{z_i-z_j}\geq \frac{r}{2}$, $i \neq j$, and 
\begin{align}\label{ex. Ueberdeckung}
 r_{A_j} < r \leq 16 m^{-\frac{1}{d}}, \qquad \qquad \text{f.a.}\,\,j\in \lbrace 1,\ldots,m \rbrace
\end{align}
are satisfied.
\end{itemize}
Note that if one considers a Voronoi partition $(A_{j})_{j=1,\ldots,m}$ of $B_{\ell_2^d}$ based on a $r$-net $z_1,\ldots,z_m  \in B_{\ell^d_2}$ with  $r \leq 16 m^{-\frac{1}{d}}$ and $\enorm{z_i-z_j}\geq \frac{r}{2}$, $i \neq j$, the assumptions above are immediately satisfied, see \citep{EbSt16}.

Besides the assumption on the partition above, we need some assumptions on the probability measure $P$ itself. To this end, we recall some notions from \citep[Chap.~8]{StCh08}. Let $\n \colon X \to [0,1]$, defined by $\n(x):=P(y=1|x), x\in X$, be a version of the posterior probability of $P$, which means
that the probability measures $P(\,\cdot \,|x)$ form a regular conditional probability of $P$. Clearly, if we have $\n(x)=0$ resp.\ $\n(x)=1$ for $x \in X$ we observe the label $y=-1$ resp. $y=1$ with probability $1$. Otherwise, if, e.g., $\n(x)\in (1/2,1)$  we observe the label $y=-1$ with the probability $1-\n(x) \in (0,1/2)$ and we call the latter probability noise. 
Obviously, in the worst case this probability equals $1/2$ and we define the set containing those $x \in X$ by $X_{0} := \lbrace\,x \in X \colon \n(x)=1/2 \,\rbrace$. Furthermore, we write
\begin{align*}
X_{1} &:= \lbrace\,x \in X \colon \n(x)>1/2 \,\rbrace,\\
X_{-1} &:=  \lbrace\, x \in X \colon \n(x)<1/2 \,\rbrace.
\end{align*}
Moreover, the function $\D_{\n} \colon X \to [0,\infty]$ defined by
\begin{align}\label{delta}
\begin{split}
\D_{\n}(x):=\begin{cases}
 d(x,X_1)  & \text{if}\, x \in X_{-1},\\
  d(x,X_{-1})  & \text{if}\, x \in X_{1},\\
  0 & \text{otherwise},
\end{cases}
\end{split}
\end{align}
where $d(x,A):=\inf_{x' \in A}d(x,x')$, is called distance to the decision boundary. The following exponents, which describe the mass of the marginal distribution $P_X$ of $P$ around the decision boundary and/or the amount of noise, are weak assumptions to obtain fast learning rates in classification.
We say that $P$ has (Tsybakov) noise exponent (NE) $q \in [0, \infty ]$ if there exist a constant $c_{\text{NE}}>0$ such that
\begin{align}\label{NE}
P_X(\lbrace x\in X: |2\n(x) -1|<\e \rbrace ) \leq (c_{\text{NE}} \e)^q 
\end{align}
for all $\e > 0$, c.f.\ \citep[Def.~8.22]{StCh08}. Note that this exponent is also known as margin exponent. Since it measures the amount of critical noise and does  \textit{not} locate the noise we call \eqref{NE} noise exponent. 
Moreover, we say
that $P$ has margin-noise exponent (MNE) $\b \in (0,\infty]$ if there exists a version $\n$ and a constant $c_{\text{MNE}} > 0$ such that
\begin{align}\label{MNE}
\int_{\lbrace \D_{\n}(x)<t \rbrace} |2\n(x)-1| \,dP_X(x) \leq (c_{\text{MNE}}t)^{\beta}
\end{align}
for all $t>0$. That is, we have a large margin-noise exponent, if we have low mass and/or a large amount of noise around the decision boundary. Next, we say that the distance to the decision boundary $\D_{\n}$ controls the noise from below if there exist a $\z \in [0,\infty)$, a version $\n$, and a constant $c_{\text{LC}}>0$ such that 
\begin{align}\label{def_lc}
\D_{\n}^{\z}(x) \leq c_{\text{LC}}|2\n(x) -1|
\end{align}
for $P_X$-almost all $x \in X$. Descriptively, if $\n(x)$ is close to $1/2$ for some $x \in X$, then \eqref{def_lc} forces $x$ to be located close to the decision boundary. Hence, small values of $\z$ are preferable for learning. For examples of typical values of these exponents and relations between them we refer the reader to \citep[Chap.~8]{StCh08}. 

Finally, we define some mild geometrical assumption on the decision boundary. To this end, we say according to \cite[Sec.~3.2.14(1)]{Federer69} that a general set $T \subset X$ is $m$-rectifiable for an integer $m>0$, if there exists a Lipschitzian function mapping some bounded subset of $\mathbb{R}^m$ onto $T$. Furthermore, we denote by $\partial_X T$ the relative boundary of $T$ in $X$ and we denote by $\mathcal{H}^{d-1}$ the $(d-1)$-dimensional Hausdorff measure on $\mathbb{R}^d$, see \citep[Introduction]{Federer69}.  Then, we state the following assumptions on the decision boundary.
\begin{itemize}
 \item[\textbf{(G)}]\label{assumpG}
Let $\n: X \to [0,1]$ be a fixed version of the posterior probability of $P$. Let $X_0=\partial_X X_1=\partial_X X_{-1}$ and let $X_0$ be $(d-1)$-rectifiable with $\mathcal{H}^{d-1}(X_0)>0$ .
\end{itemize}
Remember that under assumption \textbf{(G)} we have $\mathcal{H}^{d-1}(X_0)<\infty$.
In particular, in \citep[Lemma~2.1]{BlSt18} we showed under assumption $\textbf{(G)}$ how to measure the $d$-dimensional Lebesgue measure $\lb^d$ of a set in the vicinity of the decision boundary, more precisely, we showed that there exists a $\d^{\ast}>0$ and a constant $c_d>0$ such that 
\begin{align}\label{ineq_lebesgue.blst18}
\lb^d(\lbrace \D_{\n}(x) \leq \d \rbrace) \leq c_d \cdot \d, \qquad \text{f.a.}\,\, \d \in (0,\d^{\ast}].
\end{align}

We remark that for some sequences $a_n,b_n \in \R$ we write $a_n \simeq b_n$ if there exists constants $c_1,c_2>0$ such that $a_n \leq c_1 b_n$ and $a_n \geq c_2 b_n$ for sufficiently large $n$.

\section{Classification Rates}\label{sec:classrates}

\subsection{Learning Rates for localized SVMs}\label{subsec:learnrates}

In this section we derive  \textit{global} learning rates for local SVMs with Gaussian kernel and hinge loss.  We apply the splitting technique developed in \citep{BlSt18}, that is, we analyze the excess risk separately on overlapping sets that consists of cells that are close to and sufficiently far away from the decision boundary. By choosing \textit{individual} kernel 
parameters on these sets we obtain local learning rates that we balance out in a last step  to derive global learning rates. To this end, we define for $s>0$ and a fixed version $\n$ of the posterior probability of $P$ the set of indices of cells \textit{near} the decision boundary by
\begin{align*}
J_N^s &:=\set{j\in\{1,\ldots,m\}}{\forall \,x \in A_j : \D_{\n}(x) \leq 3s}
\end{align*}
and the set of indices of cells that are sufficiently \textit{far} away by
\begin{align*}
J_F^s &:= \set{j\in\{1,\ldots,m\}}{\forall \,x \in A_j :  \D_{\n}(x) \geq s}.
\end{align*}
Moreover, we write
\begin{align}\label{def_NF}
N^s := \bigcup_{j\in J_N^{s}} A_j\quad \text{and}\quad F^s := \bigcup_{j\in J_F^{s}} A_j.
\end{align}
Clearly, by dividing our input space into the two overlapping sets defined above we have to be sure to capture all cells in the input space and to assign the cells in $F^s$ either to the class $X_{-1}$ or to $X_1$. The following lemma gives a sufficient condition on our separation parameter $s$. Since the proof is almost identical to the one in \citep[Lemma~3.1]{BlSt18} we skip it here.

\begin{lemma}\label{lemma_sets_svm}
Let $(A_{j})_{j=1,\ldots,m}$ be a partition of $B_{\ell_d^2}$ such that for every $j \in \lbrace 1,\ldots,m \rbrace$ we have $\mathring{A}_j \neq \emptyset$ and (\ref{ex. Ueberdeckung}) is satisfied for some $r>0$. For $s \geq r$ define the sets $N^s$ and $F^s$ by \eqref{def_NF}. Moreover, let $X_0=\partial_X X_1=\partial_X X_{-1}$. Then, we have
\begin{itemize}
\item[i)] $X \subset N^s \cup F^s$,
\item[ii)]  either $A_j \cap X_1 = \emptyset$ or $A_j \cap X_{-1} = \emptyset$ for all $j \in J_F^s$.
\end{itemize}
\end{lemma}

To prevent notational overload, we omit in the sets (of indices) defined above the dependence on $s$ for the rest of this paper, while keeping in mind that all sets depend on this separation parameter. 

Based on an analysis on the sets defined above, we present in the subsequent theorem our main result 
that yields global learning rates for localized SVMs under margin conditions. After that, we  proceed with a detailed explanation of various effects that lead to the theorem. 

\begin{theorem}\label{theorem:main}
Let $P$ be a probability measure on $\R^d \times \{-1,1\}$ for which $P$ has MNE $\b \in (0,\infty]$, NE $q \in [0,\infty]$ and LC $\z\in [0,\infty)$ and let \textbf{(G)} be satisfied for one $\n$. Define $\k:=\frac{q+1}{\b(q+2)+d(q+1)}$. Let  assumption \textbf{(A)} be satisfied for $m_n$ and define
\begin{align*}
r_n:=n^{-\nu},
\end{align*}
where $\nu$ satisfies
\begin{align}\label{def_choice_nu}
\begin{split}
\nu \leq \begin{cases}
  \frac{\k}{1-\k}  &\text{if}\,\, \b \geq (q+1)(1+\max \lbrace d,\z \rbrace -d),\\
  \frac{1-\b \k}{\b \k+\max \lbrace d,\z \rbrace} &\text{else},
  \end{cases}
  \end{split}
\end{align}
and assume that \textbf{(H)} holds. Define for $J = \lbrace 1,\ldots , m_n \rbrace$ the set of indices
\begin{align*}
J_{N_1}&:=\set{j \in J}{\forall x \in A_j: \D_{\n}(x) \leq 3r_n \,\, \text{and}\,\, P_X(A_j \cap X_1)>0\,\, \text{and}\,\, P_X(A_j \cap X_{-1})>0},
\end{align*}
as well as
\begin{align}\label{def_rglb_main}
\begin{split}
\g_{n,j}&\simeq \begin{cases}
  r_n^{\k}n^{-\k}  &\text{for}\,\, j \in J_{N_1},\\
  r_n  &\text{else,}\\
  \end{cases}\\
\lb_{n,j} &\simeq n^{-\s}
\end{split}
\end{align}
for some $\s \geq 1$ and  for every $j \in J$. Moreover, let $\t\geq 1$ be fixed and define for $\d^{\ast}$ considered in \eqref{ineq_lebesgue.blst18}, $n^{\ast}:= \left(\d^{\ast}\right)^{-\frac{1}{\nu}} \max\lbrace \left(\d^{\ast}\right)^{-\frac{1}{\a}} ,2  \rbrace$. Then, for all $\e>0$ there exists a constant $c_{\b,d,\e,q}>0$ such that for all $n \geq n^{\ast}$ the localized SVM classifier satisfies
\begin{align}\label{mainrate}
 \RP{L_J}{\cl{f}_{D,\bs\lb_n,\bs\g_n}}-\RPB{L_J} &\leq c_{\b,d,\e,q} \t  \cdot   n^{-\b \k(\nu+1)+\e}  
\end{align}
with probability $P^n$ not less than $1-9e^{-\t}$. 
\end{theorem}

\begin{remark}\label{rem1}
The exponents of $r_n$ in \eqref{def_choice_nu} match for $\b=(q+1)(1+\max \lbrace d,\z \rbrace -d)$. Moreover, a short calculation shows that in the case  
$\b \geq (q+1)(1+\max \lbrace d,\z \rbrace -d)$
the best possible rate is achieved for  $\nu:=\frac{\k}{1-\k}$ and equals
\begin{align*}
n^{-\frac{\b (q+1)}{\b(q+2)+(d-1)(q+1)}+\e}
\end{align*}
In the other case,  
$\b < (q+1)(1+\max \lbrace d,\z \rbrace -d)$, 
the best possible rate is  achieved for $\nu:=\frac{1-\b \k}{\b \k+\max \lbrace d,\z \rbrace}$ and equals 
\begin{align*}
n^{- \frac{\b \k\left[1+\max \lbrace d,\z \rbrace \right]}{\b \k+\max \lbrace d,\z \rbrace}+\e}.
\end{align*}
\end{remark}

\begin{remark}\label{rem2}
The rates in Theorem~\ref{theorem:main} are better the smaller we choose the cell sizes $r_n$. The smaller $r_n$ the more cells $m_n$ are considered and training localized SVMs is more efficient. To be more precise, the complexity of the kernel matrices or the time complexity of the solver are reduced, see \citep{ThBlMeSt17}. However, \eqref{def_choice_nu} gives a lower bound on $r_n=n^{-\nu}$. For smaller $r_n$ we do achieve rates for localized SVMs, but, we cannot ensure that they learn with the rate \eqref{mainrate}. Indeed, we achieve slower rates. We illustrate this for the case $\b < (q+1)(1+\max \lbrace d,\z \rbrace -d)$. The proof of Theorem~\ref{theorem:main} shows that if 
\begin{align*}
 \left(\frac{q+1}{q+2}\right)(1+\max \lbrace d,\z \rbrace -d) < \b < (q+1)(1+\max \lbrace d,\z \rbrace -d)
\end{align*}
we can choose some $\nu \in \left[ \frac{1-\b \k}{\b \k+\max \lbrace d,\z \rbrace}, \frac{\k}{1-\k} \right]$ such that the localized SVM classifier learns for some $\e>0$  with rate 
\begin{align*}
n^{-(1-\nu \max \lbrace d, \z \rbrace)+\e}.
\end{align*}
A short calculation shows that this rate is indeed slower than \eqref{mainrate} for the given  range of $\nu$ and matches the rate in \eqref{mainrate} only for $\nu := \frac{1-\b \k}{\b \k+\max \lbrace d,\z \rbrace}$. In the worst case, that is, $\nu :=  \frac{\k}{1-\k}$ the rate equals
\begin{align*}
n^{-(1-\nu \max \lbrace d, \z \rbrace)}\!=\!n^{-\left(1-\frac{\k \max \lbrace d,\z \rbrace}{1-\k} \right)}\!=\! n^{-\left(1-\frac{(q+1) \max \lbrace d,\z \rbrace }{\b(q+2)+(d-1)(q+1)} \right)}\!=\!n^{- \frac{\b(q+2)+(q+1)(d- 1- \max \lbrace d,\z \rbrace)}{\b(q+2)+(d-1)(q+1)}}
\end{align*}
up to $\e$ in the exponent, where the numerator is positive since $\b>\left(\frac{q+1}{q+2}\right)(1+\max \lbrace d,\z \rbrace -d)$.

\end{remark}

We discuss the various choices in \eqref{def_choice_nu} and \eqref{def_rglb_main} leading to the theorem above by giving an overview of the main effects influencing its proof. Learning rates are derived from finite sample bounds on the excess risk which follow a typical decomposition into a bound on the approximation error
and on the stochastic error. A key property to bound the stochastic error is to have a variance bound, that is a bound of the form
\begin{align}\label{def_varbound.allgemein}
 \E_P (L \circ f-L \circ f^{\ast}_{L,P})^2&\leq V \cdot (\E_P (L \circ f-L \circ f^{\ast}_{L,P}))^{\th}
\end{align}
with exponent $\th \in (0,1]$ and some constant $V>0$, which descriptively says that if we have a function whose risk is close to $f^{\ast}_{L,P}$ we have low variance. Clearly, the best exponent is $\th=1$ and is obtained e.g.,  for the least-squares loss, see \citep[Example~7.3]{StCh08}. Moreover, \cite[Theorem~8.24]{StCh08} shows for the hinge loss  $\th=\frac{q}{q+1}$ for some NE $q\in [0,\infty]$ and thus, we obtain $\th=1$ only in the special case 
$q=\infty$. However, we show in the next lemma that it is still possible for the hinge loss  to obtain the best possible variance bound $\th=1$ on sets that are sufficiently far away from the decision boundary by using a different margin condition.

\begin{lemma}\label{lem:varbound}
Let $\n \colon X \to [0,1]$ be a fixed version of the posterior probability of $P$. Assume that the associated distance to the decision boundary  $\D_{\n}$ controls the noise from below by the exponent $\z \in [0,\infty)$ and define the set $F:=F^s$ as in \eqref{def_NF}. Furthermore, let $L:=L_{hinge}$ be the hinge loss and let $f^{\ast}_{L,P}: X \to [-1,1]$ be a fixed Bayes decision function. Then, there exists a constant $c_{\text{LC}}>0$ independent of $s$ such that for all measurable $f \colon X \to \R$ we have
\begin{align*}
 \E_P (L_F \circ \fcl-L_F \circ f^{\ast}_{L,P})^2 &\leq \frac{2c_{\text{LC}}}{s^{\z}}  \E_P (L_F \circ \fcl-L_F\circ f^{\ast}_{L,P}).
\end{align*}
\end{lemma}

Besides the stochastic error, we have to bound the approximation error. More precisely, we aim to find an appropriate $f_0 \in H_J$ such that the bound on
\begin{align*}
 \sum_{j\in J} \!\lb_j\|\eins_{A_j}f_0\|^2_{\hat{H}_j}
  +\RP{L_J}{f_0}-\RPB{L_J} 
\end{align*}
is small. Obviously, we control the error above if we control both, the norm and the excess risk. Concerning the norm, we will make the observation that the term is not important since we will be able to choose the regularization parameters $\lb_j$ sufficiently small on each cell, see Sections \ref{sec:oracratesN1} and \ref{sec:oracratesN2F}.  The excess risk is small if $f_0 \in H_J$ is close to a Bayes decision function since its risk is then close to the Bayes risk. Note that we cannot assume the Bayes decision function to be contained in the RKHS $H_J$, see \citep{StCh08}. Nonetheless, we find a function $f_0 \in H_J$ that is similar to a Bayes decision function. To this end, we define $f_0$ on every cell $A$ as the convolution of functions $K_{\g}: \R^d \to \R$ and $f \in L_2(\R^d)$ so that
\begin{align*}
(K_{\g} \ast f)_{|A}  \in H(A),
\end{align*}
and chose $f$ as a function that is similar to a Bayes decision function on a ball containing $A$. Doing this, we observe the following cases.
If a cell $A$ has no intersection with the decision boundary and e.g., 
$A \cap X_1 \neq \emptyset$, but $A \cap X_{-1} = \emptyset$, we have for all $x \in A \cap X_1$ that $f^{\ast}_{\lclass,P}(x)=1$. Otherwise, if 
the cell intersects the decision boundary we find for the decision function that $f^{\ast}_{\lclass,P}:=\sign (2\n-1)$. In order to approximate $f^{\ast}_{\lclass,P}$ by the convolution above, we chose $f$ as constant if the considered cell has no intersection with the decision boundary and as $\sign (2\n-1)$ otherwise. 
Since both depicted cases can occur on the set $N$ we 
divide the set of indices $J_N$ into \begin{align}\label{split_N}
\begin{split}
J_{N_1}&:=\set{j\in J_N}{P_X(A_j \cap X_1)>0 \,\, \text{and}\,\,  P_X(A_j \cap X_{-1})>0 },\\
J_{N_2}&:=\set{ j\in J_N}{P_X(A_j \cap X_1)=0 \,\, \text{or}\,\,  P_X(A_j \cap X_{-1})=0},
\end{split}
\end{align}
and consider our analysis on the corresponding sets $N_1$, $N_2$, and on the set $F$. We refer the reader for a more detailed analysis on the approximation error on those sets to Section~\ref{Bounds on Approximation Error}.

%
%
%
%
Applying the tools above, we obtain by Theorem~\ref{theo:rate_in} on the set $N_1$ with  high probability the bound
\begin{align*}
 \RP{L_{N_1}}{\cl{f}_{D,\bs\lb_n,\bs\g_n}}-\RPB{L_{N_1}} \preceq r_n^{\b k}n^{-\b k}
\end{align*}
for lower bounded $r_n$. In the oracle inequality in Theorem~\ref{theo:oracle_in}, on which the result above is based on, we observe a different behavior in $\g_n$. That means, while the bound on the excess risk in the approximation error tends to zero for $\g_n \to 0$, the bound on the stochastic error behaves in $\g_n$ exactly the opposite way. Motivated by the  approximation of the Bayes decision function described above we  choose sufficiently small kernel parameters $\g_n$, see \eqref{def_rglb_main}, leading to a convolution with a steep kernel, while still having control over the stochastic error. The restriction on $r_n$ guarantees that this $\g_n$ satisfies the condition $\g_n \leq r_n$, which is required to measure the capacity of the underlying Gaussian RKHSs by entropy numbers, see Section~\ref{Learning rates on sets}. If $r_n \to 0$ the bound over $N_1$  tends to zero. This is not the case for the bounds on the sets $N_2$ and $F$ that have no intersection with the decision boundary. By Theorems \ref{theo:rate_inout} and \ref{theo:rate_out} we obtain with high probability on the sets $N_2$ and $F$ bounds of the form
\begin{align*}
 \RP{L_{N_2}}{\cl{f}_{D,\bs\lb_n,\bs\g_n}}-\RPB{L_{N_2}} \preceq \left(\frac{s_n}{r_n^d}\right)^{\frac{q+1}{q+2}} n^{-\frac{q+1}{q+2}} 
\end{align*}
and
\begin{align*}
 \RP{L_F}{\cl{f}_{D,\bs\lb_n,\bs\g_n}}-\RPB{L_F} \preceq \max  \lbrace r_n^{-d},s_n^{-\z}\rbrace  \cdot n^{-1 +\e}.
\end{align*}
These bounds are based on the oracle inequalities in Theorems \ref{theo:oracle_inout} and \ref{theo:oracle_out} in which we observe the same trade-off in $\g$, as described above for the bound in Theorem~\ref{theo:oracle_in}. However, in these cases we choose large $\g_n$, see \eqref{def_rglb_main}, leading to convolutions with  flat kernels. As noted above, the largest possible $\g_n$ equals $r_n$. Both bounds depend in an opposite way on the separation parameter $s_n$. In \citep{BlSt18} this is handled by a straightforward optimization over the parameter $s_n$. Unfortunately, in our case the optimal $s^{\ast}$ does not fulfil the basic requirement $r_n \leq s^{\ast}$ that results from Lemma \ref{lemma_sets_svm}. 
We bypass this difficulty by choosing $s_n=r_n$ in the proof of our main Theorem~\ref{theorem:main}. This choice has two effects. First, the rates on $N_2$ are always better than the rates on $N_1$. Second, for the rates on $N_1$ and $F$ the combination of our considered margin parameters and the dimension $d$ affects the speed of the rates. This leads to the differentiation of $r_n$ in \eqref{def_choice_nu}. If $\b \geq (q+1)(1+\max \lbrace d,\z \rbrace -d)$ the rate on $N_1$ dominates the one on $F$ and $\nu$ has to fulfil $\nu \leq \frac{\k}{1-\k}$. In the other case, if $\b \leq (q+1)(1+\max \lbrace d,\z \rbrace -d)$, the rate on $F$ dominates $N_1$, but only if $\nu \leq \frac{1-\b \k}{\b k+\max \lbrace d,\z \rbrace}$. Unfortunately, we find in the latter case $\frac{1-\b \k}{\b k+\max \lbrace d,\z \rbrace} \leq \frac{\k}{1-\k}$ such that $r_n$ cannot be chosen that small as in the other case in order to learn with rate $n^{-\b \k (\nu+1)}$. Larger $r_n$ would lead to a worse learning rate.  In summary, the interplay of the considered margin conditions together with the dimension $d$ affects the rate presented in Theorem~\ref{theo:oraclemain}. 

Before comparing our rates in \eqref{mainrate} with rates obtained by other algorithms in the next section, we show that our rates are achieved adaptively by a training validation approach. That means, without knowing the MNE $\b$, the NE $q$ and LC $\z$ in advance. To this end, we briefly describe the training validation support vector machine ansatz given in \citep{EbSt16}. We define $\Lambda:=\left(\Lambda_n \right)$ and $\Gamma:=\left( \Gamma_n  \right)$ as sequences of finite subsets $\Lambda_n \subset (0,n^{-1}]$ and $\Gamma_n \subset (0,r_n]$. For a dataset $D:=((x_1,y_1),\ldots,(x_n,y_n))$ we define
\begin{align*}
D_1&:=((x_i,y_i),\ldots,(x_l,y_l)),\\
D_2&:=((x_{l+1},y_{l+1}),\ldots,(x_n,y_n)),
\end{align*}
where $l:=  \lfloor \frac{n}{2} \rfloor +1$ and $n \geq 4$. Moreover, we split these sets into
\begin{align*}
D_j^{(1)}&:=\left( (x_i,y_i)_{i \in \lbrace 1,\ldots, l \rbrace}  : x_i \in A_j \right) , \qquad j \in \lbrace 1,\ldots, m_n \rbrace,\\
D_j^{(2)}&:=\left( (x_i,y_i)_{i \in \lbrace l+1,\ldots, n \rbrace} : x_i \in A_j \right), \qquad j \in \lbrace 1,\ldots, m_n \rbrace,
\end{align*}
and define $l_j :=|D_j^{(1)}|$ for all $j \in \lbrace 1,\ldots, m_n \rbrace$
such that $\sum_{j=1}^{m_n} l_j=l$. We use $D_j^{(1)}$ as a training set by  computing a local SVM predictor
\begin{align*}
f_{D_j^{(1)},\lb_j,\g_j}:=\underset{f \in \hat{H}_{\g_j}(A_j)}{\arg \min} \lb_j \snorm{f}_{\hat{H}_{\g_j}(A_j)}^2 + \Rx{L_j}{D_j^{(1)}}{f}
\end{align*}
for every $j \in \lbrace 1,\ldots, m_n  \rbrace$. Then, we use $D_j^{(2)}$ to determine $(\lb_j,\g_j)$ by choosing a pair $(\lb_{D_2,j},\g_{D_2,j})\in \Lambda_n \times \Gamma_n$ such that
\begin{align*}
\Rx{L_j}{D_2}{\cl f_{D_j^{(1)},\lb_{D_2,j},\g_{D_2,j}} }= 
\min_{(\lb_j,\g_j) \in \Lambda_n \times \Gamma_n} \Rx{L_j}{D_j^{(2)}}{ \cl f_{D_j^{(1)},\lb_j,\g_j}}.
\end{align*}
Finally, we call the function $f_{D_1,\bs\lb_{D_2},\bs\g_{D_2}}$, defined by
\begin{align}\label{def_TV-SVM}
f_{D_1,\bs\lb_{D_2},\bs\g_{D_2}}:=\sum_{j=1}^{m_n} \eins_{A_j} f_{D_j^{(1)},\lb_{D_2,j},\g_{D_2,j}},
\end{align} 
training validation support vector machine (TV-SVM) w.r.t\ $\Lambda$ and $\Gamma$. We remark that the parameter selection is performed \textit{independently} on each cell and leads to $m \cdot |\Lambda| \cdot |\Gamma|$ many candidates. For more details we refer the reader to \citep[Sec.~4.2]{EbSt16}.

The subsequent theorem shows that the TV-SVM, defined in \eqref{def_TV-SVM}, achieves the same rates as the local SVM predictor in \eqref{locSVMpredictor}.

\begin{theorem}\label{theo:tvmain}
Let the assumptions of Theorem \ref{theorem:main} be satisfied with
\begin{align*}
r_n \simeq n^{-\nu},
\end{align*}
where
\begin{align}\label{def_choice_nu_TV}
\begin{split}
\nu \leq \begin{cases}
  \frac{\k}{1-\k}  &\text{if}\,\, \b \geq (q+1)(1+\max \lbrace d,\z \rbrace -d),\\
  \frac{1-\b \k}{\b \k + \max \lbrace d,\z \rbrace} &\text{else}.\\
  \end{cases}
  \end{split}
\end{align}
Furthermore, fix an $\r_n$-net $\Lambda_n \subset (0,n^{-1}]$ and an $\d_n r_n$-net $\Gamma_n \subset (0,r_n]$ with $\r_n \leq n^{-2}$ and  $\d_n \leq n^{-1}$. Assume that the cardinalities  $|\Lambda_n|$ and $|\Gamma_n|$ grow polynomially in $n$.  Let $\t \geq 1$. Then, for all $\e>0$ there exists a constant $c_{d,\b,q,\e}>0$ such that the TV-SVM, defined in \eqref{def_TV-SVM}, satisfies
\begin{align*}
\RP{L_J}{f_{D_1,\bs\lb_{D_2},\bs\g_{D_2}}}-\RPB{L_J} \leq c_{d,\b,q,\e} \t  \cdot n^{-\b \k (\nu+1)+\e}
\end{align*} 
with probability $P^n$ not less than $1-e^{-\t}$.
\end{theorem}


\subsection{Comparison of Rates}\label{subsec:comp}

In this section we compare the results for localized SVMs with Gaussian kernel and hinge loss from Theorem~\ref{theorem:main} to the results from various classifiers, we mentioned in the introduction. We compare the rates to the ones obtained by global and local SVMs with Gaussian kernel and hinge loss in  \citep[Theorem~3.2]{ThBlMeSt17}, \citep[(8.18)]{StCh08} and \citep{LiZeCh17}. Moreover, we make comparisons with the rates achieved by various plug-in classifier in  
  \citep{KoKr07}, \citep{AuTs07}, \citep{BiCoDaDe14},  \citep{BeHsMi18}, as well as to rates obtained by DNN-classifier in \citep{KiOhKi18}, and by the histogram rule in \citep{BlSt18}. 
We remark that in all comparisons we try to find reasonable sets of assumptions  such that both, our conditions and the conditions of the compared methods are satisfied. This means in particular that our rates as well as the other rates are achieved under less  assumptions. We emphasize that the rates for localized SVMs in Theorem~\ref{theorem:main} do not need an assumption on the existence of a density of the marginal distributions. 
  
Throughout this section we assume  \textbf{(A)} for some $r_n:=n^{-\nu}$, \textbf{(G)} for some $\n$, and  \textbf{(H)} to be satisfied. Moreover, we denote by (i), (ii) and (iii) the following assumptions on $P$:
\begin{itemize}
\item[(i)] $P$ has MNE $\b \in (0,\infty]$,
\item[(ii)] $P$ has NE $q \in [0,\infty]$, 
\item[(iii)] $P$ has LC $\z \in [0,\infty)$.
\end{itemize}
Note that under the just mentioned assumptions the assumptions of Theorem~\ref{theorem:main} for localized SVMs using hinge loss are satisfied. First, we  compare the rates to the known ones for local and global SVMs.

\textbf{Local and global SVM.} 
Under assumptions (i) and (ii), \cite[(8.18)]{StCh08} show that global SVMs using hinge loss and Gaussian kernels learn with the rate
\begin{align}\label{globSVMrate}
n^{-\b\k}=n^{-\frac{\b (q+1)}{\b(q+2)+d(q+1)}}.
\end{align}
We remark, that in the special case that (i) is satisfied for $\b=\infty$ this rate is also achieved for the same method in \citep{LiZeCh17}. The rate is also matched by localized SVMs in \citep{ThBlMeSt17} using hinge loss and Gaussian kernel as well as cell sizes $r_n=n^{-\nu}$  for some  $\nu \leq \k$. We show now that under a mild  additional assumption our derived rates for localized SVMs outperform the one above. To this end, we assume (iii) in addition to (i) and (ii). Then, the rate in \eqref{mainrate} is satisfied and better by $\nu \b\k$ for all $\nu$ our analysis is applied to. According to Remark~\ref{rem1} the improvement is at most $q+1$ in the denominator if $\b \geq (q+1)(1+\max \lbrace d,\z \rbrace -d)$. In the other case, we obtain the fastest rate with $\nu:= \frac{1-\b \k}{\b \k+\max \lbrace d,\z \rbrace}$  such that the exponent of our rate in  \eqref{mainrate}  equals
\begin{align}\label{rate_ii_exp}
\tfrac{\b \k (1+ \max \lbrace d,\z \rbrace )}{\b \k + \max \lbrace d,\z \rbrace}= \tfrac{\b (q+1) (1+ \max \lbrace d,\z \rbrace) }{\b(q+1)+\max \lbrace d,\z \rbrace (\b(q+2)+d(q+1))}=\tfrac{\b (q+1)}{\b(q+2)+d(q+1)-\frac{\b+d(q+1)}{1+\max \lbrace d,\z \rbrace}}.
\end{align}
Compared to \eqref{globSVMrate} we then have at most an improvement of $\frac{\b+d(q+1)}{1+\max \lbrace d,\z \rbrace}$ in the denominator.$\blacktriangleleft$

The main improvement in the comparison above results from the strong effect of the lower-control condition (iii). Descriptively, (iii) restricts the location of  noise in the sense that if we have high noise for some $x \in X$, that is $\n(x)\approx 1/2$, then, (iii) forces this $x$ to be located close to the decision boundary. Note that this does not mean that we have no noise far away from the decision boundary.  It is still allowed to have noise $\n(x) \in (0,1/2-\e] \cup [1/2+\e,1)$ for $x \in X$ and some $\e>0$, only the case that $\n(x)=1/2$ is prohibited. We refer the interested reader to a more precise description of this effect to \citep{BlSt18} and proceed with our next comparison.

In the following, we compare our result with results that make besides assumption (ii) some smoothness condition on $\n$, namely that
\begin{itemize}
\item[(iv)] $\n$ is H\"older-continuous for some $\r  \in (0,1]$.
\end{itemize}
This assumption can be seen as a strong reverse assumption to (iii) since it implies that the distance to the decision boundary controls the noise from above, which means that there exists a $\r $ and a constant $\tilde{c} >0$ such that $\tilde{c} |2\n(x)-1|\leq \D_{\n}^{\r }(x)$ for all $x \in X$, see \citep[Lemma~A.2]{BlSt18}. In particular, if (iii) and (iv) are satisfied, then $\r \leq \z$. Note that we  observe vice versa that a reverse H\"older-continuity assumption implies (iii) if $\n$ is continuous, see Lemma~\ref{lem:revhoelder_lc}.

If we assume (iii) in addition to (ii) and (iv) we satisfy the assumptions for localized SVMs in  Theorem~\ref{theorem:main} since we find with \citep[Lemma~A.2]{BlSt18} and \citep[Lemma~8.23]{StCh08} that the MNE equals $\b=\r(q+1)$. We observe that
\begin{align}\label{comp_schranke}
\b=\r(q+1) \leq (q+1)(1+\max \lbrace d,\z \rbrace -d)
\end{align}
and according to Theorem~\ref{theorem:main} the localized SVMs learn with the rate  
\begin{align}\label{comp_locrate}
n^{-\b \k(\nu+1)}
\end{align}
for arbitrary $\nu \leq \frac{1-\b \k}{\b \k + \max \lbrace d,\z \rbrace}$. In particular, this rate is upper bounded by
\begin{align}\label{ineq:rate_exp} 
n^{-\b \k(\nu+1)}<n^{-\b \k}=n^{-\frac{\r(q+1)}{\r(q+2)+d}}.
\end{align}

\textbf{DNN and Plug-in classifier.} 
Under assumption (ii), (iv) and the assumption that the support of the marginal distribution $P_X$ is included in a compact set, the so called ``Hybrid'' plug-in classifiers in \citep[Eq.~(4.1)]{AuTs07} 
learn with the optimal rate
\begin{align}\label{comp_opt}
n^{-\frac{\r(q+1)}{\r(q+2)+d}},
\end{align}
see \citep[Theorem~4.3]{AuTs07}. The same rate is achieved by  deep neural network classifiers in \citep[Theorem~2]{KiOhKi18}. If we assume in addition (iii),  the localized SVM rate again equals \eqref{comp_locrate} and satisfies \eqref{ineq:rate_exp} such that our rate is faster for arbitrary $\nu \leq \frac{1-\b \k}{\b \k + \max \lbrace d,\z \rbrace}$.  For $\nu:=\frac{1-\b \k}{\b \k + \max \lbrace d,\z \rbrace}$ we find for the exponent in \eqref{comp_locrate} that
\begin{align*}
\b \k (\nu+1)\!=\!\tfrac{\b \k (1+\max \lbrace d,\z \rbrace)}{\b \k + \max \lbrace d,\z \rbrace}\!=\! \tfrac{\r(q+1)}{\frac{\r(q+1) +\max \lbrace d,\z \rbrace (\r(q+2)+d)}{1+\max \lbrace d,\z \rbrace}} 
\!=\!\tfrac{\r(q+1)}{\r(q+2)+d-\frac{\r+d}{1+\max \lbrace d,\z \rbrace}}
\end{align*}
such that we have at most an improvement of $\frac{\r+d}{1+\max \lbrace d,\z \rbrace}$ in the denominator.$\blacktriangleleft$

In the comparison above the localized SVM rate outperforms the optimal rate by making the additional assumption (iii). This is not surprising, since the assumptions we made imply the assumptions of \citep{AuTs07}. We emphasize once again that our rates as well as the other rates are achieved under less  assumptions.

\textbf{Tree-based and Plug-in classifier.} 
Assume that (ii) and (iv) are satisfied. Then, the classifiers resulting from the tree-based adaptive partitioning methods in \citep[Sec.~6]{BiCoDaDe14} yield under assumptions (ii) and (iv) the rate
\begin{align*}
n^{-\frac{\r(q+1)}{\r(q+2)+d}},
\end{align*}
see \citep[Theorems~6.1(i) and 6.3(i)]{BiCoDaDe14}. 
In fact the rate is achieved under milder assumptions, namely (ii) and some condition on the behavior of the approximation error w.r.t.\ $P$, however, by \citep[Prop.~4.1]{BiCoDaDe14} the latter is immediately satisfied under (ii) and (iv). Moreover, \cite[Theorems~1, 3, and 5]{KoKr07} showed that plug-in-classifiers  based on kernel, partitioning and nearest neighbor regression estimates learn with rate
\begin{align}\label{rate_kokr1}
n^{-\frac{\r (q+1)}{\r(q+3)+d}}.
\end{align}
Actually, this rate holds under a slightly weaker assumption than (ii), namely that there exists a $\bar{c}>0$ and some $\a>0$ such that for all $\d>0$ the inequality
\begin{align*}
\mathbb{E}(|\n-1/2| \cdot  \eins_{\lbrace |\n-1/2|\leq \d \rbrace} ) \leq \bar{c} \cdot \d^{1+\a}
\end{align*}
is satisfied, but this is implied by (ii), see \citep[Sec.~5]{DoeGyWa15}. To compare our rates we add (iii) to (ii) and (iv). Then, the localized SVM rate again equals \eqref{comp_locrate} and is faster for all $\nu$ our analysis is applied to. The improvement to the rate from \citep{BiCoDaDe14}
is equal to the improvement in the previous comparison, whereas compared to the rate from \citep{KoKr07} the improvement is at least better by $\r$ in the denominator.$\blacktriangleleft$

The three comparisons above have in common that rates are solely improved by assumption (iii). This condition was even sufficient enough to improve the optimal rate in \eqref{comp_opt}. It is to emphasize that neither for the rates from Theorem~\ref{theorem:main} or the rates from the mentioned authors above nor in our comparisons assumptions on the existence of a density of the marginal distribution $P_X$ have to be made.
As mentioned in the introduction assumptions without conditions on the density of distributions are preferable, however, to compare our rates we find subsequently   assumption sets that do contain those.

\textbf{Plug-in classifier I.} 
Let us assume that (ii) and (iv) are satisfied and that $P_X$ has a uniformly bounded density w.r.t.\ the Lebesgue measure.
Then, \cite[Theorem~4.1]{AuTs07} shows that plug-in classifiers learns with the optimal rate
\begin{align*}
n^{-\frac{\r(q+1)}{\r(q+2)+d}}.
\end{align*}
If we assume in addition (iii),  the localized SVM rate again equals \eqref{comp_locrate} and satisfies \eqref{ineq:rate_exp} such that our rate is faster for arbitrary $\nu \leq \frac{1-\b \k}{\b \k + \max \lbrace d,\z \rbrace}$. $\blacktriangleleft$

Before we proceed, we define another margin condition that measures the amount of mass close to the decision boundary and we say according to \citep[Definition~8.6]{StCh08} that $P$ has margin exponent (ME) $\a \in (0,\infty]$, if there exists a constant $c_{\text{ME}}>0$ such that
\begin{align}\label{ME}
P_X(\lbrace \D_{\n}(x)<t\rbrace) \leq (c_{\text{ME}}t)^{\a}
\end{align}
for all $t>0$. Descriptively, large values of $\a$ reflect a low concentration of mass in the vicinity of the decision boundary.

\textbf{Plug-in classifier II.} 
Let us assume that (ii), (iv) are satisfied and that
that $P_X$ has a density with respect to the Lebesgue measure that is bounded away from zero. 
Then, the authors in \cite{BeHsMi18} show that 
plug-in classifiers based on a weighted and interpolated nearest neighbor scheme obtain the rate
\begin{align}\label{comp_belkin}
n^{-\frac{\r q}{p(q+2)+d}}.
\end{align} 
Under the same conditions, 
\cite{KoKr07} improved for 
plug-in-classifier  based on kernel, partitioning, and nearest neighbor regression estimates the rate 
in \eqref{rate_kokr1} to 
\begin{align}\label{comp:rate}
n^{-\frac{\r(q+1)}{2\r+d}}.
\end{align}
By reason of comparison we add (iii) to (ii) and (iv). Then, the localized SVM rate equals \eqref{comp_locrate} and satisfies \eqref{ineq:rate_exp} such that our rate is obviously  faster than the rate in \eqref{comp_belkin} for all possible choices of $\nu$. The improvement  compared to \eqref{comp_belkin} is at least $\frac{\r}{\r(q+2)+d}$. In order to compare our rate with \eqref{comp:rate} we take a closer look on the rate and the margin parameters under the stated conditions.   A short calculation shows for the exponent of the rate in \eqref{comp_locrate} that 
\begin{align*}
\b \k(\nu+1)=\tfrac{\r (q+1)}{\frac{\r (q+2)+d}{(\nu+1)}}=\tfrac{\r (q+1)}{2 \r +d + \frac{\r (q+2)- 2\r(\nu+1) + d  -d(\nu+1)}{(\nu+1)}}
\end{align*}
and its easy to derive that our exponent is only larger than the one in \eqref{comp:rate} or equals it if $\nu \geq \frac{pq}{2\r+d}$. We show that the largest $\nu$ we can choose satisfies this bound if $\r=1$ and derive a rate for this case. Since $P_X$ has a density with respect to the Lebesgue measure that is bounded away from zero, we restrict ourselves to the case that $\r q \leq 1$ and hence $q \leq 1$, see Remark~\ref{rem:app}. Moreover, \citep[Lemma~A.2]{BlSt18} and \citep[Lemma~8.23]{StCh08} yield $\a=\r q=q$. Furthermore, we find by  Lemma~\ref{lem:lc+me=ne} that $q=\frac{\a}{\z}$ and we follow $\z=\r = 1$. Thus, a short calculation shows that
\begin{align*}
\nu:=\tfrac{1-\b \k}{\b \k + \max \lbrace d,\z \rbrace}= \tfrac{1-\b \k}{\b \k + d}=\tfrac{\r+d}{\r(q+1)+d(\r(q+2)+d)}=\tfrac{1+d}{q+1+d(q+2+d)} \geq \tfrac{q}{2+d}
\end{align*} 
is satisfied for all $q \leq 1$. By inserting this $\nu$ into the exponent of the localized SVM rate in \eqref{comp_locrate} we find
\begin{align*}
\b \k (\nu+1)\!=\!\tfrac{\b \k (1+d)}{\b \k + d}\!=\!\tfrac{\r(q+1)}{\r(q+2)+d+\frac{\r(q+1)-(\r(q+2)+d)}{1+d}}\!=\!\tfrac{q+1}{q+2+d+\frac{q+1-(q+2+d)}{1+d}}\!=\!\tfrac{q+1}{q+1+d}.
\end{align*}
Hence, the localized SVM rate is faster than the rate in \eqref{comp:rate} for all $q < 1$ and matches it if $q=1$.
$\blacktriangleleft$ 

Under assumptions that contained that $P_X$ has a density w.r.t.\ Lebesgue measure that is bounded away from zero, we improved in the previous comparison the rates from \citep{BeHsMi18} and in the case that $\n$ is Lipschitz, the rates from \citep{KoKr07}. We remark that under a slight stronger density assumption \cite{AuTs07} showed that certain plug-in classifier achieve the optimal rate in \eqref{comp:rate}.


Finally, we compare our rates to the ones derived for the histogram rule in \citep{BlSt18}, where we also considered a set of margin conditions and a similar strategy to derive their rates. Note that under a certain assumption set the authors showed that the histogram rule outperformed the global SVM rates from \citep[(8.18)]{StCh08} and the localized SVM rates from \citep{ThBlMeSt17}.

\textbf{Histogram rule.} 
Let us assume that (i) and (iii) are satisfied and that
\begin{itemize}
\item[(v)] $P$ has ME $\a \in (0,\infty]$,
\end{itemize}
see \eqref{ME}. 
Then, we find by Lemma~\ref{lem:lc+me=ne} that we have NE $q=\frac{\a}{\z}$ and according to 
 \cite[Theorem~3.5]{BlSt18} the histogram rule then learns with rate 
\begin{align}\label{rate_hist}
n^{-\frac{\b(q+1)}{\b (q+1)+d(q+1)+ \frac{\b\z}{1+\z}}}
\end{align}
as long as $\b \leq (1+\z)(q+1)$. Under these assumptions the localized SVM learns with the rate from Theorem~\ref{theorem:main} that is
\begin{align*}
n^{-\b \k (\nu +1)}, 
\end{align*}
where our rate depends on $\nu$. To compare our rates we have to pay attention to the range of $\b$ that provides a suitable $\nu$, see \eqref{def_choice_nu}. If we have that $(q+1)(1+\max \lbrace d,\z \rbrace -d)\leq \b \leq (q+1)(1+\z)$, then a short calculation shows that our local SVM rate in \eqref{mainrate} is faster if $\nu$ is not too small, that is if $\nu$ satisfies
\begin{align*}
 \b \left( (\b+d)(q+1)(\z+1)+\b \z \right)^{-1}  \leq \nu \leq \tfrac{\k}{1-\k}
\end{align*}

According to Remark~\ref{rem1} the best possible rate is achieved for $\nu:=\frac{\k}{1-\k}$ and has then exponent
\begin{align*}
\tfrac{\b (q+1)}{\b(q+2)+(d-1)(q+1)}=\tfrac{\b (q+1)}{\b(q+1)+d(q+1)+\frac{\b\z}{(1+\z)}+\b-(q+1)-\frac{\b\z}{1+\z}}=\tfrac{\b (q+1)}{\b(q+1)+d(q+1)-\left[q+\frac{\z}{1+\z}\right]}
\end{align*}
such that compared to \eqref{rate_hist} we have an improvement of $q+\frac{\z}{1+\z}$ in the denominator. In the other case, that is, if  $\b \leq (q+1)(1+\max \lbrace d,\z \rbrace -d)$ a short calculation shows that our local SVM rate is better for all choices
\begin{align*}
\b \left( (\b+d)(q+1)(\z+1)+\b \z \right)^{-1}  \leq \nu \leq \tfrac{1-\b \k}{\b k+\max \lbrace d,\z \rbrace},
\end{align*}
In this case we find due to Remark~\ref{rem1} that the best possible rate is achieved for $\nu:= \tfrac{1-\b \k}{\b k+\max \lbrace d,\z \rbrace}$ and has exponent
\begin{align*}
 \tfrac{\b \k\left[1+\max \lbrace d,\z \rbrace \right]}{\b k+\max \lbrace d,\z \rbrace}=\tfrac{\b (q+1)}{\b(q+1)+d(q+1)+\frac{\b\max \lbrace d,\z \rbrace -d(q+1)}{1+\max \lbrace d,\z \rbrace}}&=\tfrac{\b (q+1)}{\b(q+1)+d(q+1)+ \frac{\b\z}{1+\z}-\left[\frac{d(q+1)-\b\max \lbrace d,\z \rbrace}{1+\max \lbrace d,\z \rbrace}+ \frac{\b\z}{1+\z}\right]}.
\end{align*}
Compared to  \eqref{rate_hist} the rate is better by $\frac{d(q+1)-\b\max \lbrace d,\z \rbrace}{1+\max \lbrace d,\z \rbrace}+ \frac{\b\z}{1+\z}>0$ in the denominator. We remark that the lower bound on $\nu$ is not surprising since if $\nu \to 0$ our rate matches the global rate in \eqref{globSVMrate} and \cite{BlSt18} showed that under a certain assumption set the rate of the histogram classifier is faster than the one of the global SVM. Moreover, we remark that our rates in Theorem~\ref{theorem:main} hold  for all values of $\b$ and not only for a certain range of $\b$. $\blacktriangleleft$

\section{Proofs} \label{sec:proof}

In this section we state the proofs of the previous sections. We define $\g_{\max}:=\max_{j \in J} \g_j$ and $\g_{\min}:=\min_{j \in J} \g_j$ w.r.t.\ some $J \subset \lbrace 1,\ldots,m \rbrace$.

 \subsection{Proof of Main Results}

\begin{proof}[Proof of Lemma \ref{lem:varbound}]
Since $\fcl: X \to [-1,1]$ we consider functions $f: X \to [-1,1]$. Then, an analogous calculation as in the proof of \citep[Theorem~8.24]{StCh08} yields $(L_{F} \circ f-L_{F} \circ f^{\ast}_{L,P})^2=f-f^{\ast}_{L,P}$. Following the same arguments as in \citep[Lemma 3.4]{BlSt18} we find for all $x \in F$ with the lower-control assumption that 
\begin{align*}
1 \leq \frac{c_{LC}}{s^{\z}}|2\n(x)-1|.
\end{align*}
Then, we have
\begin{align*}
 \E_P (L_{F} \circ f-L_{F} \circ f^{\ast}_{L,P})^2 &= \int_{F} |f(x)-f^{\ast}_{L,P}(x)|^2 dP_X(x)\\
 &\leq 2\int_{F} |f(x)-f^{\ast}_{L,P}(x)| dP_X(x)\\
 &\leq \frac{2c_{LC}}{s^{\z}} \int_{F} |f(x)-f^{\ast}_{L,P}(x)| |2\n(x)-1|dP_X(x)\\
 &\leq \frac{2c_{LC}}{s^{\z}}  \E_P (L_{F} \circ f-L_{F}\circ f^{\ast}_{L,P}).
\end{align*}
\end{proof}

\begin{proof}[Proof of Theorem \ref{theorem:main}]
By Theorem \ref{lemma_sets_svm} for $s_n:=n^{-\nu}$ we find that
\begin{align}\label{ineq:risk_main}
\begin{split}
&\RP{L_J}{\cl{f}_{D,\bs\lb_n,\bs\g_n}}-\RPB{L_J}\\
&\leq \RP{L_N}{\cl{f}_{D,\bs\lb_n,\bs\g_n}} -\RPB{L_N}+ \RP{L_F}{\cl{f}_{D,\bs\lb_n,\bs\g_n}}-\RPB{L_F} \\
&\leq \RP{L_{N_1}}{\cl{f}_{D,\bs\lb_n,\bs\g_n}} -\RPB{L_{N_1}}+\RP{L_{N_2}}{\cl{f}_{D,\bs\lb_n,\bs\g_n}} -\RPB{L_{N_2}}+ \RP{L_F}{\cl{f}_{D,\bs\lb_n,\bs\g_n}}-\RPB{L_F}.
\end{split}
\end{align} 
In the subsequent steps we bound the excess risks above separately for both choices of $\nu$ by  applying Theorems \ref{theo:rate_in}, \ref{theo:rate_inout} and \ref{theo:rate_out} for $\a:=\nu$. 
First, we consider the case $\b \geq (q+1)(1+\max \lbrace d,\z \rbrace -d)$ and check some requirements for the mentioned theorems. Since $\b \geq \frac{q+1}{q+2}(1+ \max \lbrace d,\z \rbrace -d)$ we have
\begin{align*}
\nu &\leq \tfrac{\k}{1-\k}=\tfrac{q+1}{\b(q+2)+(d-1)(q+1)}\leq  \tfrac{1}{\max \lbrace d,\z \rbrace}.
\end{align*}
Moreover, 
\begin{align*}
1+\nu(1-d) \geq 1- \tfrac{\k(d-1)}{1-\k}=1-\tfrac{(q+1)(d-1)}{\b(q+2)+(d-1)(q+1)}=\tfrac{\b(q+2)}{\b(q+2)+(d-1)(q+1)} >0.
\end{align*}
Hence, we apply 
Theorem \ref{theo:rate_in} and Theorems \ref{theo:rate_inout}, \ref{theo:rate_out} with $\a:=\nu$. That means, together with 
\begin{align}\label{eq1}
\b\k(\nu+1)\leq \tfrac{\b\k}{1-\k}=\tfrac{(q+1)\b(q+2)}{(q+2)\left[\b(q+2)+(d-1)(q+1)\right]}=\tfrac{q+1}{q+2} \left[ 1-\tfrac{\k(d-1)}{1-\k} \right] \leq \tfrac{(q+1)(1-\nu (d-1))}{q+2} 
\end{align}
and
\begin{align}\label{eq2}
\b\k(\nu+1)\leq \tfrac{\b(q+1)}{\b(q+2)+(d-1)(q+1)} =1-\tfrac{(d-1)(q+1)+\b}{\b(q+2)+(d-1)(q+1)}\leq 1-\tfrac{(q+1)\max \lbrace d,\z \rbrace }{\b(q+2)+(d-1)(q+1)} \leq 1-\tfrac{\k \max \lbrace d,\z \rbrace }{1-\k} 
\end{align}
so that $\b\k(\nu+1) \leq 1-\nu \max \lbrace d,\z \rbrace$ for $\nu \leq \frac{\k}{1-\k}$, we obtain in  \eqref{ineq:risk_main} for $\e_1,\e_2,\e_3>0$ and with probability $P^n$ not less than $1-9e^{-\t}$ that
\begin{align}\label{risk_i}
\begin{split}
&\RP{L_J}{\cl{f}_{D,\bs\lb_n,\bs\g_n}}-\RPB{L_J}\\
&\leq \RP{L_{N_1}}{\cl{f}_{D,\bs\lb_n,\bs\g_n}} -\RPB{L_{N_1}}+\RP{L_{N_2}}{\cl{f}_{D,\bs\lb_n,\bs\g_n}} -\RPB{L_{N_2}}+ \RP{L_F}{\cl{f}_{D,\bs\lb_n,\bs\g_n}}-\RPB{L_F}\\
&\leq c_1 \t \left(  n^{-\b \k(\nu+1)} n^{\e_1}  +  n^{-\frac{(q+1)(1+\a-\nu d)}{q+2}} n^{\e_2} + n^{-(1-\max \lbrace \nu d,\a \z \rbrace)} n^{\e_3} \right)\\
&\leq c_2\t n^{\e} \left(  2n^{-\b \k(\nu+1)}+ n^{-(1-\nu \max \lbrace d, \z \rbrace)} \right)\\
&\leq c_3 \t  n^{-\b \k(\nu+1)+\e},
\end{split}
\end{align}
holds for some $\e:=\max \lbrace \e_1,\e_2,\e_3 \rbrace$ and some constants  $c_1$ depending on $d,\b,q,\xi,\e_1,\e_2,\e_3$, and $c_2,c_3>0$ depending on $d,\b,q,\xi,\e$.

Second, we consider the case $\b<(q+1)(1+\max \lbrace d,\z \rbrace -d)$ and check again the requirements on $\nu \leq \frac{1-\b \k}{\b k+\max \lbrace d,\z \rbrace}$ for the theorems applied above.  We have
\begin{align}\label{eq3}
\nu \leq \tfrac{1-\b \k}{\b \k+\max \lbrace d,\z \rbrace}=\tfrac{\b+d(q+1)}{\b(q+1)+\max \lbrace d,\z \rbrace \left[ \b(q+2) +d(q+1)\right]}
\leq \tfrac{q+1}{\b(q+2)+(d-1)(q+1)} =\tfrac{\k}{1-\k},
\end{align}
and
\begin{align*}
\nu \leq \tfrac{1-\b \k}{\b \k+\max \lbrace d,\z \rbrace}\leq  \tfrac{1-\b \k}{\max \lbrace d,\z \rbrace} \leq  \tfrac{1}{\max \lbrace d,\z \rbrace}.
\end{align*}
Moreover,
\begin{align*}
1+\nu(1- d) \geq 1- \tfrac{(1-\b \k)(d-1)}{\b \k+\max \lbrace d,\z \rbrace} =  \tfrac{\max \lbrace d,\z \rbrace -d(1-\b\k)+1}{\b \k+\max \lbrace d,\z \rbrace}  \geq \tfrac{\max \lbrace d,\z \rbrace -d+1}{\b \k+\max \lbrace d,\z \rbrace} >0.
\end{align*}
Again, we apply  Theorem~\ref{theo:oracle_in}  and Theorems~\ref{theo:oracle_inout}, \ref{theo:oracle_out} for $\a:=\nu$. Together with  \eqref{eq3}  we find similar to \eqref{eq1} and \eqref{eq2} that
\begin{align*}
\b\k(\nu+1)\leq \tfrac{\b\k}{1-\k}=\tfrac{q+1}{q+2} \left[ 1-\tfrac{\k(d-1)}{1-\k} \right]\leq  \tfrac{q+1}{q+2} \left[ 1-\tfrac{(1-\b \k)(d-1)}{\b k+\max \lbrace d,\z \rbrace} \right] \leq \tfrac{(q+1)(1-\nu (d-1))}{q+2}, 
\end{align*}
and
\begin{align*}
\b\k(\nu+1) \leq \tfrac{\b \k (1+ \max \lbrace d,\z \rbrace )}{\b \k + \max \lbrace d,\z \rbrace} =1-\tfrac{\max \lbrace d,\z \rbrace (1-\b \k)}{\b \k + \max \lbrace d,\z \rbrace}\leq 1-\nu \max \lbrace d,\z \rbrace
\end{align*}
such that we obtain in \eqref{ineq:risk_main} that
\begin{align*}
&\RP{L_J}{\cl{f}_{D,\bs\lb_n,\bs\g_n}}-\RPB{L_J}\\
&\leq \RP{L_{N_1}}{\cl{f}_{D,\bs\lb_n,\bs\g_n}} -\RPB{L_{N_1}}+\RP{L_{N_2}}{\cl{f}_{D,\bs\lb_n,\bs\g_n}} -\RPB{L_{N_2}}+ \RP{L_F}{\cl{f}_{D,\bs\lb_n,\bs\g_n}}-\RPB{L_F}\\
&\leq  c_1 \t  \left(n^{\e_1}  n^{-\b \k(\nu+1)}  + n^{\e_2} n^{-\frac{(q+1)(1+\a-\nu d)}{q+2}} + n^{\e_3}n^{-(1-\nu \max \lbrace d,\z \rbrace)}\right)\\
&\leq c_2\t n^{\e} \left(  2n^{-\b \k(\nu+1)}+ n^{-(1-\nu \max \lbrace d,\z \rbrace)}\right)\\
&\leq c_3 \t n^{-\b \k(\nu+1)+\e}, 
\end{align*}
holds with probability $P^n$ not less than $1-9e^{-\t}$.
\end{proof}

\begin{proof}[Proof of Theorem \ref{theo:tvmain}]
We analyze the excess risk $\RP{L}{f_{D_1,\bs\lb_{D_2},\bs\g_{D_2}}}-\RPB{L}$ by applying a generic oracle inequality for empirical risk minimization given in \citep[Theorem~7.2]{StCh08}. According to \cite[Theorem~7.2]{StCh08} we have variance bound $\th=\frac{q}{q+1}$ with constant $V:=6c_{NE}^{\frac{q}{q+1}}$.  Then, for fixed dataset $D_1$ and $\t_n:=\t+\ln(1+(|\Lambda_n| \times |\Gamma_n|)^{m_n})$, as well as $n-k \geq n/4$ for $n \geq 4$, we find by  \citep[Theorem~7.2]{StCh08} with probability  $P^{n-k} \geq 1-e^{-\t}$ that
\begin{align}\label{tv_riskdecomp}
\begin{split}
&\RP{L_J}{f_{D_1,\bs\lb_{D_2},\bs\g_{D_2}}}-\RPB{L_J}\\
&\leq 6\left( \inf_{(\bs\lb,\bs\g) \in (\Lambda_n \times\Gamma_n)^{m_n}} \RP{L}{f_{D_1,\bs\lb,\bs\g}}-\RPB{L} \right) +  4 \left( \frac{48c_{NE}^{\frac{q}{q+1}} (\t+\ln(1+(|\Lambda_n| \times |\Gamma_n|)^{m_n})}{n-k}\right)^{\frac{q+1}{q+2}}\\
&\leq 6\left( \inf_{(\bs\lb,\bs\g)   \in (\Lambda_n \times\Gamma_n)^{m_n}}  \RP{L}{f_{D_1,\bs\lb,\bs\g}}-\RPB{L} \right) +  c_q \left( \frac{\t_n}{n}\right)^{\frac{q+1}{q+2}}\\
&\leq 6\left( \inf_{(\bs\lb,\bs\g)   \in (\Lambda_n \times\Gamma_n)^{m_n}}  \RP{L_N}{f_{D_1,\bs\lb_{D_2},\bs\g_{D_2}}} -\RPB{L_N}+ \RP{L_F}{f_{D_1,\bs\lb_{D_2},\bs\g_{D_2}}}-\RPB{L_F} \right) +  c_q \left( \frac{\t_n}{n}\right)^{\frac{q+1}{q+2}}\\
&\leq  6\biggl( \inf_{(\bs\lb,\bs\g)   \in (\Lambda_n \times\Gamma_n)^{m_n}}  \RP{L_{N_1}}{f_{D_1,\bs\lb_{D_2},\bs\g_{D_2}}} -\RPB{L_{N_1}}+\RP{L_{N_2}}{f_{D_1,\bs\lb_{D_2},\bs\g_{D_2}}} -\RPB{L_{N_2}}\\
&\qquad \qquad+ \RP{L_{F}}{f_{D_1,\bs\lb_{D_2},\bs\g_{D_2}}}-\RPB{L_F} \biggr) +  c_q \left( \frac{\t_n}{n}\right)^{\frac{q+1}{q+2}},
\end{split}
\end{align}
where we decomposed the excess risks according to Theorem~\ref{lemma_sets_svm} for $s_n=r_n$. Next, we consider the infimum over each set separately and we start with set $N_1$. By Theorem~\ref{theo:oracle_in}  for $s_n=r_n$ and $p \in (0,\frac{1}{2})$ we obtain  with probability $P^k \geq 1-3(|\Lambda_n| \times |\Gamma_n|)^{m_n}e^{-\t}$ that
\begin{align*}
&\RP{L_{N_1}}{f_{D_1,\bs\lb,\bs\g}}-\RPB{L_{N_1}}\\
 &\leq  c_1 \left(  \sum_{j \in J} \frac{\lb_j r^d}{\g_j^d}  +  \g_{\text{max}}^{\b}
  +\left(\frac{r_n}{n}\right)^{\frac{q+1}{q+2-p}} \left( \sum_{j \in J} \lb_j^{-1} \g_j^{-\frac{d}{p}}  P_X(A_j) \right)^{\frac{p(q+1)}{q+2-p}} + \left(\frac{\t}{n}\right)^{\frac{q+1}{q+2}}\right)
\end{align*}
holds for all $(\bs\lb,\bs\g)\in \Lambda_n^{m_n} \times \Gamma_n^{m_n}$ simultaneously and some constant $c_1>0$ depending on $d,\b,p,q$. Hence,
we find together with Lemma~\ref{Lemma_adapt}~i) that 
\begin{align}\label{inf1}
\begin{split}
 &\inf_{(\bs\lb,\bs\g) \in (\Lambda_n \times\Gamma_n)^{m_n}} \RP{L_{N_1}}{f_{D_1,\bs\lb,\bs\g}}-\RPB{L_{N_1}} \\
 &\leq \inf_{(\bs\lb,\bs\g) \in (\Lambda_n \times\Gamma_n)^{m_n}} c_1 \left(  \sum_{j \in J} \frac{\lb_j r^d}{\g_j^d}  + \g_{\text{max}}^{\b}
  +\left(\frac{r_n}{n}\right)^{\frac{q+1}{q+2-p}} \left( \sum_{j \in J} \lb_j^{-1} \g_j^{-\frac{d}{p}}  P_X(A_j) \right)^{\frac{p(q+1)}{q+2-p}} + \left(\frac{\t}{n}\right)^{\frac{q+1}{q+2}}\right)\\
  &\leq c_2 \t \cdot n^{-\b \k (\nu+1) + \e_1} 
  \end{split}
\end{align}
for some positive constants $c_2$  depending on $d,\b,q$ and $\e_1$. Second, by Theorem~\ref{theo:oracle_inout} for $s_n=r_n$, $p \in (0,\frac{1}{2})$ and $\hat{\e}>0$ we obtain  with probability  $P^n \geq 1-(1+3(|\Lambda_n|\times |\Gamma_n|)^{m_n})e^{-\t}$ that
\begin{align*}
&\RP{L_{N_2}}{f_{D_1,\bs\lb,\bs\g}}-\RPB{L_{N_2}}\\
&\leq c_3 \left( \left( \frac{r_n}{\g_{\text{min}}}\right)^{d}  \sum_{j \in J} \lb_j  n^{\hat{\e}}+\left(\frac{r_n}{n}\right)^{\frac{q+1}{q+2-p}} \left( \sum_{j \in J} \lb_j^{-1} \g_j^{-\frac{d}{p}}  P_X(A_j) \right)^{\frac{p(q+1)}{q+2-p}} + \left(\frac{\t}{n}\right)^{\frac{q+1}{q+2}}  \right)
\end{align*}
holds for all $(\bs\lb,\bs\g)\in \Lambda_n^{m_n} \times \Gamma_n^{m_n}$ simultaneously and some constant $c_3>0$ depending on $d,\b,p,q$ and $\hat{\e}$. Then,  Lemma~\ref{Lemma_adapt}~ii) yields
\begin{align}\label{inf2}
\begin{split}
&\inf_{(\bs\lb,\bs\g) \in (\Lambda_n \times\Gamma_n)^{m_n}}    \RP{L_{N_2}}{f_{D_1,\bs\lb,\bs\g}}-\RPB{L_{N_2}}\\
&\leq \inf_{(\bs\lb,\bs\g) \in (\Lambda_n \times\Gamma_n)^{m_n}}        c_3 \left(  \left( \frac{r_n}{\g_{\text{min}}}\right)^{d}  \sum_{j \in J} \lb_j  n^{\tilde{\e}}+\left(\frac{r_n}{n}\right)^{\frac{q+1}{q+2-p}} \left( \sum_{j \in J} \lb_j^{-1} \g_j^{-\frac{d}{p}}  P_X(A_j) \right)^{\frac{p(q+1)}{q+2-p}} + \left(\frac{\t}{n}\right)^{\frac{q+1}{q+2}}  \right)\\
&\leq c_4   \t^{\frac{q+1}{q+2}} \cdot n^{\e_2}  \left(r_n^{d-1} n\right)^{-\frac{q+1}{q+2}},
\end{split}
\end{align}
where $c_4>0$ is a constant depending on $d,\b,q$ and $\e_2$. Finally, we examine the infimum on the set $F$. To this end, we have by Theorem~\ref{theo:oracle_out}  for $s_n=r_n$, $p \in (0,\frac{1}{2})$ and $\tilde{e}>0$ that
\begin{align*}
\RP{L_F}{f_{D_1,\bs\lb,\bs\g}}-\RPB{L_F} \leq c_5 \left( \left(\frac{r_n}{\g_{\text{min}}}\right)^{d} \sum_{j \in J} \lb_j n^{\tilde{\e}}   +   \left( \sum_{j \in J} \lb_j^{-1} \g_j^{-\frac{d}{p}}  P_X(A_j) \right)^p n^{-1}+ \frac{\t}{r^{\z}n}  \right),
\end{align*} 
holds with probability $P^k \geq 1-3(|\Lambda_n| \times |\Gamma_n|)^{m_n}e^{-\t}$ and for all $(\bs\lb,\bs\g)\in \Lambda_n^{m_n} \times \Gamma_n^{m_n}$  simultaneously and some  constant $c_5>0$ depending on $d,p,q$ and $\tilde{\e}$. Again, Lemma~\ref{Lemma_adapt}~iii) yields
\begin{align}\label{inf3}
\begin{split}
&\inf_{(\bs\lb,\bs\g) \in (\Lambda_n \times\Gamma_n)^{m_n}}    \RP{L_F}{f_{D_1,\bs\lb,\bs\g}}-\RPB{L_F} \\
&\leq \inf_{(\bs\lb,\bs\g) \in (\Lambda_n \times\Gamma_n)^{m_n}} c_5  \left( \left(\frac{r_n}{\g_{\text{min}}}\right)^{d}  \sum_{j \in J} \lb_j n^{\tilde{\e}}  +   \left( \sum_{j \in J} \lb_j^{-1} \g_j^{-\frac{d}{p}}  P_X(A_j) \right)^p n^{-1}+ \frac{\t}{r_n^{\z}n}\right)\\
&\leq c_6 \t \cdot  \max\lbrace r_n^{-d},r_n^{-\z} \rbrace \cdot  n^{-1+\e_3},
\end{split}
\end{align}
where $c_6>0$ is a constant depending on $d,q$ and $\e_3$. Putting \eqref{inf1}, \eqref{inf2} and \eqref{inf3} into \eqref{tv_riskdecomp} we find with \eqref{ex. Ueberdeckung} and probability $P^n \geq 1-(1+9(|\Lambda_n|\times |\Gamma_n|)^{m_n})e^{-\t}$  that 
\begin{align}
\begin{split}
&\RP{L_J}{f_{D_1,\bs\lb_{D_2},\bs\g_{D_2}}}-\RPB{L_J}\\
&\leq  6\biggl( \inf_{(\bs\lb,\bs\g)   \in (\Lambda_n \times\Gamma_n)^{m_n}}  \RP{L_{N_1}}{f_{D_1,\bs\lb_{D_2},\bs\g_{D_2}}} -\RPB{L_{N_1}}+\RP{L_{N_2}}{f_{D_1,\bs\lb_{D_2},\bs\g_{D_2}}} -\RPB{L_{N_2}}\\
&\qquad \qquad+ \RP{L_{F}}{f_{D_1,\bs\lb_{D_2},\bs\g_{D_2}}}-\RPB{L_F} \biggr) +  c_q \left( \frac{\t_n}{n}\right)^{\frac{q+1}{q+2}}\\
&\leq 6 \left(c_2\t \cdot n^{-\b \k (\nu+1) + \e_1}+   c_4\t^{\frac{q+1}{q+2}}  \cdot n^{\e_2}  \left(r_n^{d-1} n\right)^{-\frac{q+1}{q+2}} + c_6\t  \cdot   \max\lbrace r_n^{-d},r_n^{-\z} \rbrace n^{-1+\e_3} \right) + c_q \left( \frac{\t_n}{n}\right)^{\frac{q+1}{q+2}}\\
&\leq c_7\! \left( \t n^{\e} \!\left(n^{-\b \k (\nu+1)}\!+ \!  \left(r_n^{d-1} n\right)^{-\frac{q+1}{q+2}} +\frac{ \max\lbrace r_n^{-d},r_n^{-\z} \rbrace}{n} \right)\! +\! \left( \frac{\t+\ln(1+(|\Lambda_n| \times |\Gamma_n|)^{m_n})}{n}\right)^{\frac{q+1}{q+2}} \right)\\
&\leq c_7\! \left(\t n^{\e}\! \left(n^{-\b \k (\nu+1)}\!+ \!  \left(r_n^{d-1} n\right)^{-\frac{q+1}{q+2}}\! +\! \frac{ \max\lbrace r_n^{-d},r_n^{-\z} \rbrace}{n} \right)\! +\! \left( \frac{\t}{n}\right)^{\frac{q+1}{q+2}} \!+\!\left( \frac{m_n\ln(2(|\Lambda_n| \times |\Gamma_n|))}{n}\right)^{\frac{q+1}{q+2}} \right)\\
&\leq c_7\! \left( \t n^{\e}\! \left(n^{-\b \k (\nu+1)}\!+\!  2 \left(r_n^{d-1} n\right)^{-\frac{q+1}{q+2}} +  \frac{ \max\lbrace r_n^{-d},r_n^{-\z} \rbrace}{n} \right) +c_d \left( \frac{\ln(2(|\Lambda_n| \times |\Gamma_n|))}{r^d n}\right)^{\frac{q+1}{q+2}} \right)\\
&\leq c_8\t  n^{\e} \left(n^{-\b \k (\nu+1)}+  \left(n^{-\nu(d-1)} n\right)^{-\frac{q+1}{q+2}} + \max\lbrace n^{\nu d},n^{\nu \z} \rbrace n^{-1}\right)\\
&\leq c_8 \t n^{\e} \left(2n^{-\b \k (\nu+1)}+ n^{-1+\nu \max \lbrace d, \xi \rbrace } \right) \\
&\leq c_9\t  \cdot n^{-\b \k (\nu+1)+\e}, 
\end{split}
\end{align}
where in the last step we applied analogous to the calculations as in the proof of Theorem~\ref{theorem:main} that $\b \k (\nu+1) \leq 1-\nu \max \lbrace d, \xi \rbrace$,  where $\e:=\max \lbrace \e_1,\e_2,\e_3 \rbrace$ and where $c_7,c_8,c_9>0$  are constants depending on $d,\b,q$ and $\e$. Finally, a variable transformation in $\t$ yields the result.
\end{proof}

\subsection{Oracle Inequalities and Learning rates on predefined sets}\label{Learning rates on sets}

In this subsection we state the theorems leading to the proof of our main result in Theorem \ref{theorem:main}. They show the individual oracle inequalities and learning rates on the sets defined in \eqref{def_NF} resp.\ \eqref{split_N}. We present first  the general oracle inequality for localized SVMs on that all results are based on and  discuss some necessary results concerning entropy numbers of localized Gaussian kernels. After that we decompose our analysis in the following way. We derive in Section~\ref{Bounds on Approximation Error} bounds on the approximation error on our predefined sets. Then,  in Sections~\ref{sec:oracratesN1} and \ref{sec:oracratesN2F} we present the oracle inequalities and learning rates on the sets $N_1$ resp.\ $N_2$ and $F$. 

Before we state a more general oracle inequality in the next theorem, we recall the definition of so-called entropy numbers, see \citep{CaSt90} or \citep[Definition~A.5.26]{StCh08}, which are necessary to measure the capacity of the underlying RKHS. For normed spaces $(E,\|\,\cdot\,\|_E)$ and $(F,\|\,\cdot\,\|_F)$, as well as 
an integer $i\geq 1$, the $i$-th (dyadic) entropy number of a bounded, linear operator 
$S : E\to F$ is defined by
\begin{align*}
 e_i(S : E\to F) & := e_i(SB_E,\|\,\cdot\,\|_F) \\
 & := \inf\Biggl\{\e>0:\exists s_1,\ldots,s_{2^{i-1}}\in SB_E 
\text{ such that } SB_E \subset \bigcup_{j=1}^{2^{i-1}}(s_j + \e B_F)\Biggr\}\,,
\end{align*}
where we use the convention $\inf\emptyset:=\infty$, and $B_E$ as well as $B_F$ denote 
the closed unit balls in $E$ and $F$, respectively. 

\begin{theorem}[Oracle Inequality for Localized SVMs]\label{theo:oraclemain}
Let $L: X\times Y\times\R\to [0,\infty)$ be the hinge loss.  Based on a partition $(A_{j})_{j=1,\ldots,m}$ of $B_{\ell^d_2}$, where $\mathring{A}_j\neq\emptyset$ 
for every $j\in\{1,\ldots,m\}$, we assume \textbf{(H)}.
Furthermore, for an arbitrary index set 
$J\subset\{1,\ldots,m\}$, we assume that for $\th \in [0,1]$ to be the exponent of the variance bound \eqref{def_varbound.allgemein} w.r.t.\ the loss $L_J$. 
Assume that for fixed $n\geq 1$ there exist constants $p\in(0,1)$ and $a_J > 0$ 
such that 
\begin{align}\label{GO:aventropy} 
\E_{D_X \sim P_X^n} e_i(\id: H_J \to L_2(D_X)) \leq a_J \, i^{-\frac{1}{2p}} 
 \,,\qquad\qquad i\geq1\,. 
\end{align}
Finally, fix an $f_0\in H_J$ with $\inorm{f_0}\leq 1$. Then, for all fixed $\t>0$, 
$\bs\lb:=(\lb_1,\ldots,\lb_m)>0$, and $a := \max\lbrace a_J,2 \rbrace$ the localized SVM predictor given by \eqref{locSVMpredictor} using $\hat{H}_1,\ldots,\hat{H}_m$ and $L_J$ satisfies
\begin{align*}
\begin{split}
  &\sum_{j\in J}\lb_j \|\cl{f}_{D_{j},\lb_j}\|^2_{\hat{H}_j}+\RP{L_J}{\cl{f}_{D,\bs\lb}}-\RPB{L_J} \nonumber\\ 
  &\leq  9 \left(\sum_{j\in J} \!\lb_j\|\eins_{A_j}f_0\|^2_{\hat{H}_j}
  \!+\!\RP{L_J}{f_0}\!-\!\RPB{L_J} \right) 
  \!+\!C \left(\frac{a^{2p}}{n}
  \right)^{\frac{1}{2-p-\th+\th p}}
  \!+\!3 \left(\frac{72 V\t}{n}\right)^{\frac{1}{2-\th}} \! +\! \frac{30\t}{n} 
  \end{split}
\end{align*}
with probability $P^n$ not less than $1-3e^{-\t}$, where 
$C>0$ is a constant only depending on $p,V,\th$.
\end{theorem}

\begin{proof}
We apply \citep[Theorem~5]{EbSt16}. The hinge loss is Lipschitz continuous and can be clipped at $M=1$.  Since $\inorm{f_0}\leq 1$ we have $\inorm{L \circ f_0} \leq 2$ such that $B_0=2$. A look into the proof of \citep[Theorem~5]{EbSt16} shows that two things can be slightly modified. First, it suffices to assume to have average entropy numbers of the form in \eqref{GO:aventropy}. Second, it suffices to consider the individual RKHS-norms on the local set $J \subset \lbrace 1,\ldots,m\rbrace$ instead of the whole set $J=\lbrace 1,\ldots,m\rbrace$. By combining these observations yields the result.
\end{proof}

We remark that the constant $C>0$ in Theorem~\ref{theo:oraclemain} is exactly the constant from \citep[Theorem~7.23]{StCh08}. As the following two lemmata shows, we obtain  a bound of the form \eqref{GO:aventropy}

	\begin{lemma}\label{lemma:entropy}
Let $A \subset B_{\ell^d_2}$ be such that  $\mathring{A} \neq \emptyset$ and $A \subset B_r(z)$ with $r>0,z \in X$. Let $H_{\g}(A)$ be the RKHS of the Gaussian kernel $k_{\g}$ over $A$. Then, for all  $p \in (0,\frac{1}{2})$ there exists a constant $c_{d,p}>0$ such that for all $\g \leq r$ and $i\geq 1$ we have
\begin{align*}
e_i(\id: H_{\g}(A)  \to L_2(P_{X|A})) \leq c_{d,p} \sqrt{P_X(A)} \cdot r^{\frac{d}{2p}} \g^{-\frac{d}{2p}}i^{-\frac{1}{2p}},
\end{align*} 
where $c_{d,p}:= (3c_d)^{\frac{1}{2p}}\left( \frac{d+1}{2ep}\right)^{\frac{d+1}{2p}}$.
\end{lemma}

\begin{proof}
Following the lines of \citep[Theorem~6]{EbSt16} we consider the commutative diagram
\begin{align*}
  \begin{xy}
  \xymatrix{
    H_{\g}(A) \ar[rrr]^{\mathrm{id}} \ar[dd]_{I_{B_r}^{-1} \circ I_A} 
    & & & L_2(P_{X|A}) \\ \\
     H_{\g}(B_r)   \ar[rrr]_{\mathrm{id}} 
    & & & \ell_\infty(B_r) \ar[uu]_{\mathrm{id}}}
  \end{xy}
\end{align*}
where the extension operator  $I_A:  H_{\g}(A) \to H_{\g}(\R^d)$ and the restriction operator 
$I_{B_r}^{-1}:  H_{\g}(\R^d) \to H_{\g}(B_r)$, defined in \citep[Theorem~4.37]{StCh08}, are isometric isomorphisms such that $\snorm{I_{B_r}^{-1} \circ I_A:  H_{\g}(A)  \to H_{\g}(B_r)}=1$. According to \cite[(A.38) and (A.39)]{StCh08} we then have
\begin{align}\label{ineq_entropy}
\begin{split}
&e_i(\id: H_{\g}(A)  \to L_2(P_{X|A}))\\
&\leq \snorm{I_{B_r}^{-1} \circ I_A:  H_{\g}(A)  \to H_{\g}(B_r)} \cdot e_i(\id: H_{\g}(B_r) \to \ell_\infty(B_r) ) \cdot \snorm{\id: \ell_\infty(B_r) \to L_2(P_{X|A})},
\end{split}
\end{align}
where we find for $f\in \ell_\infty(B_r)$ that 
\begin{align}\label{ineq_opnorm}
\snorm{\id: \ell_\infty(B_r) \to L_2(P_{X|A})} \leq \snorm{f}_{\infty} \sqrt{P_X(A)}
\end{align}
since
\begin{align*}
\snorm{f}_{L_2(P_{X|A})}=\left(  \int_X \eins_{A}(x) |f(x)|^2  dP_X(x)\right)^{\frac{1}{2}} \leq \snorm{f}_{\infty} \cdot \left(  \int_X \eins_{A}(x) dP_X(x)\right)^{\frac{1}{2}} \leq  \snorm{f}_{\infty}  \sqrt{P_X(A)}.
\end{align*}
Furthermore, by \citep[(A.38) and (A.39)]{StCh08} and \citep[Theorem~5]{FaSt17} we obtain
\begin{align}\label{ineq_entrfar}
e_i(\id: H_{\g}(B_r)  \to \ell_{\infty}(B_r)) \leq e_i(\id: H_{\frac{\g}{r}}(r^{-1}B) \to \ell_\infty(r^{-1}B))  \leq c_{d,p} \cdot r^{\frac{d}{2p}} \g^{-\frac{d}{2p}}i^{-\frac{1}{2p}},
\end{align}
where $c_{d,p}:= (3c_d)^{\frac{1}{2p}}\left( \frac{d+1}{2ep}\right)^{\frac{d+1}{2p}}$.
Plugging \eqref{ineq_opnorm} and \eqref{ineq_entrfar} into \eqref{ineq_entropy} yields 
\begin{align*}
e_i(\id: H_{\g}(A)  \to L_2(P_{X|A})) \leq c_{d,p} \sqrt{P_X(A)} \cdot r^{\frac{d}{2p}} \g^{-\frac{d}{2p}}i^{-\frac{1}{2p}}.
\end{align*}
\end{proof}

	\begin{lemma}\label{lemma_entropy_set_u_average}
Based on a partition $(A_{j})_{j=1,\ldots,m}$ of $B_{\ell^d_2}$, where $\mathring{A}_j\neq\emptyset$ and  $A_j \subset B_r(z_j)$ for $r>0, z_j \in B_{\ell^d_2}$ 
for every $j\in\{1,\ldots,m\}$, we assume \textbf{(H)}. We denote by $D_X$ the empirical measure w.r.t.\ the dataset $D$. 
Then, for all  $p \in (0,\frac{1}{2})$ there exists a constant $\tilde{c}_{d,p}>0$  such that for all $\g_j \leq r$ and $i\geq 1$ we have
\begin{align*}
e_i(\id: H_J \to L_2(D_X)) &\leq \tilde{c}_{d,p} |J|^{\frac{1}{2p}} r^{\frac{d}{2p}} \left( \sum_{j \in J} \lb_j^{-1}\g_j^{-\frac{d}{p}}  D_X(A_j) \right)^{\frac{1}{2}}i^{\frac{1}{2p}}, \qquad i\geq 1,
\end{align*}
and, for the average entropy numbers we have
\begin{align*}
\E_{D_X \sim P_X^n} e_i(\id: H_J \to L_2(D_X)) \leq \tilde{c}_{d,p}|J|^{\frac{1}{2p}} r^{\frac{d}{2p}}  \left( \sum_{j \in J} \lb_j^{-1}\g_j^{-\frac{d}{p}}   P_X(A_j) \right)^{\frac{1}{2}}i^{\frac{1}{2p}}, \qquad i \geq 1.
\end{align*}
The proof shows that the constant is given by $\tilde{c}_{d,p}:= 2\left(9 \ln(4)c_d\right)^{\frac{1}{2p}}\left( \frac{d+1}{2ep}\right)^{\frac{d+1}{2p}}$.
\end{lemma}

\begin{proof}
We define $a_j:= c_{d,p} \sqrt{D_X(A_j)} \cdot r^{\frac{d}{2p}} \g_j^{-\frac{d}{2p}}$. By Lemma \ref{lemma:entropy} we have 
\begin{align*}
e_i(\id: H_{\g_j}(A_j)  \to L_2(D_{X|A_j})) \leq a_j i^{-\frac{1}{2p}}
\end{align*}
for $j \in J, i \geq 1$. Following the lines of the proof of \citep[Theorem~11]{EbSt16} we find that
\begin{align*}
e_i(\id: H_J \to L_2(D_X)) &\leq 2 |J|^{\frac{1}{2}} \left( 3 \ln(4)  \sum_{j \in J} \lb_j^{-p} a_j^{2p} \right)^{\frac{1}{2p}}i^{\frac{1}{2p}}.
\end{align*}
By inserting $a_j$ and by applying $\snorm{\cdot}_{\ell_p^{|J|}}^p \leq |J|^{1-p} \snorm{\cdot}_{\ell_1^{|J|}}^p$ we obtain 
\begin{align*}
e_i(\id: H_J \to L_2(D_X)) &\leq 2 |J|^{\frac{1}{2}} \left( 3 \ln(4) \sum_{j \in J} \lb_j^{-p} a_j^{2p} \right)^{\frac{1}{2p}}i^{\frac{1}{2p}}\\
&= 2 |J|^{\frac{1}{2}} (3 \ln(4))^{\frac{1}{2p}} \left( \sum_{j \in J}  \lb_j^{-p}\left(c_{d,p} \sqrt{D_X(A_j)} \cdot r^{\frac{d}{2p}} \g_j^{-\frac{d}{2p}}\right)^{2p} \right)^{\frac{1}{2p}}i^{\frac{1}{2p}}\\
&=c_{d,p}2 (3 \ln(4))^{\frac{1}{2p}}|J|^{\frac{1}{2}} r^{\frac{d}{2p}}  \left( \sum_{j \in J} \left( \lb_j^{-1}\g_j^{-\frac{d}{p}}  D_X(A_j) \right)^{p} \right)^{\frac{1}{2p}}i^{\frac{1}{2p}}\\
&\leq \tilde{c}_{d,p} |J|^{\frac{1}{2}} r^{\frac{d}{2p}}  |J|^{\frac{1-p}{2p}} \left( \sum_{j \in J} \lb_j^{-1}\g_j^{-\frac{d}{p}}  D_X(A_j) \right)^{\frac{1}{2}}i^{\frac{1}{2p}}\\
&= \tilde{c}_{d,p} |J|^{\frac{1}{2p}} r^{\frac{d}{2p}} \left( \sum_{j \in J} \lb_j^{-1}\g_j^{-\frac{d}{p}}  D_X(A_j) \right)^{\frac{1}{2}}i^{\frac{1}{2p}},
\end{align*}
where $\tilde{c}_{d,p}:= c_{d,p}2 (3 \ln(4))^{\frac{1}{2p}}$ and $c_{d,p}$ is the constant from  Lemma \ref{lemma:entropy}. Finally, by considering the above inequality in expectation yields
\begin{align*}
\E_{D_X \sim P_X^n} e_i(\id: H_J \to L_2(D_X)) &\leq \tilde{c}_{d,p}|J|^{\frac{1}{2p}} r^{\frac{d}{2p}}  \left( \sum_{j \in J} \lb_j^{-1}\g_j^{-\frac{d}{p}}  \E_{D_X \sim P_X^n} D_X(A_j) \right)^{\frac{1}{2}}i^{\frac{1}{2p}}\\
&\leq \tilde{c}_{d,p}|J|^{\frac{1}{2p}} r^{\frac{d}{2p}}  \left( \sum_{j \in J} \lb_j^{-1}\g_j^{-\frac{d}{p}}   P_X(A_j) \right)^{\frac{1}{2}}i^{\frac{1}{2p}}.
\end{align*}
\end{proof}

\subsubsection{Bounds on Approximation Error}\label{Bounds on Approximation Error}

We define for an $f_0:X \to \R$ the function
\begin{align}\label{def.AEF}
A_J^{(\bs\g)}(\bs\lb)&:= \sum_{j \in J}  \lb_j \snorm{\eins_{A_j} f_0}_{\hat{H}_j}^2 +  \RP{L_J}{f_0} -\RPB{L_J}.
\end{align}
Recall that we aim to find an $f_0 \in H_J$ such that both, the norm and the approximation error in $A_J^{(\bs\g)}(\bs\lb)$ are small. We show in the following that a suitable choice for $f_0$ is a function that is constructed by convolutions of some $f \in L^2(\R^d)$ with the function $K_{\g}: \R^d \to \R$, defined by
\begin{align}\label{def:fm.gauss}
K_{\g}(x)=\left(\frac{2}{\pi^{1/2} \g}\right)^{d/2} e^{-2\g^{-2} \snorm{x-\cdot }^2_2}.
\end{align}
Note that $K_{\g} \ast f(x)= \langle \Phi_{\g}(x),f\rangle_{L_2(\R^d)}$ for $x \in \R^d$, where $\Phi_{\g}$ is a feature map of a Gaussian kernel, see \citep[Lemma~4.45]{StCh08}. The following lemma shows that a restriction of the convolution is contained in a local RKHS and that we control the individual RKHS norms in \eqref{def.AEF}.

\begin{lemma}[\textbf{Convolution}]\label{lemma:norms} 
Let $A \subset B_r(z)$ for some $r>0, z \in B_{\ell_2^d}$. Furthermore, let $H_{\g}(A)$ be the RKHS of the Gaussian kernel $k_{\g}$ over $A$ with $\g>0$ and let the function $K_{\g}: \R^d \to  \R$ by defined as in \eqref{def:fm.gauss}. 
Moreover, for $\r \geq r$ define the function $f^{\r}_{\g}: \R^d \to \R$ by
\begin{align*}
f^{\r}_{\g}(x)&:=(\pi \g^2)^{-d/4} \cdot \eins_{B_{\r}(z)}(x)\cdot \tilde{f}(x),
\end{align*}
where $\tilde{f}: \R^d \to \R$ is some function with $\snorm{\tilde{f}}_{\infty}\leq 1$. Then, we have $\snorm{K_{\g} \ast f^{\r}_{\g}}_{\infty}\leq 1$ and $\eins_{A}(K_{\g} \ast f^{\r}_{\g}) \in \hat{H}_{\g}(A)$ with
\begin{align*}
\snorm{\eins_{A}(K_{\g} \ast f^{\r}_{\g})}_{\hat{H}_{\g}(A)}^2\leq \left( \frac{\r^2}{\pi \g^2}\right)^{d/2} \vol_d(B).
\end{align*} 
\end{lemma}

\begin{proof}
Obviously, $f^{\r}_{\g} \in \Lx{2}{\R^d}$ such that we find
\begin{align}\label{calc_g_rho_L2Norm}
\begin{split}
\snorm{f^{\r}_{\g}}_{\Lx{2}{\R^d}}^2 &= \int_{\R^d} |(\pi \g^2)^{-d/4}\cdot \eins_{B_{\r}(z)}(x)  \tilde{f}(x) |^2 \dx{x} \\
						   &\leq (\pi \g^2)^{-d/2} \int_{\R^d}  |\eins_{B_{\r}(z)}(x)|^2 \dx{x}\\
						   &= (\pi \g^2)^{-d/2} \int_{B_{\r}(z)} 1 \dx{x}\\
						   &= (\pi \g^2)^{-d/2} \vol_d(B_{\r}(z))\\
						   &= \left( \frac{\r^2}{\pi \g^2}\right)^{d/2} \vol_d(B).
\end{split}
\end{align}
Since the map $K_{\g} \ast \cdot: \Lx{2}{\R^d} \to H_{\g}(A)$ given by  
\begin{align*}
K_{\g} \ast g(x):=\left(\frac{2}{\pi^{1/2} \g}\right)^{d/2} \int_{\R^d} e^{-2\g^{-2} \snorm{x-y}^2_2} \cdot g(y)  \dx{y},\qquad \qquad g \in \Lx{2}{\R^d}, x \in A
\end{align*}
is a metric surjection, see \citep[Proposition~4.46]{StCh08}, we find 
\begin{align}\label{ineq_H_und_L2_norm}
\snorm{(K_{\g} \ast f^{\r}_{\g})_{|_{A}}}_{H_{\g}(A)}^2 \leq \snorm{f^{\r}_{\g}}_{\Lx{2}{\R^d}}^2.
\end{align}
Next, Young's inequality, see \citep[Theoreom~A.5.23]{StCh08}, yields
\begin{align}\label{ineq_faltung_kleiner_1}
 \snorm{K_{\g} \ast f^{\r}_{\g}}_{\infty}=\left(\frac{2}{\pi \g^2}\right)^{d/2} \snorm{k_{\g} \ast  (\eins_{B_{\r}(z)} \tilde{f})}_{\infty} \leq \left(\frac{2}{\pi \g^2}\right)^{d/2}  \snorm{k_{\g}}_{1}\snorm{\tilde{f}}_{\infty} \leq 1.
\end{align}
Hence, with (\ref{ineq_H_und_L2_norm}) and (\ref{calc_g_rho_L2Norm}) we find 
\begin{align*}
\snorm{\eins_{A}(K_{\g} \ast f^{\r}_{\g})}_{\hat{H}_{\g}(A)}^2= \snorm{(K_{\g} \ast f^{\r}_{\g})_{|_{A}}}_{H_{\g}(A)}^2 \leq \snorm{f^{\r}_{\g}}_{\Lx{2}{\R^d}}^2 \leq \left( \frac{\r^2}{\pi \g^2}\right)^{d/2} \vol_d(B).
\end{align*}
\end{proof}

In order to bound the excess risks in \eqref{def.AEF} over the sets $N_1, N_2$ and $F$  we apply Zhang's equality given by
\begin{align}\label{zhang_approx}
\RP{L_J}{f_0}-\RPB{L_J}  = \int_{\bigcup_{j\in J}A_j} |f_0-f^{\ast}_{\lclass,P}| |2\n-1|dP_X,
\end{align}
see \citep[Theorem~2.31]{StCh08}. We begin with an analysis on the set $N_1$, whose cells have no intersection with the decision boundary. For such cells the subsequent lemma presents a suitable function $f_0$ and its difference to $f^{\ast}_{\lclass,P}$ that occurs in \eqref{zhang_approx}. In particular, the function $f_0$ is a convolution of $K_{\g}$ with $2\n-1$ since we have $f^{\ast}_{\lclass,P}(x)=2\n(x)-1$ for $x \in A_j$  with $j \in J_{N_1}^s$  as mentioned in Section~\ref{subsec:learnrates}.


\begin{lemma}[\textbf{Convolution on $N_1$ and its difference to $f^{\ast}_{\lclass,P}$}]\label{lemma:abs_in}
Let the assumptions of Lemma~\ref{lemma:norms} be satisfied with  $A \cap X_1 \neq \emptyset$ and $A \cap X_{-1} \neq \emptyset$.
We define the function $f^{3r}_{\g}: \R^d \to \R$ by
\begin{align}\label{def:faltung_in}
f^{3r}_{\g}(x)&:=(\pi \g^2)^{-d/4} \cdot \eins_{B_{3r}(z) \cap (X_1 \cup X_{-1})}(x) \sign(2\n(x)-1).
\end{align}
Then, we find for all $x \in A$ that 
\begin{align*}
\vert K_{\g} \ast f_{\g}^{3r}(x) - f^{\ast}_{\lclass,P}(x)\vert \leq \frac{2}{\G(d/2)}\int_{2\D_{\n}^2(x)\g^{-2}}^{\infty} e^{-t}t^{d/2-1}  \dx{t},
\end{align*}
where $K_{\g}$ is the function defined in \eqref{def:fm.gauss}.
\end{lemma}

\begin{proof}
Let us consider w.l.o.g.\ $x \in A \cap X_1$. Then,
\begin{align}\label{delta_bounded}
\D_{\n}(x)=\inf_{\bar{x} \in X_{-1}} \enorm{x-\bar{x}} \leq \diam (B_r(z))=2r.
\end{align}
Next, we denote by  $\mathring{B}$ the open ball and show that $\mathring{B}_{\D_{\n}(x)}(x) \subset B_{3r}(z)\cap X_1$. For $x' \in \mathring{B}_{\D_{\n}(x)}(x)$ we have $\snorm{x'-x}_{2} < \D_{\n}(x)$ such that $x'\in X_1$. Furthermore, (\ref{delta_bounded}) yields  $\snorm{x'-z_j}_{2}\leq \snorm{x'-x}_{2} + \snorm{x-z_j}_{2} < \D_{\n}(x) + r \leq 2r + r = 3r$ and hence $x' \in B_{3r}(z)$. We find
\begin{align*}
K_{\g} \ast f_{\g}^{3r}(x)&=\left(\frac{2}{\pi^{1/2} \g}\right)^{d/2} \int_{\R^d} e^{-2\g^{-2} \snorm{x-y}^2_2}(\pi \g^2)^{-d/4}  \eins_{B_{3r}(z)\cap (X_1 \cup X_{-1})} (y) \text{sign}\left(2\n(y)-1 \right) \dx{y} \\
			&=\left(\frac{2}{\pi \g^2}\right)^{d/2} \int_{\R^d} e^{-2\g^{-2} \snorm{x-y}^2_2}\cdot \eins_{B_{3r}(z)\cap (X_1 \cup X_{-1})} (y) \text{sign}\left(2\n(y)-1 \right) \dx{y}\\
			&=\left(\frac{2}{\pi \g^2}\right)^{d/2} \left(\int_{B_{3r}(z)\cap X_1} e^{-2\g^{-2} \snorm{x-y}^2_2}\dx{y} - \int_{B_{3r}(z)\cap X_{-1}} e^{-2\g^{-2} \snorm{x-y}^2_2} \dx{y}\right)\\
			&\geq \left(\frac{2}{\pi \g^2}\right)^{d/2} \left(\int_{\mathring{B}_{\D_{\n}(x)}(x)} e^{-2\g^{-2} \snorm{x-y}^2_2} \dx{y} - \int_{\R^d \setminus \mathring{B}_{\D_{\n}(x)}(x)} e^{-2\g^{-2} \snorm{x-y}^2_2} \dx{y}\right)\\
			&= 2 \left(\frac{2}{\pi \g^2}\right)^{d/2} \int_{\mathring{B}_{\D_{\n}(x)}(x)} e^{-2\g^{-2} \snorm{x-y}^2_2} \dx{y} - 1.
\end{align*}
Since $f^{\ast}_{\lclass,P}(x)=1$ we obtain by Lemma~\ref{lemma:norms} for $\r=3r$ and $\tilde{f}=\eins_{X_1 \cup X_{-1}} \sign(2\n-1)$, and by Lemma~\ref{lemma:tech_int} that
\begin{align}\label{calc_A2_Kg-f}
\begin{split}
\vert K_{\g} \ast f_{\g}^{3r}(x) - f^{\ast}_{\lclass,P}(x)\vert &= \vert K_{\g} \ast f_{\g}^{3r}(x) - 1 \vert\\
 								&=1- K_{\g} \ast f_{\g}^{3r}(x)\\
								&\leq 2- 2 \left(\frac{2}{\pi \g^2}\right)^{d/2} \int_{\mathring{B}_{\D_{\n}(x)}(x)} e^{-2\g^{-2} \snorm{x-y}^2_2} \dx{y}\\
&= 2-\frac{2}{\G(d/2)}\int_{0}^{2\D_{\n}^2(x)\g^{-2}}  e^{-t}t^{d/2-1}\dx{t}\\
&=\frac{2}{\G(d/2)} \left( \int_0^{\infty} e^{-t}t^{d/2-1}\dx{t} -\int_{0}^{2\D_{\n}^2(x)\g^{-2}}  e^{-t}t^{d/2-1}\dx{t}\right)\\
&=\frac{2}{\G(d/2)}\int_{2\D_{\n}^2(x)\g^{-2}}^{\infty} e^{-t}t^{d/2-1}  \dx{t}.
\end{split}
\end{align}
The case $x \in A \cap X_0$ is clear and for $x \in A \cap X_{-1}$ the calculation yields the same inequality, hence (\ref{calc_A2_Kg-f}) holds for all $x \in A$.
\end{proof}

Under the assumption that $P$ has some MNE $\b$ we immediately obtain in the next theorem a bound on the approximation error on the set $N_1$. 
\begin{theorem}[\textbf{Approximation Error on $N_1$}]\label{theo:approx.in}
Let \textbf{(A)} and  \textbf{(H)} be satisfied and let $P$ have MNE $\b \in (0,\infty]$. Define the set of indices
\begin{align*}
J:=\set{j\in \lbrace 1, \ldots m \rbrace}{ A_j \cap X_1 \neq \emptyset\, \,\text{and}\, \,A_j \cap X_{-1} \neq \emptyset }.
\end{align*}
and the function $f_0: X \to \R$ by
\begin{align*}
f_0:= \sum_{j \in J} \eins_{A_j} \left( K_{\g_j} \ast f_{\g_j}^{3r} \right), 
\end{align*}
where the functions $K_{\g}$ and $f_{\g_j}^{3r}$ are defined in \eqref{def:fm.gauss} and \eqref{def:faltung_in}.
Then, $f_0 \in H_J$ and $\inorm{f_0} \leq 1$. Moreover, there exist constants $c_d,c_{d,\b}>0$ such that 
\begin{align*}
 A_J^{(\bs\g)}(\bs\lb) \leq c_d  \cdot \sum_{j \in J} \frac{\lb_j r^d}{\g_j^d} +  c_{d,\b} \cdot  \max_{j \in J} \g_j^{\b}.
\end{align*}
\end{theorem}

\begin{proof}
By Lemma~\ref{lemma:norms} for $\r=3r$ and $\tilde{f}=\eins_{X_1 \cup X_{-1}} \sign(2\n-1)$ we have immediately that $f_0 \in H_J$ as well as  $\inorm{f_0}=\inorm{K_{\g} \ast f_{\g}^{3r}} \leq 1$. Moreover, Lemma~\ref{lemma:norms} yields 
\begin{align*}
 \sum_{j \in J}  \lb_j \snorm{\eins_{A_j} f_0}_{\hat{H}_j}^2 = \sum_{j \in J}  \lb_j \snorm{\eins_{A_j}(K_{\g_j} \ast f_{\g_j}^{3r})}_{\hat{H}_j}^2 \leq c_d \cdot \sum_{j \in J} \frac{\lb_j r^d}{\g_j^d} 
\end{align*}
for some constant $c_d>0$. Next, we bound the excess risk of $f_0$. To this end, 
we fix w.l.o.g.\ an $x \in A_j \cap X_1$ and find
\begin{align*}
\D_{\n}(x)=\inf_{\bar{x} \in X_{-1}} \enorm{x-\bar{x}} \leq \diam (B_r(z))=2r
\end{align*}
such that $A_j \subset \lbrace \D_{\n}(x)\leq 2r \rbrace$ for every $j \in J$. 
Together with Zhang's equality, see \citep[Theorem~2.31]{StCh08}, and  Lemma~\ref{lemma:abs_in} we then obtain
\begin{align*}
&\RP{L_J}{f_0}-\RPB{L_J}\\
&=\sum_{j \in J} \int_{A_j} |(K_{\g_j} \ast f_{\g_j}^{3r})(x)- f^{\ast}_{L_{\text{class}},P}(x)||2\n(x)-1| \dxy{P_X}{x}\\
&\leq \frac{2}{\G(d/2)} \sum_{j \in J} \int_{A_j}\int_{2\D_{\n}^2(x)\g_j^{-2}}^{\infty} e^{-t}t^{d/2-1}  \dx{t}|2\n(x)-1| \dxy{P_X}{x}\\
&= \frac{2}{\G(d/2)} \sum_{j \in J} \int_{A_j}\int_{0}^{\infty}\eins_{[2\D_{\n}^2(x)\g_j^{-2}, \infty)} (t) e^{-t} t^{d/2-1} \dx{t}|2\n(x)-1| \dxy{P_X}{x}\\
&=\frac{2}{\G(d/2)}\sum_{j \in J} \int_{A_j} \int_{0}^{\infty} \eins_{[0,\sqrt{t/2}\g_j)}	(\D_{\n}(x)) e^{-t} t^{d/2-1}\dx{t}|2\n(x)-1| \dxy{P_X}{x}\\
&=\frac{2}{\G(d/2)} \int_{0}^{\infty} \sum_{j \in J} \int_{A_j}  \eins_{[0,\sqrt{t/2}\g_j)}(\D_{\n}(x)) |2\n(x)-1| \dxy{P_X}{x}e^{-t} t^{d/2-1}\dx{t}\\
&\leq \frac{2}{\G(d/2)} \int_{0}^{\infty} \int_{\lbrace \D_{\n}(x)\leq 2r \rbrace} \eins_{[0,\sqrt{t/2}\g_{\max})}(\D_{\n}(x)) |2\n(x)-1| \dxy{P_X}{x}e^{-t} t^{d/2-1}\dx{t}\\
&\leq \frac{2}{\G(d/2)} \int_{0}^{\infty}  \int_{\lbrace \D_{\n}(x) \leq \min\lbrace 2r,\sqrt{t/2}\g_{\max}\rbrace \rbrace}|2\n(x)-1| \dxy{P_X}{x} e^{-t}t^{d/2-1} \dx{t}.
\end{align*}
Next, a simple calculation shows that 
\begin{align*}
\min\left\lbrace 2r,\sqrt{t/2} \g_{\max} \right\rbrace= \begin{cases}
  2r,  & \text{if } t \geq 8r^2 \g_{\max}^{-2},\\
  \sqrt{t/2} \g_{\max}, & \text{if else}.
\end{cases}
\end{align*}
and that  $1 \leq \left(   \frac{t\g_{\max}^2}{8r^2}\right)^{\b/2}$ for   $t \geq 8r^2 \g_{\max}^{-2}$. Finally, the definition of the margin-noise exponent $\b$ yields
\begin{align*}
&\RP{L_J}{f_0}-\RPB{L_J}\\
&\leq \frac{2}{\G(d/2)} \int_{0}^{\infty}  \int_{\lbrace \D_{\n}(x) \leq \min\lbrace 2r,\sqrt{t/2}\g_{\max}\rbrace \rbrace}|2\n-1| \dx{P_X} e^{-t}t^{d/2-1} \dx{t}\\
&\leq \frac{2c_{\text{MNE}}^{\b}}{\G(d/2)} \int_{0}^{\infty}  \left( \min\lbrace 2r,\sqrt{t/2}\g_{\max}\rbrace\right)^{\b} e^{-t}t^{d/2-1} \dx{t}\\
&= \frac{2c_{\text{MNE}}^{\b}}{\G(d/2)} \Bigg( \int_{0}^{8r^2 \g_{\max}^{-2}}  \g_{\max}^{\b} \left(\frac{t}{2}\right)^{\b/2} e^{-t}t^{d/2-1} \dx{t}+  \int_{8r^2 \g_{\max}^{-2}}^{\infty}  (2r)^{\b} e^{-t}t^{d/2-1} \dx{t} \Bigg)\\
&\leq \frac{2c_{\text{MNE}}^{\b}}{\G(d/2)} \Bigg( \frac{\g_{\max}^{\b}}{2^{\b/2}} \int_{0}^{8r^2 \g_{\max}^{-2}}  e^{-t}t^{(d+\b)/2-1} \dx{t}+ (2r)^{\b} \int_{8r^2 \g_{\max}^{-2}}^{\infty}  e^{-t}t^{d/2-1} \dx{t} \Bigg)\\
&\leq \frac{2c_{\text{MNE}}^{\b}}{\G(d/2)} \Bigg( \frac{\g_{\max}^{\b}}{2^{\b/2}} \int_{0}^{8r^2 \g_{\max}^{-2}}  e^{-t}t^{(d+\b)/2-1} \dx{t}+ (2r)^{\b} \left(\frac{ \g_{\max}^2}{8r^2}\right)^{\b/2}\int_{8r^2 \g_{\max}^{-2}}^{\infty}  e^{-t}t^{(d+\b)/2-1} \dx{t} \Bigg)\\
&= \frac{2c_{\text{MNE}}^{\b}}{\G(d/2)}  \frac{\g_{\max}^{\b}}{2^{\b/2}} \Bigg( \int_{0}^{8r^2 \g_{\max}^{-2}}  e^{-t}t^{(d+\b)/2-1} \dx{t}+\int_{8r^2 \g_{\max}^{-2}}^{\infty}  e^{-t}t^{(d+\b)/2-1} \dx{t} \Bigg)\\
&= \frac{2^{1-\b/2}c_{\text{MNE}}^{\b} \G((d+\b)/2)}{\G(d/2)} \g_{\max}^{\b}. 
\end{align*}
\end{proof}

In the next step we develop bounds on the approximation error on sets that have no intersection with the decision boundary, that is,  $N_2$ and $F$. Recall that we apply \eqref{zhang_approx} and again, the subsequent lemma presents a suitable function $f_0$ and its difference to $f^{\ast}_{\lclass,P}$ that occurs in \eqref{zhang_approx}. Note that on those sets we have $f^{\ast}_{\lclass,P}(x)=1$ for $x \in A_j$  with $j \in J_{N_2}^s$ or  $j \in J_{F}^s$ and hence, we we choose a function $f_0 \in H_J$ that is a convolution of $K_{\g}$, defined in \eqref{def:fm.gauss}, with a constant function that we have to cut off to ensure that it is an element of $L_2(\R^d)$. Unfortunately, we will always make an error on such cells since \cite[Corollary~4.44]{StCh08} shows that Gaussian RKHSs do not contain constant functions. In order to make the convoluted function as flat as possible on a cell, we choose the radius $\om_{+}$ of the ball on which $f$ is a constant arbitrary large, that is $\om_{+}>r$. We remark that although the radius is arbitrary large we receive by convolution a function that is still contained in a local RKHS over a cell $A_j$.

\begin{lemma}[Difference to $f^{\ast}_{\lclass,P}$ on cells in $N_2$ or $F$]\label{lemma:abs_out}
Let the assumptions of Lemma \ref{lemma:norms} be satisfied with $A \cap X_1 = \emptyset$ or $A \cap  X_{-1} = \emptyset$. For $\om_{-} >0$ we define $\om_{+}:=\om_{-}+r$ and the function $f^{\om_{+}}_{\g}: \R^d \to \R$ by
\begin{align}\label{def:faltung_out}
f^{\om_{+}}_{\g}(x)&:=\begin{cases} 
(\pi \g^2)^{-d/4} \cdot \eins_{B_{\om_{+}}(z)\cap (X_1 \cup X_0)}(x), & \text{if}\ x \in A \cap (X_1 \cup X_0),\\
(-1) \cdot (\pi \g^2)^{-d/4} \cdot \eins_{B_{\om_{+}}(z) \cap X_{-1}}(x), & \text{else}.
\end{cases}
\end{align}
Then, we find for all $x \in A$ that  
\begin{align*}
\vert (K_{\g} \ast f_{\g}^{\om_{+}})(x) - f^{\ast}_{\lclass,P}(x)\vert &\leq \frac{1}{\G(d/2)}\int_{(\om_{-})^22\g^{-2}}^{\infty} e^{-t}t^{d/2-1}  \dx{t},
\end{align*}
where $K_{\g}$ is the function defined in \eqref{def:fm.gauss}.
\end{lemma}

\begin{proof}
We assume w.l.o.g.\ that $x \in  A \cap (X_1 \cup X_0)$ and show in a first step that $\mathring{B}_{\om_{-}}(x) \subset B_{\om_{+}}(z)$. To this end, consider an $x' \in \mathring{B}_{\om_{-}}(x)$. Since $A \subset B_r(z)$ we find 
\begin{align*}
\snorm{x'-z}_{2}\leq \snorm{x'-x}_{2}+ \snorm{x-z}_{2} < \om_{-}+ r = \om_{+}
\end{align*}
and hence, $x' \in  B_{\om_{+}}(z)$.
Next, we obtain with Lemma \ref{lemma:tech_int} that
\begin{align*}
K_{\g} \ast f_{\g}^{\om_{+}}(x)&=\left(\frac{2}{\pi^{1/2} \g}\right)^{d/2} \int_{\R^d} e^{-2\g^{-2} \snorm{x-y}^2_2}(\pi \g^2)^{-d/4} \cdot \eins_{B_{\om_{+}}(z)}  \dx{y} \\
			&=\left(\frac{2}{\pi \g^2}\right)^{d/2}  \int_{B_{\om_{+}}(z)} e^{-2\g^{-2} \snorm{x-y}^2_2}\dx{y}\\
			&\geq \left(\frac{2}{\pi \g^2}\right)^{d/2}  \int_{\mathring{B}_{\om_{-}}(x) } e^{-2\g^{-2} \snorm{x-y}^2_2}\dx{y}\\
			&=\frac{1}{\G(d/2)}  \int_0^{2(\om_{-})^2 \g^{-2}} e^{-t}t^{d/2-1}\dx{t}.
\end{align*}
Since $f^{\ast}_{\lclass,P}(x)=1$, we finally obtain with Lemma~\ref{lemma:norms} for $\r=\om_{+}$ and $\tilde{f}:=\eins_{X_1 \cup X_0}$ that 
\begin{align*}
\vert (K_{\g} \ast f_{\g}^{\om_{+}})(x) - f^{\ast}_{\lclass,P}(x)\vert &= \vert (K_{\g} \ast f_{\g}^{\om_{+}})(x) - 1 \vert\\
 								&=1- (K_{\g} \ast f_{\g}^{\om_{+}})(x)\\
								&\leq 1- \frac{1}{\G(d/2)}  \int_0^{2(\om_{-})^2 \g^{-2}} e^{-t}t^{d/2-1}\dx{t}\\
								&=\frac{1}{\G(d/2)}  \int_{2(\om_{-})^2 \g^{-2}}^{\infty} e^{-t}t^{d/2-1}\dx{t}.
\end{align*}
For $x \in A \cap X_{-1}$ the latter calculations yields with $\tilde{f}:=\eins_{X_{-1}}$  the same results and hence, the latter inequality holds for all $x \in A$.
\end{proof}

In the next theorem we state bounds on the approximation error over the sets $F$ and $N_2$. We obtain directly a bound for the set $F$, however, to obtain a bound on $N_2$ we need the additional assumption that $P$ has MNE $\b$.

\begin{theorem}[\textbf{Approximation Error on $F$ and $N_2$}]\label{theo:approx.out}
Let \textbf{(A)} and \textbf{(H)} be satisfied and define the set of indices
\begin{align*}
J:=\set{j\in \lbrace 1, \ldots m \rbrace}{ A_j \cap X_1 = \emptyset \, \,\text{or}\, \ A_j \cap X_{-1} = \emptyset}.
\end{align*}
For some $\om_->0$ define $\om_+:=\om_- +r >0$ and let the function $f_{\g_j}^{\om_{+}}$ for every $j \in J$ be defined as in \eqref{def:faltung_out}. Moreover, define the function $f_0: X \to \R$ by
\begin{align*}
f_0:= \bigcup_{j \in J} \eins_{A_j} \left( K_{\g_j} \ast f_{\g_j}^{\om_{+}}\right).
\end{align*}
Then, $f_0 \in H_J$ and $\inorm{f_0} \leq 1$. Furthermore,  for all $\xi>0$ there exist constants $c_d,c_{d,\xi}>0$ such that  
\begin{align*}
A_J^{(\bs\g)}(\bs\lb)
 \leq c_d \cdot \sum_{j \in J} \lb_j \left( \frac{\om_+ }{\g_j}\right)^{d}  +  c_{d,\xi} \left(\frac{\max_{j \in J} \g_j}{\om_{-}}\right)^{2\xi}   \sum_{j \in J}  P_X(A_j).
\end{align*}
In addition, if $P$ has MNE $\b \in (0,\infty]$ and we have $A_j \subset \lbrace \D_{\n}(x) \leq s \rbrace$ for every $j \in J$, then
\begin{align*}
A_J^{(\bs\g)}(\bs\lb) \leq c_d \cdot \sum_{j \in J} \lb_j \left( \frac{\om_+ }{\g_j}\right)^{d}  +  c_{d,\xi}  \left(\frac{\max_{j \in J} \g_j}{\om_{-}} \right)^{2\xi} \left(c_{\mathrm{MNE}} \cdot s \right)^{\b}
\end{align*}
\end{theorem}

\begin{proof}
By Lemma~\ref{lemma:norms} with $\r=\om_{+}$ and $\tilde{f}:=\eins_{X_1 \cup X_0}$ resp.\  $\tilde{f}:=\eins_{X_{-1}}$  we have immediately $f_0 \in H_J$ and $\inorm{f_0}=\inorm{K_{\g_j} \ast f_{\g_j}^{\om_{+}} } \leq 1$. Moreover, Lemma~\ref{lemma:norms} yields
\begin{align*}
 \sum_{j \in J}  \lb_j \snorm{\eins_{A_j} f_0}_{\hat{H}_j}^2 = \sum_{j \in J}  \lb_j \snorm{\eins_{A_j}(K_{\g_j} \ast f_{\g_j}^{\om_{+}})}_{\hat{H}_j}^2 \leq c_d \cdot \sum_{j \in J} \lb_j \left( \frac{\om_+ }{\g_j}\right)^{d} 
\end{align*}
for some constant $c_d>0$. Next, we bound the excess risk of $f_0$. 
We find by applying Zhang's equality (e.g., \citep[Theorem~2.31]{StCh08}), Lemma \ref{lemma:abs_out} and  \citep[Lemma~A.1.1]{StCh08} for some arbitrary $\xi>0$ that
\begin{align*}
\RP{L_J}{f_0}-\RPB{L_J}&=\sum_{j \in J} \int_{A_j}   |(K_{\g_j} \ast f_{\g_j}^{\om_{+}})- f^{\ast}_{L_{\text{class}},P}||2\n-1| \dx{P_X}\\
&\leq \sum_{j \in J} \int_{A_j}  \frac{1}{\G(d/2)}\int_{(\om_{-})^22\g_j^{-2}}^{\infty} e^{-t}t^{d/2-1}  \dx{t} |2\n-1|  \dx{P_X}\\
&\leq  \frac{1}{\G(d/2)}\int_{2(\om_{-})^2\g_{\max}^{-2}}^{\infty} e^{-t}t^{d/2-1}  \dx{t} \sum_{j \in J} \int_{A_j} |2\n-1| \dx{P_X}\\
&\leq \frac{\G\left(d/2,2(\om_{-})^2\g_{\max}^{-2}\right)}{\G(d/2)}\sum_{j \in J} \int_{A_j}  |2\n-1| \dx{P_X}\\
&\leq \frac{2^{-\xi}\G(d/2+\xi)}{\G(d/2)}\left(\frac{\g_{\max}}{\om_{-}}\right)^{2\xi}
\cdot \sum_{j \in J} P_X(A_j).
\end{align*}
If in addition $P$ has MNE $\b$ and  $A_j \subset \lbrace \D_{\n} \leq s \rbrace$ for every $j \in J$  we modify the previous 
calculation of the excess risk. Then, we obtain again with 
 Zhang's equality, Lemma \ref{lemma:abs_out} and  \citep[Lemma~A.1.1]{StCh08} for some arbitrary $\xi>0$ that
\begin{align*}
\RP{L_J}{f_0}-\RPB{L_J}
&\leq \sum_{j \in J} \int_{A_j}  \frac{1}{\G(d/2)}\int_{(\om_{-})^22\g_j^{-2}}^{\infty} e^{-t}t^{d/2-1}  \dx{t} |2\n-1|  \dx{P_X}\\
&\leq \frac{1}{\G(d/2)} \int_{(\om_{-})^22\g_{\max}^{-2}}^{\infty} e^{-t}t^{d/2-1}  \dx{t} \int_{\lbrace \D_{\n} \leq s \rbrace} |2\n-1| \dx{P_X}\\
&\leq \frac{c_{\text{MNE}}^{\b} \G(d/2,(\om_{-})^2 2 \g_{\max}^{-2})}{\G(d/2)}  \cdot s^{\b}\\
&\leq  \frac{c_{\text{MNE}}^{\b}2^{-\xi}\G(d/2+\xi)}{\G(d/2)}\left(\frac{\g_{\max}}{\om_{-}} \right)^{2\xi} s^{\b}.
\end{align*}
By combining the results for the norm and the excess risk yields finally the bounds on the respective approximation error.
\end{proof}

Both bounds in the theorem above depend on the parameter $\xi>0$. However, we will observe in the theorems in Section~\ref{sec:oracratesN2F}, which state the corresponding oracle inequalities, that by setting $\om_{-}$ appropriately this $\xi$ will not have an influence any more.

\subsubsection{Oracle inequalities and learning rates on $N_1$} \label{sec:oracratesN1}

Based on the the general oracle inequality in Section~\ref{theo:oraclemain} and the results from the previous section we establish in this section an oracle inequality on  the set $N_1$ and derive learning rates.

\begin{theorem}[\textbf{Oracle Inequality on $N_1$}]\label{theo:oracle_in}
Let $P$ have MNE $\b \in (0,\infty]$ and NE $q \in [0,\infty]$ and let \textbf{(G)} and \textbf{(H)} be satisfied. Moreover, let \textbf{(A)} be satisfied for some $r:=n^{-\nu}$ with $\nu>0$ and define the set of indices 
\begin{align*}
J:=\set{j \in \lbrace 1,\ldots,m \rbrace}{\forall x \in A_j: \,P_X(A_j \cap X_1)>0\, \,\text{and}\, \, P_X(A_j \cap X_{-1})>0 }.
\end{align*}
Let $\t\geq 1$ be fixed and define $n^{\ast}:=\left(\frac{4}{\d^{\ast}}\right)^{\frac{1}{\nu}}$. Then, for all $p\in (0,\frac{1}{2})$, $n \geq  n^{\ast}$, $\bs\lb:=(\lb_1,\ldots, \lb_m) \in (0,\infty)^m$ and $\bs\g:=(\g_1,\ldots,\g_m) \in (0,r]^m $ the SVM given in \eqref{locSVMpredictor} satisfies
\begin{align}\label{oracle_innen}
\begin{split}
  &\RP{L_J}{\cl{f}_{D,\bs\lb,\bs\g}}-\RPB{L_J}\\
  &\leq 9 c_{d,\b} \left( \sum_{j \in J} \frac{\lb_j r^d}{\g_j^d}  +  \max_{j \in J} \g_j^{\b}\right)
  +c_{d,p,q}  \left(\frac{r}{n}\right)^{\frac{q+1}{q+2-p}} \left( \sum_{j \in J} \lb_j^{-1} \g_j^{-\frac{d}{p}}  P_X(A_j) \right)^{\frac{p(q+1)}{q+2-p}} + \tilde{c}_{p,q}\left(\frac{\t}{n}\right)^{\frac{q+1}{q+2}}
  \end{split}
\end{align}
with probability $P^n$ not less than $1-3e^{-\t}$ and for some constants $c_{d,\b}, c_{d,p,q}>0$ and  $\tilde{c}_{p,q}>0$.
\end{theorem}

\begin{proof}
 We apply the generic oracle inequality given in Theorem~\ref{theo:oraclemain} and bound  first of all the contained constant $a^{2p}$ . To this end, we remark that an analogous calculation as in the proof of Theorem~\ref{theo:approx.in} shows that $A_j \subset \lbrace \D_{\n}\leq 2r \rbrace$ for every $j \in J$. Since
\begin{align*}
n \geq \left(\frac{4}{\d^{\ast}}\right)^{\frac{1}{\nu}}   \Leftrightarrow 4r \leq \d^{\ast}
\end{align*}
 we obtain by Lemma~\ref{lemm:numbercell} for $s=2r$ that
\begin{align}\label{macht_innen}
|J| \leq c_1 r^{-d+1},
\end{align} 
where $c_1$ is a positive constant only depending on $d$. Together with 
 Lemma~\ref{lemma_entropy_set_u_average} we then find that
\begin{align*}
a^{2p}&=\max \left\lbrace \tilde{c}_{d,p}|J|^{\frac{1}{2p}} r^{\frac{d}{2p}}  \left( \sum_{j \in J} \lb_j^{-1}\g_j^{-\frac{d}{p}}   P_X(A_j) \right)^{\frac{1}{2}},2 \right\rbrace^{2p}\\
&\leq \tilde{c}_{d,p}^{2p}|J|r^d   \left( \sum_{j \in J} \lb_j^{-1} \g_j^{-\frac{d}{p}}  P_X(A_j) \right)^p+4^p\\
&\leq c_1\tilde{c}_{d,p}^{2p} \cdot r  \left( \sum_{j \in J} \lb_j^{-1} \g_j^{-\frac{d}{p}}  P_X(A_j) \right)^p+4^p,
\end{align*}
where $c_{d,p}:=2c_1 \left(9 \ln(4)c_d\right)^{\frac{1}{2p}}\left( \frac{d+1}{2ep}\right)^{\frac{d+1}{2p}}$.  Moreover,  \cite[Lemma~8.24]{StCh08} delivers a variance bound for $\th=\frac{q}{q+1}$ and constant $V:=6c_{NE}^{\frac{q}{q+1}}$. We denote by $A_J^{(\bs\g)}(\bs\lb)$ the approximation error, defined in \eqref{def.AEF}. Then, we have  by Theorem~\ref{theo:oraclemain} with $\t \geq 1$ that
\begin{align*}
\begin{split}
  &\RP{L_J}{\cl{f}_{D,\bs\lb,\bs\g}}-\RPB{L_J} \\ 
  &\leq  9A_J^{(\bs\g)}(\bs\lb) 
  +c_{p,q} \left(\frac{a^{2p}}{n}
  \right)^{\frac{q+1}{q+2-p}} 
  +3c_{NE}^{\frac{q}{q+2}} \left(\frac{432\t}{n}\right)^{\frac{q+1}{q+2}}  + \frac{30\t}{n}\\
  &\leq 9A_J^{(\bs\g)}(\bs\lb) 
  +c_{p,q} \left[  c_1\tilde{c}_{d,p}^{2p} \cdot r  \left( \sum_{j \in J} \lb_j^{-1} \g_j^{-\frac{d}{p}}  P_X(A_j) \right)^p+4^p \right]^{\frac{q+1}{q+2-p}}\cdot n^{-\frac{q+1}{q+2-p}} 
  +c_q \left(\frac{\t}{n}\right)^{\frac{q+1}{q+2}}\\
  &\leq 9A_J^{(\bs\g)}(\bs\lb) 
  +c_{d,p,q}  \left(\frac{r}{n}\right)^{\frac{q+1}{q+2-p}} \left(
   \sum_{j \in J} \lb_j^{-1} \g_j^{-\frac{d}{p}}  P_X(A_j) \right)^{\frac{p(q+1)}{q+2-p}} + c_{p,q}4^{\frac{p(q+1)}{q+2-p}}n^{-\frac{q+1}{q+2}}
  +c_q \left(\frac{\t}{n}\right)^{\frac{q+1}{q+2}}\\
   &\leq 9A_J^{(\bs\g)}(\bs\lb) 
  +c_{d,p,q} \left(\frac{r}{n}\right)^{\frac{q+1}{q+2-p}} \left( \sum_{j \in J} \lb_j^{-1} \g_j^{-\frac{d}{p}}  P_X(A_j) \right)^{\frac{p(q+1)}{q+2-p}} + \tilde{c}_{p,q}\left(\frac{\t}{n}\right)^{\frac{q+1}{q+2}}
\end{split}
\end{align*}
holds with probability $P^n$ not less than $1-3e^{-\t}$ and with positive constants $c_{d,p,q}:=c_{p,q} \left(c_1\tilde{c}_{d,p}^{2p}\right)^{\frac{q+1}{q+2-p}}$, $c_q:=2\max \left\lbrace 3c_{NE}^{\frac{q}{q+2}}432^{\frac{q+1}{q+2}}, 30\right\rbrace$ and $\tilde{c}_{p,q}:=2 \max\left\lbrace c_{p,q} 4^{\frac{p(q+1)}{q+2-p}}, c_q \right\rbrace$. 
Finally, Theorem~\ref{theo:approx.in} yields for the approximation error the bound
\begin{align*}
A_J^{(\bs\g)}(\bs\lb) \leq c_2 \left( \sum_{j \in J} \frac{\lb_j r^d}{\g_j^d}  +  \max_{j\in J} \g_j^{\b}\right),
\end{align*}
where $c_2>0$ is a constant depending on $d$ and  $\b$. By plugging this into the oracle inequality above yields the result.
\end{proof}


\begin{theorem}[\textbf{Learning Rates on} $N_1$]\label{theo:rate_in}
Let the assumptions of Theorem~\ref{theo:oracle_in} be satisfied with $m_n$ and with 
\begin{align}\label{def_rglb_innen}
\begin{split}
r_n&\simeq n^{-\nu},\\
\g_{n,j} &\simeq r_n^{\k}n^{-\k},\\
\lb_{n,j} &\simeq n^{-\s},
\end{split}
\end{align}
for all $j\in  \lbrace 1,\ldots, m_n\rbrace$. Moreover, define  $\k:=\frac{q+1}{\b(q+2)+d(q+1)}$ and let
\begin{align}\label{ineq_alpha}
\nu \leq \frac{\k}{1-\k}
\end{align}
and $\s \geq 1$
be satisfied. Then, for all $\e >0$ there exists a constant $c_{\b,d,\e,q}>0$ such that for $\bs\lb_n:=(\lb_{n,1},\ldots, \lb_{n,m_n})\in (0,\infty)^m$, and $\bs\g_n:=(\g_{n,1},\ldots,\g_{n,m_n}) \in (0,r_n]^{m_n} $, and all n sufficiently large we have  with probability $P^n$ not less than $1-3e^{-\t}$ that 
\begin{align*}
 \RP{L_J}{\cl{f}_{D,\bs\lb_n,\bs\g_n}}-\RPB{L_J}\leq c_{\b,d,\e,q} \cdot \t^{\frac{q+1}{q+2}} \cdot r_n^{\b \k}n^{- \b \k+\e}.
\end{align*}
In particular, the proof shows that one can even choose $\s \geq \k(\b+d)(\nu+1)-\nu>0$.
\end{theorem}

\begin{proof}
 We write $\lb_n := n^{-\s}$ and $\g_n := r_n^{\k}n^{-\k}$. As in the proof of Theorem~\ref{theo:oracle_in} we find 
\begin{align*}
|J| \leq c_d r_n^{-d+1}
\end{align*} 
for some constant $c_d>0$. Together with  Theorem~\ref{theo:oracle_in} we then obtain that 
\begin{align*}
\begin{split}
 &\RP{L_J}{\cl{f}_{D,\bs\lb_n,\bs\g_n}}-\RPB{L_J}\\
 &\leq c_1\left( \sum_{j \in J} \frac{\lb_{n,j} r^d}{\g_{n,j}^d} + \max_{j \in J} \g_{n,j}^{\b}
  + \left(\frac{r}{n}\right)^{\frac{q+1}{q+2-p}}  \left( \sum_{j \in J} \lb_{n,j}^{-1} \g_{n,j}^{-\frac{d}{p}}  P_X(A_j) \right)^{\frac{p(q+1)}{q+2-p}} +\left(\frac{\t}{n}\right)^{\frac{q+1}{q+2}}\right)\\
  &\leq c_2\left( |J|\frac{\lb_n r_n^d}{\g_n^d}+ \g_n^{\b} 
  + \left(\frac{r_n}{n}\right)^{\frac{q+1}{q+2-p}}  \left( \lb_n^{-1}\g_n^{-\frac{d}{p}} \sum_{j \in J} P_X(A_j) \right)^{\frac{p(q+1)}{q+2-p}} +\left(\frac{\t}{n}\right)^{\frac{q+1}{q+2}}\right)\\
   &\leq c_2\left( \frac{\lb_n r_n}{\g_n^d}+ \g_n^{\b} 
  + \left(\frac{r_n \lb_n^{-p}\g_n^{-d}}{n}\right)^{\frac{q+1}{q+2-p}} +\left(\frac{\t}{n}\right)^{\frac{q+1}{q+2}}\right)
  \end{split}
  \end{align*}
holds with probability $P^n$ not less than $1-3e^{-\t}$ and for some positive constant $c_1,c_2$ depending on $\b,d,p$ and $q$. Moreover, with \eqref{def_rglb_innen},  $\s \geq \k(\b+d)(\nu+1)-\nu$ and $\frac{(1-d\k)(q+1)}{q+2}=\left(\frac{\b(q+2)}{\b(q+2)+d(q+1)} \right)\frac{q+1}{q+2} =\b \k$  we find
\begin{align*}
 &\RP{L_J}{\cl{f}_{D,\bs\lb_n,\bs\g_n}}-\RPB{L_J}\\
  &\leq c_2\left( \frac{\lb_n r_n}{\g_n^d}+ \g_n^{\b} 
  + \left(\frac{r_n \lb_n^{-p}\g_n^{-d}}{n}\right)^{\frac{q+1}{q+2-p}} +\left(\frac{\t}{n}\right)^{\frac{q+1}{q+2}}\right)\\
  &= c_2 \left( \frac{ r_n}{n^{\k(\b+d)(\nu+1)-\nu} \g_n^d} + r_n^{\b \k}n^{- \b \k}
  + \left(\frac{r_n^{1-d \k}}{n^{1-d\k}} \right)^{\frac{q+1}{q+2-p}}\left(n^{-\s}\right)^{\frac{p(q+1)}{q+2-p}} +\left(\frac{\t}{n}\right)^{\frac{q+1}{q+2}}\right)\\ 
   &\leq c_3 \left( \frac{n^{-\nu}}{n^{\k(\b+d)(\nu+1)-\nu} n^{-\nu d\k}n^{-d\k}} + r_n^{\b \k}n^{- \b \k}
  + \left(\frac{r_n}{n} \right)^{\frac{(1-d \k)(q+1)}{q+2-p}}n^{\e} +\left(\frac{\t}{n}\right)^{\frac{q+1}{q+2}}\right)\\ 
    &\leq c_3 \left( n^{-\nu \b \k}n^{- \b \k} + r_n^{\b \k}n^{- \b \k}
  + \left(\frac{r_n}{n} \right)^{\frac{(1-d \k)(q+1)}{q+2}}n^{\e} +\left(\frac{\t}{n}\right)^{\frac{q+1}{q+2}}\right)\\   
  &\leq c_4\left( r_n^{\b \k}n^{- \b \k +\e} +\t^{\frac{q+1}{q+2}}n^{-\frac{q+1}{q+2}}\right)\\
  &\leq c_5\t^{\frac{q+1}{q+2}} \cdot  r_n^{\b \k}n^{- \b \k +\e},
\end{align*}
where $\e$ is chosen sufficiently small such that $\e \geq \frac{p\s (q+1)}{q+2} \geq 0$ and where the constants $c_3,c_4,c_5>0$ depend on  $\b,d,\e$ and $q$.
\end{proof}

\subsubsection{Oracle inequalities and learning rates on $N_2,F$}\label{sec:oracratesN2F}

Based on the the general oracle inequality in Section~\ref{theo:oraclemain} and the results from the previous section we establish in this section an oracle inequality on  the set $N_2$ and $F$. Moreover, we derive learning rates.

\begin{theorem}[\textbf{Oracle inequality} on $N_2$]\label{theo:oracle_inout}
Let $P$ have MNE $\b \in (0,\infty]$ and NE $q \in [0,\infty]$ and let \textbf{(G)}   and \textbf{(H)} be satisfied. Moreover, let \textbf{(A)} be satisfied for some $r:=n^{-\nu}$ with $\nu>0$. Define for $s:=n^{-\a}$ with $\a>0$ and $\a\leq \nu$ the set of indices 
\begin{align*}
J:=\set{j \in \lbrace 1,\ldots,m \rbrace}{\forall x \in A_j: \D_{\n}(x) \leq 3s\, \, \text{and}\, \,P_X(A_j \cap X_1)=0 \, \, \text{or}\, \, P_X(A_j \cap X_{-1})=0}.
\end{align*}
Let $\t\geq 1$ be fixed and define  $n^{\ast}:=\left(\d^{\ast}\right)^{-\frac{1}{\nu+\a}}$.  Then, for all $\e>0$, $p\in (0,\frac{1}{2})$, $n \geq n^{\ast}$, $\bs\lb:=(\lb_1,\ldots, \lb_m)\in (0,\infty)^m$, and $\bs\g:=(\g_1,\ldots,\g_m) \in (0,r]^m $ the SVM given in \eqref{locSVMpredictor} satisfies
\begin{align}\label{oracle_mitte}
\begin{split}
  &\RP{L_J}{\cl{f}_{D,\bs\lb,\bs\g}}-\RPB{L_J} \\ 
  &\leq \left(  \frac{c_{d,\b,\e  } \cdot r }{\min_{j\in J} \g_j}\right)^d    n^{\e}\sum_{j \in J} \lb_j +c_{d,p,q} \left(\frac{s}{n}\right)^{\frac{q+1}{q+2-p}} \left( \sum_{j \in J} \lb_j^{-1} \g_j^{-\frac{d}{p}}  P_X(A_j) \right)^{\frac{p(q+1)}{q+2-p}} + c_{d,\b,\e,p,q}\left(\frac{\t}{n}\right)^{\frac{q+1}{q+2}}
\end{split}
\end{align}
with probability $P^n$ not less than $1-3e^{-\t}$ and with constants $c_{d,\b,\e} ,c_{d,p,q}>0$ and $c_{d,\b,\e,p,q}>0$.
\end{theorem}

\begin{proof}
 We apply the generic oracle inequality given in Theorem~\ref{theo:oraclemain} and bound  first of all the contained constant $a^{2p}$ . To this end, we remark that
\begin{align*}
n \geq \left(\d^{\ast}\right)^{-\frac{1}{\nu+\a}}   \Leftrightarrow 4r \leq \d^{\ast}
\end{align*}
such that we obtain by Lemma \ref{lemm:numbercell} that
\begin{align}\label{macht_mitte}
|J| \leq c_1 \cdot s r^{-d},
\end{align} 
where $c_1$ is a positive constant only depending on $d$.  Together with 
 Lemma~\ref{lemma_entropy_set_u_average} we then find for the constant $a^{2p}$ from Theorem~\ref{theo:oraclemain} that
\begin{align*}
a^{2p}&=\max \left\lbrace \tilde{c}_{d,p}|J|^{\frac{1}{2p}} r^{\frac{d}{2p}}  \left( \sum_{j \in J} \lb_j^{-1}\g_j^{-\frac{d}{p}}   P_X(A_j) \right)^{\frac{1}{2}},2 \right\rbrace^{2p}\\
&\leq \tilde{c}_{d,p}^{2p} \cdot |J|r^d   \left( \sum_{j \in J} \lb_j^{-1} \g_j^{-\frac{d}{p}}  P_X(A_j) \right)^p+4^p\\
&\leq c_1 \tilde{c}_{d,p}^{2p} \cdot s  \left( \sum_{j \in J} \lb_j^{-1} \g_j^{-\frac{d}{p}}  P_X(A_j) \right)^p+4^p,
\end{align*}
where $c_{d,p}:=2c_1 \left(9 \ln(4)c_d\right)^{\frac{1}{2p}}\left( \frac{d+1}{2ep}\right)^{\frac{d+1}{2p}}$. Again, \cite[Lemma~8.24]{StCh08} yields a variance bound for $\th=\frac{q}{q+1}$ and constant $V:=6c_{NE}^{\frac{q}{q+1}}$. We denote by $A_J^{(\bs\g)}(\bs\lb)$ the approximation error, defined in \eqref{def.AEF}, and find by Theorem~\ref{theo:oraclemain} with $\t \geq 1$ that
\begin{align}\label{ineq:N2risk}
\begin{split}
  &\RP{L_J}{\cl{f}_{D,\bs\lb,\bs\g}}-\RPB{L_J} \\ 
  &\leq  9A_J^{(\bs\g)}(\bs\lb) 
  +c_{p,q} \left(\frac{a^{2p}}{n}
  \right)^{\frac{q+1}{q+2-p}} 
  +3c_{NE}^{\frac{q}{q+2}} \left(\frac{432\t}{n}\right)^{\frac{q+1}{q+2}}  + \frac{30\t}{n}\\
  &\leq 9A_J^{(\bs\g)}(\bs\lb) 
  +c_{p,q} \left[  c_1\tilde{c}_{d,p}^{2p}  \cdot s \left( \sum_{j \in J} \lb_j^{-1} \g_j^{-\frac{d}{p}}  P_X(A_j) \right)^p+4^p \right]^{\frac{q+1}{q+2-p}} n^{-\frac{q+1}{q+2-p}} 
  +c_q \left(\frac{\t}{n}\right)^{\frac{q+1}{q+2}}\\
  &\leq 9A_J^{(\bs\g)}(\bs\lb) 
  +c_{d,p,q}  \left(\frac{s}{n}\right)^{\frac{q+1}{q+2-p}} \left(
   \sum_{j \in J} \lb_j^{-1} \g_j^{-\frac{d}{p}}  P_X(A_j) \right)^{\frac{p(q+1)}{q+2-p}} + c_{p,q} 4^{\frac{p(q+1)}{q+2-p}} \cdot n^{-\frac{q+1}{q+2}}
  +c_q \left(\frac{\t}{n}\right)^{\frac{q+1}{q+2}}\\
   &\leq 9A_J^{(\bs\g)}(\bs\lb) 
  +c_{d,p,q} \left(\frac{s}{n}\right)^{\frac{q+1}{q+2-p}} \left( \sum_{j \in J} \lb_j^{-1} \g_j^{-\frac{d}{p}}  P_X(A_j) \right)^{\frac{p(q+1)}{q+2-p}} + \tilde{c}_{p,q}\left(\frac{\t}{n}\right)^{\frac{q+1}{q+2}}
\end{split}
\end{align}
holds with probability $P^n$ not less than $1-3e^{-\t}$ and with positive constants $c_{d,p,q}:=c_{p,q} \left(c_1\tilde{c}_{d,p}^{2p}\right)^{\frac{q+1}{q+2-p}}$, $c_q:=2\max \left\lbrace 3c_{NE}^{\frac{q}{q+2}}432^{\frac{q+1}{q+2}}, 30\right\rbrace$ and $\tilde{c}_{p,q}:=2 \max\left\lbrace c_{p,q} 4^{\frac{p(q+1)}{q+2-p}}, c_q \right\rbrace$. 
 Finally, Theorem~\ref{theo:approx.out} for $\om_- := \g_{\max} n^{\frac{q+1}{2\xi(q+2)}}$, where $\xi>0$, and $\om_{+}:=\om_{-}+r$, yields
\begin{align}\label{def_AN2}
\begin{split}
A_J^{(\bs\g)}(\bs\lb)&\leq  c_2 \left( \sum_{j \in J} \lb_j  \left( \frac{\om_+ }{\g_j}\right)^{d} +   \left(\frac{\g_{\max}}{\om_-} \right)^{2\xi} s^{\b} \right)\\
&\leq  c_2 \left( \left( \frac{\om_+ }{\g_{\min}}\right)^{d} \sum_{j \in J} \lb_j  +   \left(\frac{\g_{\max}}{\om_-} \right)^{2\xi} s^{\b} \right)\\
&=   c_2 \left(  \left( \frac{\g_{\max} n^{\frac{q+1}{2\xi(q+2)}} +r }{\g_{\min}}\right)^{d} \sum_{j \in J} \lb_j +  n^{-\frac{q+1}{q+2}} s^{\b} \right) \\
&\leq   c_3 \left(  n^{\frac{d(q+1)}{2\xi(q+2)}}\left( \frac{\g_{\max}  +r }{\g_{\min}}\right)^{d}\sum_{j \in J} \lb_j  +  n^{-\frac{q+1}{q+2}} s^{\b}   \right) \\
&\leq   c_4 \left( n^{\e} \left( \frac{r}{\g_{\min}}\right)^{d}  \sum_{j \in J} \lb_j +  n^{-\frac{q+1}{q+2}} s^{\b}   \right),
\end{split}
\end{align}
where in the last step that we applied $\g_{\max} \leq r$, and where we picked an arbitrary $\e>0$ and chose $\xi$ sufficiently large such that  $\e \geq  \frac{d(q+1)}{2\xi(q+2)}>0$. The constants $c_2,c_3>0$ only depend  on $d,\b$ and $\xi$, whereas $c_4>0$ depends only on $d,\b$  and $\e$.
By plugging this into the oracle inequality above yields 
\begin{align*}
  &\RP{L_J}{\cl{f}_{D,\bs\lb,\bs\g}}-\RPB{L_J} \\ 
   &\leq 9 c_4 \left( n^{\e} \left( \frac{r}{\g_{\min}}\right)^{d}  \sum_{j \in J} \lb_j +  n^{-\frac{q+1}{q+2}} s^{\b}   \right)\\
  &\qquad+c_{d,p,q} \left(\frac{s}{n}\right)^{\frac{q+1}{q+2-p}} \left( \sum_{j \in J} \lb_j^{-1} \g_j^{-\frac{d}{p}}  P_X(A_j) \right)^{\frac{p(q+1)}{q+2-p}} + \tilde{c}_{p,q}\left(\frac{\t}{n}\right)^{\frac{q+1}{q+2}}\\
  &\leq 9 c_4 n^{\e} \left( \frac{r}{\g_{\min}}\right)^{d}  \sum_{j \in J} \lb_j +c_{d,p,q} \left(\frac{s}{n}\right)^{\frac{q+1}{q+2-p}} \left( \sum_{j \in J} \lb_j^{-1} \g_j^{-\frac{d}{p}}  P_X(A_j) \right)^{\frac{p(q+1)}{q+2-p}} + c_{d,\b,\e,p,q}\left(\frac{\t}{n}\right)^{\frac{q+1}{q+2}}.
\end{align*}

\end{proof}

\begin{theorem}[\textbf{Learning Rates on} $N_2$]\label{theo:rate_inout}
Let the assumption of Theorem~\ref{theo:oracle_inout} be satisfied for $m_n$, 
 $s \simeq s_n$ and 
\begin{align}\label{def_rglb_innen_s}
\begin{split}
r_n&\simeq n^{-\nu},\\
\g_{n,j} &\simeq r_n ,\\
\lb_{n,j} &\simeq n^{-\s}
\end{split}
\end{align}
with some $\s \geq 1$ and  $1+\a-\nu d>0$, and for all $j\in  \lbrace 1,\ldots, m_n\rbrace$. 
Then, for all $\e >0$ there exists a constant $c_{\b,d,\e,q}>0$ such that for $\bs\lb_n:=(\lb_{n,1},\ldots, \lb_{n,m_n})\in (0,\infty)^m$, and $\bs\g_n:=(\g_{n,1},\ldots,\g_{n,m_n}) \in (0,r_n]^{m_n} $, and all n sufficiently large we have  with probability $P^n$ not less than $1-3e^{-\t}$ that 
\begin{align*}
 \RP{L_J}{\cl{f}_{D,\bs\lb_n,\bs\g_n}}-\RPB{L_J} &\leq c_{\b,d,\e,q} \t^{\frac{q+1}{q+2}} \cdot   \left(\frac{s_n}{r_n^d}\right)^{\frac{q+1}{q+2}} n^{-\frac{q+1}{q+2}+\e}.
\end{align*}
\end{theorem}

\begin{proof}
We write $\lb_n:=n^{-\s}$ and $\g_n:=r_n$.
By Theorem~\ref{theo:oracle_inout}, Lemma~\ref{lemm:numbercell} and   \eqref{def_rglb_innen_s} we find with probability $P^n$ not less than $1-3e^{-\t}$ that
\begin{align*}
&\RP{L_J}{\cl{f}_{D,\bs\lb_n,\bs\g_n}}-\RPB{L_J}\\
  &\leq c_1 \left( \left( \frac{r}{\g_{\text{min}}}\right)^{d}  \sum_{j \in J} \lb_{n,j} n^{\e}+\left(\frac{s_n}{n}\right)^{\frac{q+1}{q+2-p}} \left( \sum_{j \in J} \lb_{n,j}^{-1} \g_{n,j}^{-\frac{d}{p}}  P_X(A_j) \right)^{\frac{p(q+1)}{q+2-p}} +\left(\frac{\t}{n}\right)^{\frac{q+1}{q+2}} \right)\\
  &\leq c_2 \left(|J| \lb_n  n^{\e}+\left(\frac{s_n}{\g_n^d n}\right)^{\frac{q+1}{q+2-p}} \left(  \lb_n^{-1} \sum_{j \in J} P_X(A_j) \right)^{\frac{p(q+1)}{q+2-p}} +\left(\frac{\t}{n}\right)^{\frac{q+1}{q+2}}\right)\\
  &\leq c_2 \left(\frac{s_n \lb_n  n^{\e}}{r_n^d} +\left(\frac{s_n}{r_n^d n}\right)^{\frac{q+1}{q+2-p}}\lb_n^{-\frac{p(q+1)}{q+2-p}} +\left(\frac{\t}{n}\right)^{\frac{q+1}{q+2}}\right)\\
    &\leq c_2 \left(\frac{s_n n^{\e}}{r_n^d n^{\s}} +\left(\frac{s_n}{r_n^d n}\right)^{\frac{q+1}{q+2}}n^{\frac{p\s (q+1)}{q+2-p}} +\left(\frac{\t}{n}\right)^{\frac{q+1}{q+2}}\right)\\
    &\leq c_3 \left(\frac{s_n  n^{\e}}{r_n^d n} +\left(\frac{s_n}{r_n^d n}\right)^{\frac{q+1}{q+2}}n^{\hat{\e}} +\left(\frac{\t}{n}\right)^{\frac{q+1}{q+2}}\right)\\
    &\leq c_4 \t^{\frac{q+1}{q+2}} n^{\e}  \left(\left(\frac{s_n}{r_n^d n}\right)^{\frac{q+1}{q+2}}+n^{-\frac{q+1}{q+2}}\right)\\
    &\leq c_5 \t^{\frac{q+1}{q+2}}  \left(\frac{s_n}{r_n^d}\right)^{\frac{q+1}{q+2}}n^{-\frac{q+1}{q+2}+\e},
\end{align*}
where we chose $p$ sufficiently small such that   that $\e \geq \frac{p\s (q+1)}{q+2-p}>0$. The constants $c_1,c_2,c_3>0$ depend only on $d,\b,\e,p$ and $q$, whereas the constants $c_4,c_5>0$ depend on $d,\b,\e$ and $q$. 
\end{proof}

\begin{theorem}[\textbf{Oracle Inequality on} $F$]\label{theo:oracle_out}
Let $P$ have LC $\z \in [0,\infty)$ and NE $q \in [0,\infty]$ and let \textbf{(G)} and \textbf{(H)} be satisfied. Moreover, let \textbf{(A)} be satisfied for some $r:=n^{-\nu}$ with $\nu>0$. Define for $s := n^{-\a}$ with $\a>0$ and $\a\leq \nu$ the set of indices 
\begin{align*}
J:=\set{j \in \lbrace 1,\ldots,m \rbrace}{\forall x \in A_j: \D_{\n}(x) \geq s}.
\end{align*}
 Furthermore, let $\t\geq 1$ be fixed.
Then, for all $\e>0$, $p\in (0,\frac{1}{2})$, $n \geq 1$, $\bs\lb:=(\lb_1,\ldots, \lb_m)\in (0,\infty)^m$, and $\bs\g:=(\g_1,\ldots,\g_m) \in (0,r]^m $ the SVM given in \eqref{locSVMpredictor} satisfies
\begin{align}\label{oracle_aussen}
\begin{split}
  &\RP{L_J}{\cl{f}_{D,\bs\lb,\bs\g}}-\RPB{L_J} \\ 
&\leq \left(\frac{ c_{d,\e} \cdot r}{\min_{j\in J}\g_j}\right)^{d}n^{\e}  \sum_{j \in J} \lb_j   +  c_{d,p,q} \left( \sum_{j \in J} \lb_j^{-1} \g_j^{-\frac{d}{p}}  P_X(A_j) \right)^p n^{-1}+ c_{d,\e,p,q} \cdot \frac{\t}{s^{\z}n} 
 \end{split}
\end{align}
 with probability $P^n$ not less than $1-3e^{-\t}$ and some constants $c_{d,\e},c_{p,q,q},c_{d,\e,p,q}>0$.
\end{theorem}

\begin{proof}
We apply the generic oracle inequality given in Theorem~\ref{theo:oraclemain}. To this end, we find for the contained constant $a^{2p}$ with Lemma~\ref{lemma:entropy} and  \eqref{ex. Ueberdeckung} that
\begin{align}\label{ineq_a2p_far}
\begin{split}
a^{2p}&=\max \left\lbrace \tilde{c}_{d,p}|J|^{\frac{1}{2p}} r^{\frac{d}{2p}}  \left( \sum_{j \in J} \lb_j^{-1}\g_j^{-\frac{d}{p}}   P_X(A_j) \right)^{\frac{1}{2}},2 \right\rbrace^{2p}\\
&\leq \tilde{c}_{d,p}^{2p}|J|r^d   \left( \sum_{j \in J} \lb_j^{-1} \g_j^{-\frac{d}{p}}  P_X(A_j) \right)^p+4^p\\
&\leq c_1\tilde{c}_{d,p}^{2p}  \left( \sum_{j \in J} \lb_j^{-1} \g_j^{-\frac{d}{p}}  P_X(A_j) \right)^p+4^p
 \end{split}
\end{align}
where $c_1>0$ is a constant depending on $d$. According to Lemma~\ref{lem:varbound} we have variance bound $\th=1$  and constant $V:=2c_{\mathrm{LC}}s^{-\z}$. We denote by $A^{(\bs\g)}(\bs\lb)$ the approximation error, defined in \eqref{def.AEF}, and obtain by Theorem \ref{theo:oraclemain} together with  \eqref{ineq_a2p_far} with probability $P^n$ not less than $1-3e^{-\t}$ that
\begin{align}\label{risk_F}
\begin{split}
 &\RP{L_J}{\cl{f}_{D,\bs\lb_n,\bs\g_n}}-\RPB{L_J}\\
  &\leq  9A_J^{(\bs\g)}(\bs\lb)  +  \frac{c_{p,q} \cdot  a^{2p}}{n} + \frac{432c_{LC}  \t}{s^{\z}n}  + \frac{30\t}{n} \\
    &\leq  9 A_J^{(\bs\g)}(\bs\lb) +  c_{p,q} c_1\tilde{c}_{d,p}^{2p}  \left( \sum_{j \in J} \lb_j^{-1} \g_j^{-\frac{d}{p}}  P_X(A_j) \right)^p n^{-1}+ \frac{ c_{p,q}  4^p}{n} + \frac{432c_{LC}  \t}{s^{\z}n}  + \frac{30\t}{n} \\
    &\leq  9 A_J^{(\bs\g)}(\bs\lb)  +  c_{d,p,q} \left( \sum_{j \in J} \lb_j^{-1} \g_j^{-\frac{d}{p}}  P_X(A_j) \right)^p n^{-1}+ \tilde{c}_{d,p,q}\frac{\t}{s^{\z}n}
  \end{split}
 \end{align}
for some constants $c_{d,p,q},\tilde{c}_{d,p,q}>0$.  
 For the approximation error Theorem~\ref{theo:approx.out} with $\om_{-} := \g_{\max} n^{\frac{1}{2\xi}}$, where $\xi>0$, and $\om_{+}:=\om_{-}+r$, yields
\begin{align*}
A_J^{(\bs\g)}(\bs\lb)&\leq c_2\left( \cdot \sum_{j \in J} \lb_j  \left( \frac{\om_+}{\g_j}\right)^{d} + c_4 \cdot \left( \frac{\g_{\max}}{\om_-}\right)^{2\xi} P_X(F) \right)\\
&\leq c_2 \left( \left( \frac{\om_+}{\g_{\min}}\right)^{d}  \sum_{j \in J} \lb_j  + c_4 \cdot \left( \frac{\g_{\max}}{\om_-}\right)^{2\xi} \right)\\
 &=c_2 \left( \left( \frac{\g_{\max} n^{\frac{1}{2\xi}}}{\g_{\min}}+\frac{r}{\g_{\min}}\right)^{d}  \sum_{j \in J} \lb_j  + n^{-1} \right) \\
  &=c_3 \left(n^{\e} \left( \frac{r}{\g_{\min}}\right)^{d}  \sum_{j \in J} \lb_j  + n^{-1} \right),
\end{align*}
where we applied in the last step that $\g_{\max} \leq r$ and where we fixed an $\e$ and chose $\xi$ sufficiently large such that  $\e \geq \frac{d}{2\xi}>0$. The constants $c_2>0$ and $c_3>0$ only depend on $d,\xi$ resp.\ $d,\e$.
By combining the results above we have
\begin{align*}
&\RP{L_J}{f_{D,\bs\lb_n,\bs\g_n}}-\RPB{L_J}\\
    &\leq 9 c_3 \left( n^{\e}\left( \frac{r}{\g_{\min}}\right)^{d}  \sum_{j \in J} \lb_j  + n^{-1} \right)  +  c_{d,p,q} \left( \sum_{j \in J} \lb_j^{-1} \g_j^{-\frac{d}{p}}  P_X(A_j) \right)^p n^{-1}+ \tilde{c}_{d,p,q}\frac{\t}{s^{\z}n}\\
    &\leq   9c_3 n^{\e}\left(\frac{r}{\g_{\min}}\right)^{d}  \sum_{j \in J} \lb_j   +  c_{d,p,q} \left( \sum_{j \in J} \lb_j^{-1} \g_j^{-\frac{d}{p}}  P_X(A_j) \right)^p n^{-1}+ c_4\frac{\t}{s^{\z}n} 
\end{align*}
for some constant $c_4>0$ depending on $d,\e,p$ and $q$.
\end{proof}

\begin{theorem}[\textbf{Learning Rate on} $F$]\label{theo:rate_out}
 Let the assumptions of Theorem~\ref{theo:oracle_out}  be satisfied for $m_n$, $s \simeq s_n$ and with 
\begin{align}\label{def_rglb_stripe}
\begin{split}
r_n&\simeq n^{-\nu},\\
\g_{n,j} &\simeq r_n\\
\lb_{n,j} &\simeq n^{-\s},
\end{split}
\end{align}
for all $j\in  \lbrace 1,\ldots, m_n\rbrace$ and with $\max \lbrace \nu d,\a \z \rbrace<1$ and $\s \geq 1$.
Then, for all $\e>0$ there exists a constant $c_{d,\e,q}>0$ such that for $\bs\lb_n:=(\lb_{n,1},\ldots, \lb_{n,m_n})>0$, and $\bs\g_n:=(\g_{n,1},\ldots,\g_{n,m_n}) \in (0,r_n]^{m_n} $, and all $n \geq 1$ we have  with probability $P^n$ not less than $1-3e^{-\t}$ that 
\begin{align*}
 \RP{L_J}{\cl{f}_{D,\bs\lb_n,\bs\g_n}}-\RPB{L_J} &\leq c_{d,\e,q} \t \cdot \max \lbrace r_n^{-d},s_n^{-\z}\rbrace  n^{-1 +\e}.  
\end{align*}
\end{theorem}

\begin{proof}
We write $\lb_n:=n^{-\s}$ and $\g_n:=r_n$. Then, we obtain by Theorem~\ref{theo:oracle_out} and  \eqref{def_rglb_stripe} with probability  $P^n$ not less than $1-3e^{-\t}$ that
\begin{align*}
\RP{L_J}{\cl{f}_{D,\bs\lb,\bs\g}}-\RPB{L_J}&
\leq c_1 \left( n^{\e}\left(\frac{r_n}{\min_{j\in J}\g_j}\right)^{d}  \sum_{j \in J} \lb_{n,j}   +   \left( \sum_{j \in J} \lb_{n,j}^{-1} \g_{n,j}^{-\frac{d}{p}}  P_X(A_j) \right)^p n^{-1}+ \frac{\t}{s_n^{\z}n} \right)\\
&\leq c_2 \left( n^{\e}|J| \lb_n   +   \lb_n^{-p}r_n^{-d} \left( \sum_{j \in J}  P_X(A_j) \right)^p n^{-1}+ \frac{\t}{s_n^{\z}n} \right)\\
  &\leq c_2  \t \left(  r_n^{-d} n^{-\s+\e}   +  n^{\s p}r_n^{-d}n^{-1}+  s_n^{-\z}n^{-1} \right)\\
   &\leq c_3  \t \left(  r_n^{-d} n^{-1+\e}   +  n^{\e} r_n^{-d}n^{-1}+ s_n^{-\z}n^{-1} \right)\\
      &\leq c_3 \t\left( 2r_n^{-d} n^{-1+\e}+ s_n^{-\z}n^{-1} \right)\\
      &\leq c_4 \t \cdot  \max \lbrace r_n^{-d},s_n^{-\z}\rbrace  n^{-1 +\e}
\end{align*}
where $p$ is chosen sufficiently small such that $\e \geq p\s >0$ and where the constants $c_1,c_2>0$ depend only on $d,\e,p,q$ and the constants $c_3,c_4>0$ only on $d,\e,q$.
\end{proof}

\newpage

\appendix
\section*{Appendix A.}\label{sec:app}

In this appendix we state some results on margin conditions.


\begin{lemma}[Reverse H\"older yields lower control]\label{lem:revhoelder_lc}
Let $(X,d)$ be a metric space and $P$ be a probability measure on $X \times \{-1,1\}$ with fixed version $\n \colon X \to [0,1]$ of its posterior probability. Assume that $X_0=\partial_X X_1=\partial_X X_{-1}$. If $\n$ is reverse H\"older-continuous with exponent $\d \in (0,1]$, that is, if there exists a constant $c>0$ such that 
\begin{align*}
|\n(x)-\n(x')| \geq c \cdot d(x,x')^{\d}, \qquad x,x' \in X, 
\end{align*}
then, $\D_{\n}$ controls the noise from below by the exponent $\d$.
\end{lemma}

\begin{proof}
We fix w.l.o.g.\ an $x \in X_{-1}$. By the reverse H\"older continuity we obtain
 \begin{align*}
c \D_{\n}^{\r}(x) =c \inf_{\tilde{x} \in X_{1}} (d(x,\tilde{x}))^{\r} \leq \inf_{\tilde{x} \in X_{1}} |\n(x)-\n(\tilde{x} )| \leq \inf_{\tilde{x} \in X_{1}}\n(\tilde{x} ) -\n(x) .
\end{align*}
Since $\n(\tilde{x} ) > 1/2$ for all $\tilde{x} \in X_{1}$, 
we find by  continuity of $\n$ and $\partial X_1=X_0$ that $\inf_{\tilde{x} \in X_{1}}  \n(\tilde{x} )=1/2$. Thus,
\begin{align*}
\D_{\n}^{\r}(x) \leq (2c)^{-1} (1-2\n(x)).
\end{align*}
Obviously, the last inequality is immediately satisfied for $x\in X_0$ and for  $x \in X_1$ the calculation is similar.  Hence, $\D_{\n}$ controls the noise by the exponent $\r$ from below, that is,
\begin{align*}
\D_{\n}^{\r}(x) \leq c_{\mathrm{LC}} |2\n(x)-1|, \qquad x \in X,
\end{align*}
where $c_{\mathrm{LC}}:=(2c)^{-1}$.
\end{proof}

%
%

\begin{lemma}[LC and ME yield NE]\label{lem:lc+me=ne}
Let $(X,d)$ be a metric space and let  $P$ be a probability measure on $X \times \{-1,1\}$ that has ME $\a \in [0,\infty)$ for the version $\n$ of its posterior probability. Assume that the associated distance to the decision boundary controls the noise from below by the exponent $\z \in [0,\infty)$. Then, $P$ has NE $q=\frac{\a}{\z}$.
\end{lemma}

\begin{proof}
Since $P$ has ME $\a \in [0,\infty)$, we find for some $t>0$ that
\begin{align*}
\frac{\D_{\n}^{\z}(x)}{c_{\mathrm{LC}}} \leq |2\n(x) -1|<t, \qquad x \in X, 
\end{align*} 
and we follow that $\D_{\n}(x) \leq  (c_{\mathrm{LC}}t)^{\frac{1}{\z}}$. Consequently, the definition of the noise exponent  yields
\begin{align*}
P_X\left(\lbrace x \in X:|2\n(x)-1|<t \rbrace \right) &\leq P_X \left(\lbrace x \in X:\D_{\n}(x) \leq (c_{\text{LC}}t)^{\frac{1}{\z}}\rbrace \right)\\
&\leq c_{\text{ME}}^{\a} (c_{\text{LC}}t)^{\frac{\a}{\z}}.
\end{align*}
\end{proof}

\begin{remark}\label{rem:app}
\begin{itemize}
\item[i)] One can show by using similar arguments as in \citep[Lemma~2.1]{BlSt18} together with \citep[Lemma~A.10.4(i)]{Steinwart15a} that there exists a $\d^{\ast}>0$ such that the lower bound 
\begin{align*}
\lb^d(\lbrace x \in X | \D_{\n}(x) \leq \delta \rbrace) \geq c_d \cdot \delta \qquad \text{for all}\ \d \in (0,\d^{\ast}]
\end{align*}
and some $c_d>0$ is satisfied.
\item[ii)] Assume that $\n$  is H\"older-smooth with exponent $\r$, that $P$ has NE $q$ and that $P_X$ has a density w.r.t.\ the Lebesgue measure that is bounded away from zero. Then, part i) together with \citep[Lemma~A.2]{BlSt18} yields
\begin{align*}
ct^{\frac{1}{\r}} \leq P_X(\lbrace \D_{\n}(x) \leq t^{\frac{1}{\r}} \rbrace) \leq P_X\left(\lbrace x \in X:|2\n(x)-1|<  t \rbrace \right)\leq c_{\mathrm{NE}}t^{q}
\end{align*}
for some constant $c>0$. Thus, $\r q > 1$ can never be satisfied. 
\end{itemize}
\end{remark}

\section*{Appendix B.}

In this appendix we state some technical lemmata.

\begin{lemma}[\textbf{Number of cells}]\label{lemm:numbercell}
Let assumptions \textbf{(A)} and \textbf{(G)} be satisfied.
Let $s\geq r$ and $s+r \leq \d^{\ast}$, where $\d^{\ast}>0$ is the constant from \eqref{ineq_lebesgue.blst18}, and define
\begin{align*}
J &:= \set{j\in J}{\forall \,x \in A_j :  \D_{\n}(x) \leq s}.
\end{align*}
 Then, there exists a constant $c_d>0$ such that
\begin{align*}
|J| \leq c_d \cdot sr^{-d}.
\end{align*}
\end{lemma}

\begin{proof}
We define $T:=\bigcup_{j \in J} A_j$ and $\tilde{T}:=\bigcup_{j \in J} B_r(z_j)$. Obviously, $T \subset \tilde{T}$ since $A_j \subset B_r(z_j)$ for all $j \in J$. Furthermore, we have for all $x \in \tilde{T}$ that $\D_{\n}(x) \leq \tilde{s}$, where $\tilde{s}:=s+r$. Then, we obtain with \citep[Lemma~2.1]{BlSt18}  that
\begin{align}\label{ineq:lebesquetildeS_1}
\lb^d(\tilde{T})\leq \lb^d\left(\left\lbrace \D(x) \leq \tilde{s} \right\rbrace \right)\leq 4\mathcal{H}^{d-1}(X_0) \cdot \tilde{s}.
\end{align}
Moreover, 
\begin{align}\label{ineq:lebesquetildeS_2}
\lb^d(\tilde{T})\!=\!\lb^d\left( \bigcup_{j\in J} B_r(z_j) \right)\!\geq \! \lb^d\left( \bigcup_{j\in J} B_{\frac{r}{4}}(z_j) \right)\!=\! |J| \lb^d\left( B_{\frac{r}{4}}(z) \right)\!=\!|J| \left(\frac{r}{4}\right)^d \lb^d\left(B \right),
\end{align}
since $B_{\frac{r}{4}}(z_i) \cap B_{\frac{r}{4}}(z_j) = \emptyset$ for $i \neq j$. To see the latter, assume that we have an $x \in B_{\frac{r}{4}}(z_i) \cap B_{\frac{r}{4}}(z_j)$. But then,  $\snorm{z_i-z_j}_2 \leq \snorm{x-z_j}_2 + \snorm{x-z_i}_2 \leq \frac{r}{4} + \frac{r}{4} \leq \frac{r}{2}$, which is not true, since we assumed $\snorm{z_i-z_j}_2>\frac{r}{2}$ for all $i \neq j$. Hence, the balls with radius $\frac{r}{4}$ are disjoint.
Finally, by \eqref{ineq:lebesquetildeS_1} together with \eqref{ineq:lebesquetildeS_2} and $s \geq r$ we find 
\begin{align*}
|J|\! \leq \! \frac{4^d\lb^d(\tilde{T})}{r^d \lb^d\left(B \right)} \! \leq \! \frac{2^{2d+2}\mathcal{H}^{d-1}(\lbrace x \in X|\n=1/2 \rbrace) \cdot \tilde{s}}{r^d \lb^d\left(B \right)}\! \leq \! \frac{2^{2d+3}\mathcal{H}^{d-1}(\lbrace x \in X|\n=1/2 \rbrace) \cdot s}{r^d \lb^d\left(B \right)}. 
\end{align*}
\end{proof}

\begin{lemma}\label{lemma:tech_int}
Let $X \subset \R^d$ and $\g,\r>0$.
Then, we have
\begin{align*}
\left(\frac{2}{\pi \g^2}\right)^{d/2}  \int_{B_{\r}(x)} e^{-2\g^{-2} \snorm{x-y}^2_2}\dx{y}=\frac{1}{\G(d/2)}  \int_0^{2{\r}^2 \g^{-2}} e^{-t}t^{d/2-1}\dx{t}.
\end{align*}
\end{lemma}

\begin{proof}
For $\r>0$ we find that 
\begin{align*}
&\left(\frac{2}{\pi \g^2}\right)^{d/2}  \int_{B_{\r}(x)} e^{-2\g^{-2} \snorm{x-y}^2_2}\dx{y}\\
			&=\left(\frac{2}{\pi \g^2}\right)^{d/2}  \int_{B_{\r}(0)} e^{-2\g^{-2} \snorm{y}^2_2}\dx{y}\\
			&=\left(\frac{2}{\pi \g^2}\right)^{d/2} \frac{\pi^{d/2}}{\G(d/2+1)}  \int_0^{\r} e^{-2\g^{-2}t^2}d\cdot t^{d-1}\dx{t}\\
			&=\left(\frac{2}{\pi \g^2}\right)^{d/2} \frac{2\pi^{d/2}}{d\G(d/2)}  \int_0^{\sqrt{2}\r\g^{-1}} e^{-t^2}d\cdot t^{d-1}\cdot \frac{1}{\sqrt{2}\g^{-1} }\left(\frac{\g}{\sqrt{2}} \right)^{d-1}\dx{t}\\
			&= \frac{2}{\G(d/2)}  \int_0^{2{\r}^2 \g^{-2}} e^{-t}t^{d/2-1}\cdot \frac{1}{2}\dx{t}\\
			&=\frac{1}{\G(d/2)}  \int_0^{2{\r}^2 \g^{-2}} e^{-t}t^{d/2-1}\dx{t}.
\end{align*}
\end{proof}

\begin{lemma}\label{Lemma_adapt}
Let $(A_j)_{j = 1,\ldots ,m}$ be a partition of $B_{\ell^d_2}$. Let $d\geq 1, p \in (0,\frac{1}{2})$ and let $r_n \in (0,1]$. For $\r_n \leq n^{-2}$ and  $\d_n \leq n^{-1}$ fix a finite $\r_n$-net $\Lambda_n \subset (0,n^{-1}]$ and a finite $\d_n r_n$-net $\Gamma_n \subset (0,r_n]$. Let $J \subset \lbrace 1,\ldots,m \rbrace$ be an index set and for all $j \in J$  let $\g_j \in (0,r_n]$, $\lb_j>0$. Define $\g_{\max}:=\max_{j\in J} \g_j$ resp.\ $\g_{\min}:=\min_{j\in J} \g_j$. 
\begin{itemize}
\item[i)]  Let $\b\in (0,1],q \in [0,\infty)$ and let $|J| \leq c_d r_n^{-d+1}$ for some constant $c_d>0$.  Then, for all $\e_1>0$ there exists a constant $\tilde{c}_1>0$ such that 
\begin{align*}
&\inf_{(\bs\lb,\bs\g) \in (\Lambda_n \times\Gamma_n)^{m_n}}  \left( \sum_{j \in J} \frac{\lb_j r_n^d}{\g_j^d}  +  \g_{\max}^{\b}
  +\left(\frac{r_n}{n}\right)^{\frac{q+1}{q+2-p}} \left( \sum_{j \in J} \lb_j^{-1} \g_j^{-\frac{d}{p}}  P_X(A_j) \right)^{\frac{p(q+1)}{q+2-p}} \right)\\
 &\leq  \tilde{c}_1 \cdot  n^{-\b \k (\nu+1) + \e_1}.
\end{align*}
\item[ii)] Let $\b\in (0,1], q \in [0,\infty)$ and  let $|J| \leq c_d r_n^{-d+1}$ for some constant $c_d>0$.
Then, for all $\tilde{\e},\e_2>0$ there exists a constant $\tilde{c}_2>0$ such that  
\begin{align*}
&\inf_{(\bs\lb,\bs\g) \in (\Lambda_n \times\Gamma_n)^{m_n}}   \left( \left( \frac{r_n}{\g_{\min}}\right)^{d}  \sum_{j \in J} \lb_j  n^{\tilde{\e}}+ \left(\frac{r_n}{n}\right)^{\frac{q+1}{q+2-p}} \left( \sum_{j \in J} \lb_j^{-1} \g_j^{-\frac{d}{p}}  P_X(A_j) \right)^{\frac{p(q+1)}{q+2-p}}  \right)\\
&\leq  \tilde{c}_2 \cdot n^{\e_2} \left(r_n^{d-1} n\right)^{-\frac{q+1}{q+2}}.
\end{align*}
\item[iii)] Let $|J| \leq c_d r_n^{-d}$.   Then, for all $\tilde{\e},\e_3>0$ there exists a constant $\tilde{c}_3>0$ such that 
\begin{align*}
\inf_{(\bs\lb,\bs\g) \in (\Lambda_n \times\Gamma_n)^{m_n}} \left( \left(\frac{r_n}{\g_{\min}}\right)^{d}  \sum_{j \in J} \lb_j  n^{\tilde{\e}} +  \left( \sum_{j \in J} \lb_j^{-1} \g_j^{-\frac{d}{p}}  P_X(A_j) \right)^p n^{-1} \right)\leq  \tilde{c}_3 \cdot r_n^{-d} n^{-1+\e_3}.
\end{align*}
\end{itemize}
\end{lemma}

\begin{proof}
We follow the lines of the proof of \citep[Lemma~14]{EbSt16}. 
Let us assume that $\Lambda_n:=\lbrace \lb^{(1)}, \ldots \lb^{(u)} \rbrace$ and $\Gamma_n:=\lbrace \g^{(1)}, \ldots \g^{(v)} \rbrace$ such that $\lb^{(i-1)}<\lb^{(i)}$ and $\g^{(l-1)}<\g^{(l)}$ for all $i=2,\ldots,u$ and $l=2,\ldots,v$. Furthermore, let $\g^{(0)} =\lb^{(0)}:=0$ and $\lb^{(u)}:=n^{-1}, \g^{(v)}:=r_n$.  Then, fix a pair $(\lb^{\ast},\g^{\ast}) \in [0,n^{-1}] \times [0,r_n]$. Following the lines of the proof of \citep[Lemma~6.30]{StCh08} there exist indices $i \in \lbrace 1,\ldots,u \rbrace$ and $l \in \lbrace 1,\ldots,v \rbrace$ such that
\begin{align}\label{adapt1}
\begin{split}
\lb^{\ast} &\leq \lb^{(i)} \leq \lb^{\ast}+2 \r_n,\\
\g^{\ast} &\leq \g^{(l)} \leq \g^{\ast}+2\d_n r_n.
\end{split}
\end{align}

\begin{itemize}
\item[i)] 
With \eqref{adapt1} we find
\begin{align*}
& \inf_{(\bs\lb,\bs\g) \in (\Lambda_n \times\Gamma_n)^{m_n}}  \left( \sum_{j \in J} \frac{\lb_j r_n^d}{\g_j^d}  +  \g_{\max}^{\b}
  +\left(\frac{r_n}{n}\right)^{\frac{q+1}{q+2-p}} \left( \sum_{j \in J} \lb_j^{-1} \g_j^{-\frac{d}{p}}  P_X(A_j) \right)^{\frac{p(q+1)}{q+2-p}} \right)\\
&\leq \sum_{j \in J} \frac{\lb^{(i)} r_n^d}{\left(\g^{(l)}\right)^d}  + \left(\g^{(l)}\right)^{\b} + \left(\frac{r}{n}\right)^{\frac{q+1}{q+2-p}} \left( \sum_{j \in J}\left(\lb^{(i)}\right)^{-1} \left(\g^{(l)}\right)^{-\frac{d}{p}}  P_X(A_j) \right)^{\frac{p(q+1)}{q+2-p}} \\
&\leq |J| \frac{\lb^{(i)} r_n^d}{\left(\g^{(l)}\right)^d}  + \left(\g^{(l)}\right)^{\b} + \left(\frac{r_n \left(\lb^{(i)}\right)^{-p} \left(\g^{(l)}\right)^{-d}}{n}\right)^{\frac{q+1}{q+2-p}}  \left( \sum_{j \in J}  P_X(A_j) \right)^{\frac{p(q+1)}{q+2-p}}\\
&\leq \frac{(\lb^{\ast}+2\r_n) r_n}{(\g^{\ast})^d}  + (\g^{\ast}+2\d_n r_n)^{\b} + \left(\frac{r_n (\lb^{\ast})^{-p} }{(\g^{\ast})^d n}\right)^{\frac{q+1}{q+2-p}} \\
&\leq c_1 \left( \lb^{\ast} r_n (\g^{\ast})^{-d}+(\g^{\ast})^{\b}+\left( \frac{r_n(\lb^{\ast})^{-p}}{(\g^{\ast})^d n}\right)^{\frac{q+1}{q+2-p}}+\r_n r_n(\g^{\ast})^{-d}  +\left(\d_n r_n\right)^{\b} \right)
\end{align*}
for some $c_1>0$. We define $\lb^{\ast}:=n^{-\s}$ for some $s \in [1,2]$ and $\g^{\ast}:=r_n^{\k}n^{-\k}$.  Obviously, $\lb^{\ast} \in [0,n^{-1}]$. Moreover, we have $\g^{\ast} \in [0,r_n]$ since $\nu \leq \frac{\k}{1-\k}$.
Then, we obtain with $\r_n \leq n^{-2}$ and $\d_n \leq n^{-1}$, and together with $1 \geq \k(\b+d)(\nu+1)-\nu>0$ and
$\tfrac{(1-d\k)(q+1)}{q+2-p}> \tfrac{(1-d\k)(q+1)}{q+2}=\left(\tfrac{\b(q+2)}{\b(q+2)+d(q+1)} \right)\tfrac{q+1}{q+2} =\b \k$  that
\begin{align*}
&c_1 \left( \lb^{\ast} r_n (\g^{\ast})^{-d}+(\g^{\ast})^{\b}+\left( \frac{r_n(\lb^{\ast})^{-p}}{(\g^{\ast})^d n}\right)^{\frac{q+1}{q+2-p}}+\r_n r_n(\g^{\ast})^{-d}  +\left(\d_n r_n\right)^{\b} \right)\\
&\leq c_1 \left( n^{-\s}r_n  r_n^{-d \k} n^{d \k}+r_n^{\b \k}n^{-\b \k}+\left( \frac{r_n^{1-d \k} (\lb^{\ast})^{-p}}{n^{1- d \k} } \right)^{\frac{q+1}{q+2-p}}+n^{-2} r_n(\g^{\ast})^{-d}  +\left(r_n n^{-1} \right)^{\b} \right)\\
&\leq c_2 \left(r_n^{-1+(\b+d) \k} n^{-(\b+d) \k}r_n  r_n^{-d \k} n^{d \k} + r_n^{\b \k}n^{-\b \k}+ \left(r_n n^{-1} \right)^{\frac{(1-d \k)(q+1)}{q+2-p}} n^{\frac{p \s(q+1)}{q+2-p}}+\left(r_n n^{-1} \right)^{\b} \right)\\
&\leq c_2\left( r_n^{\b \k}n^{-\b \k}+ r_n^{\b \k}n^{-\b \k} n^{\e_1}+\left(r_n n^{-1} \right)^{\b} \right)\\
&\leq  c_3 \cdot n^{-\b \k (\nu+1) + \e_1}
\end{align*}
holds for some constants $c_2,c_3>0$ and where $p$ is chosen sufficiently small such that $\e_1 \geq \frac{p\s(q+1)}{q+2-p} >0$. 
\item[ii)]
With \eqref{adapt1} we find
\begin{align*}
&\inf_{(\bs\lb,\bs\g) \in (\Lambda_n \times\Gamma_n)^{m_n}}   \left( \left( \frac{r_n}{\g_{\min}}\right)^{d}  \sum_{j \in J} \lb_j  n^{\tilde{\e}}+ \left(\frac{r_n}{n}\right)^{\frac{q+1}{q+2-p}} \left( \sum_{j \in J} \lb_j^{-1} \g_j^{-\frac{d}{p}}  P_X(A_j) \right)^{\frac{p(q+1)}{q+2-p}}  \right)\\
&\leq \left( \frac{r_n}{\g^{(l)}}\right)^{d}  \sum_{j \in J} \lb^{(i)}  n^{\tilde{\e}}+ \left(\frac{r_n}{n}\right)^{\frac{q+1}{q+2-p}} \left( \sum_{j \in J} \left(\lb^{(i)}\right)^{-1} \left(\g^{(l)}\right)^{-\frac{d}{p}}  P_X(A_j) \right)^{\frac{p(q+1)}{q+2-p}} \\
&\leq \left( \frac{r_n}{\g^{(l)}}\right)^{d}  |J| \lb^{(i)}  n^{\tilde{\e}}+ \left(\frac{r_n}{\left(\g^{(l)}\right)^d n}\right)^{\frac{q+1}{q+2-p}}\left(\lb^{(i)}\right)^{-\frac{p(q+1)}{q+2-p}} \left( \sum_{j \in J}   P_X(A_j) \right)^{\frac{p(q+1)}{q+2-p}} \\
&\leq c_4  \frac{r_n \lb^{(i)}  n^{\tilde{\e}}}{\left( \g^{(l)}\right)^{d}}+ \left(\frac{r_n}{\left(\g^{(l)}\right)^d n}\right)^{\frac{q+1}{q+2-p}}\left(\lb^{(i)}\right)^{-\frac{p(q+1)}{q+2-p}} \\
&\leq c_4  \frac{r_n (\lb^{\ast}+2\r_n)  n^{\tilde{\e}}}{\left( \g^{\ast}\right)^{d}}+ \left(\frac{r_n}{\left(\g^{\ast}\right)^d n}\right)^{\frac{q+1}{q+2-p}}\left(\lb^{\ast}\right)^{-\frac{p(q+1)}{q+2-p}}\\
&\leq c_4  \frac{r_n \lb^{\ast} n^{\tilde{\e}}}{\left( \g^{\ast}\right)^{d}}+ \left(\frac{r_n}{\left(\g^{\ast}\right)^d n}\right)^{\frac{q+1}{q+2-p}}\left(\lb^{\ast}\right)^{-\frac{p(q+1)}{q+2-p}}+2c_4  \frac{\r_n r_n   n^{\tilde{\e}}}{\left( \g^{\ast}\right)^{d}}
\end{align*}
for some constant $c_4>0$ depending on $d$.
We define $\g^{\ast}:=r_n$ and $\lb^{\ast}:= n^{-\s}$ for some $\s \in [1,2]$. Then, we obtain with $\r_n \leq n^{-2}$ that
\begin{align*}
&c_4  \frac{r_n \lb^{\ast} n^{\tilde{\e}}}{\left( \g^{\ast}\right)^{d}}+ \left(\frac{r_n}{\left(\g^{\ast}\right)^d n}\right)^{\frac{q+1}{q+2-p}}\left(\lb^{\ast}\right)^{-\frac{p(q+1)}{q+2-p}}+2c_4  \frac{r_n \r_n  n^{\tilde{\e}}}{\left( \g^{\ast}\right)^{d}}\\
&= c_4  \frac{n^{-\s} n^{\tilde{\e}}}{r_n^{d-1}}+ \left(\frac{1}{r_n^{d-1} n}\right)^{\frac{q+1}{q+2-p}} n^{\frac{p\s (q+1)}{q+2-p}}+2c_4  \frac{\r_n  n^{\tilde{\e}}}{r_n^{d-1}}\\
&\leq c_4  \frac{n^{-1} n^{\tilde{\e}}}{r_n^{d-1}}+ \left(r_n^{d-1} n\right)^{-\frac{q+1}{q+2}} n^{\hat{\e}}+2c_4  \frac{\r_n n^{\tilde{\e}}}{r_n^{d-1}}\\
&\leq c_5 n^{\e_2} \left(  \left(r_n^{d-1} n\right)^{-1}+ \left(r_n^{d-1} n\right)^{-\frac{q+1}{q+2}} +  n^{-2} \left(r_n^{d-1}\right)^{-1} \right)\\
&\leq c_6 n^{\e_2} \left(r_n^{d-1} n\right)^{-\frac{q+1}{q+2}} ,
\end{align*}
where $c_5,c_6>0$ are constants depending on $d$ and where $p$ is chosen sufficiently small such that $\hat{\e}\geq \frac{p\s(q+1)}{q+2-p}$ and $\e_2:=\max \lbrace \tilde{\e},\hat{\e}\rbrace$.
\item[iii)] We find with \eqref{adapt1}  
that 
\begin{align*}
&\inf_{(\bs\lb,\bs\g) \in (\Lambda_n \times\Gamma_n)^{m_n}} \left( n^{\tilde{\e}}\left(\frac{r_n}{\g_{\min}}\right)^{d}  \sum_{j \in J} \lb_j   +  \left( \sum_{j \in J} \lb_j^{-1} \g_j^{-\frac{d}{p}}  P_X(A_j) \right)^p n^{-1} \right) \\
&\leq  n^{\tilde{\e}}\left(\frac{r_n}{\g_{\min}}\right)^{d}  \sum_{j \in J}\lb^{(i)}   +  \left( \sum_{j \in J} \left(\lb^{(i)}\right)^{-1} \left(\g^{(l)}\right)^{-\frac{d}{p}}  P_X(A_j) \right)^p n^{-1}  \\
&\leq   n^{\tilde{\e}}\left(\frac{r_n}{\g_{\min}}\right)^{d}  |J |\lb^{(i)}   + \left(\lb^{(i)}\right)^{-p} \left(\g^{(l)}\right)^{-d} \left( \sum_{j \in J}   P_X(A_j) \right)^p n^{-1}  \\
&\leq  c_7 n^{\tilde{\e}} \left(\g^{(l)}\right)^{-d} \lb^{(i)}   + \left(\lb^{(i)}\right)^{-p} \left(\g^{(l)}\right)^{-d} n^{-1} \\
&\leq c_7  n^{\tilde{\e}}  \left(\g^{\ast}\right)^{-d} (\lb^{\ast}+2\r_n)   + \left(\lb^{\ast}\right)^{-p} \left(\g^{\ast}\right)^{-d} n^{-1}\\
&=c_7  n^{\tilde{\e}}  \left(\g^{\ast}\right)^{-d} \lb^{\ast}   + \left(\lb^{\ast}\right)^{-p} \left(\g^{\ast}\right)^{-d} n^{-1}+2\r_n c_7  n^{\tilde{\e}} \left(\g^{\ast}\right)^{-d} 
\end{align*}
holds for some constant $c_7>0$ depending on $d$. We define $\g^{\ast}:=r_n$ and $\lb^{\ast}:= n^{-\s}$ for some $\s \in [1,2]$. Then,  we obtain with $\r_n \leq n^{-2}$
\begin{align*}
&c_7  n^{\tilde{\e}}  \left(\g^{\ast}\right)^{-d} \lb^{\ast}   + \left(\lb^{\ast}\right)^{-p} \left(\g^{\ast}\right)^{-d} n^{-1}+2\r_n c_7  n^{\tilde{\e}} \left(\g^{\ast}\right)^{-d} \\
&\leq c_7  n^{\tilde{\e}} r_n^{-d} n^{-\s}   + n^{p \s} r_n^{-d} n^{-1}+2 c_7  n^{\tilde{\e}} r_n^{-d}n^{-2} \\
&\leq c_7  n^{\tilde{\e}} r_n^{-d} n^{-1}   + n^{\hat{\e}} r_n^{-d} n^{-1}+2 c_7  n^{\tilde{\e}} r_n^{-d}n^{-2} \\
&\leq c_8 \cdot r_n^{-d} n^{-1+\e_3}
\end{align*}
for some constant $c_8>0$ depending on $d$ and where $p$ is chosen  sufficiently small such that $\hat{\e} \geq p \s>0$. Here, $\e_3:= \max \lbrace \tilde{\e},\hat{\e} \rbrace$.

\end{itemize}
\end{proof}



%
%
%

\vskip 0.2in
\bibliography{ing_literature}

\end{document}